\documentclass[reqno]{amsart}
\usepackage{epsfig}
\usepackage{color}
\usepackage{amssymb,amsmath,amsthm,amstext,amsfonts}

\usepackage{amssymb,amsmath,amsthm,amstext,amsfonts}
\usepackage{amsmath,amstext,amsthm,amsfonts}
\usepackage{color}
\usepackage[mathscr]{eucal}
\usepackage{dsfont}

\usepackage{psfrag}
\usepackage{url}
\usepackage{epstopdf}
\usepackage{epic,eepic}
\usepackage{upgreek}
\usepackage{amssymb}
\usepackage{mathrsfs}
\usepackage{verbatim}
\usepackage{mathrsfs}
\usepackage{amstext}
\usepackage{amsthm}
\usepackage{amssymb}
\usepackage{graphicx}

\usepackage[colorlinks=true, linkcolor=blue, urlcolor=red, citecolor=blue]{hyperref}

\makeatletter \@addtoreset{equation}{section} \makeatother

\renewcommand\thetable{\thesection.\@arabic\c@table}

\theoremstyle{plain}
\newtheorem{maintheorem}{Theorem}

\newtheorem{maincorollary}{Corollary}

\newtheorem{theorem}{Theorem}[section]
\newtheorem{proposition}{Proposition}[section]
\newtheorem{lemma}{Lemma}[section]
\newtheorem{corollary}{Corollary}[section]
\newtheorem{definition}{Definition}[section]
\newtheorem{remark}{Remark}[section]

\newtheorem{claim}{Claim}[section]

\newcommand{\al} {\alpha}

\newcommand{\vep}{\varepsilon}

\newcommand{\supp}{\operatorname{supp}}

\newcommand{\Leb}{Leb}

\newcommand{\cC}{\mathcal{C}}
\newcommand{\cK}{\mathcal{K}}

\newcommand{\cP}{\mathcal{P}}

\newcommand{\cU}{\mathcal{U}}

\newcounter{main}

\title[SRB measures for open sets of partially hyperbolic local diffeomorphisms]
{SRB measures for partially hyperbolic attractors of local diffeomorphisms}

\author[Anderson Cruz]{Anderson Cruz}
\address{Anderson Cruz, Centro de Ci\^encias Exatas e Tecnol\'ogicas, Universidade Federal do Rec\^oncavo da Bahia\\ Av. Rui Barbosa, s/n, 44380-000 Cruz das Almas, BA, Brazil}
\email{anderson.cruz@ufrb.edu.br}

\author[Paulo Varandas]{Paulo Varandas}
\address{Paulo Varandas, Departamento de Matem\'atica e Estat\'istica, Universidade Federal da Bahia\\
Av. Ademar de Barros s/n, 40170-110 Salvador, Brazil}
\email{paulo.varandas@ufba.br}

\begin{document}

\begin{abstract}
In the present paper we contribute to the thermodynamic formalism of 
partially hyperbolic attractors for local diffeomorphisms admitting an invariant stable bundle 
and a positively invariant cone field with non-uniform cone expansion at a positive Lebesgue measure
set of points. These include the case of attractors for Axiom A endomorphisms and partially hyperbolic
endomorphisms derived from Anosov.  
We prove these attractors have finitely many SRB measures, that these are hyperbolic, and that 
the SRB measure is unique provided the dynamics is transitive. 
Moreover, we show that the SRB measures are statistically stable (in the weak$^*$ topology) 
and that their entropy varies continuously with respect to the local diffeomorphism.
\end{abstract}

\keywords{SRB measures, partially hyperbolicity, cone-hyperbolicity, statistical stability.}
 \footnotetext{2010 {\it Mathematics Subject classification}:
Primary 
37C40, 
37D25; 
Secondary 
37D30, 
37D35. 

} 
\date{\today}
\maketitle

\section{Introduction}

The main goal of the use of ergodic theory in dynamical systems is to describe the  
statistical properties of the dynamics using invariant measures. 
In particular, the thermodynamic formalism aims the construction and description of the 
statistical properties of invariant measures that are physically relevant, meaning 
equilibrium states with respect to some potential.
Sinai, Ruelle and Bowen (\cite{Sin72,Rue76,Bow75}) 
constructed a special class of invariant measures for hyperbolic attractors of diffeomorphisms $f$ acting on a compact Riemannian manifold $M$:
these have absolutely continuous disintegration with respect to the Lebesgue
measure along unstable manifolds and are Gibbs equilibrium states for the geometric potential
$\phi^{u}(x)=-\log\left|\det Df(x)\mid_{E_{x}^{u}}\right|$.
Such measures are known as SRB (Sinai-Ruelle-Bowen) measures.
Furthermore, the latter measures are physical, meaning their ergodic basin of attraction 
$
B(\mu)=\big\{x\in M \colon \frac1n \sum_{j=0}^{n-1} \delta_{f^j(x)} \to \mu \, \text{as }\, n\to \infty\big\}
$
has positive Lebesgue measure in the ambient space. 
Modern methods for the construction of SRB measures for hyperbolic attractors 
and the study of their finer properties
can be found in \cite{BKL02,GL06} and references therein.

Since SRB measures were introduced in the realm of dynamical systems, their construction and 
the study of their statistical properties rapidly became a topic of interest of the mathematics and physics communities. 
Unfortunately, apart from the uniformly hyperbolic setting, where the existence of finite Markov partitions allows to semiconjugate
the dynamics to subshifts of finite type (cf. \cite{Sin72,Rue76,Bow75}) there is no systematic approach
for the construction of SRB measures.
Note that SRB measures often coincide with equilibrium states for a geometric potential (cf. \cite{LY85}). 
As for the construction of SRB measures beyond the setting of uniform hyperbolicity one should mention the 
construction of u-Gibbs measures for partially hyperbolic attractors that admit an unstable bundle by Pesin and Sinai  \cite{PeS},
the existence and uniqueness of the SRB measure for the class of robustly transitive diffeomorphisms derived from Anosov 
constructed by Ma\~n\'e by Carvalho~\cite{Car93}, and the construction of SRB measures for $C^{1+\alpha}$ partially hyperbolic 
diffeomorphisms displaying a non-uniform hyperbolicity condition along the central direction by Alves, Bonatti and Viana ~\cite{ABV00,BV00}. More recently, Mi, Cao and Yang \cite{MCY} constructed SRB measures for attractors of $C^2$ diffeomorphisms that admit H\"older continuous invariant (non-dominated) splittings with some non-uniform expansion.
After Young's~\cite{You98} axiomatic construction for studying decay of correlations, 
the decay of correlations for SRB measures can be studied through the existence of Markov towers and was shown to depend 
on the Lebesgue measure of the tails associated to non-uniform hyperbolicity. Some contributions on the study of 
the decay of correlations for the  SRB measures of related classes of uniformly hyperbolic and partially hyperbolic 
diffeomorphisms include ~\cite{Liv95,Dol00,Cas04}. 
SRB measures can also be constructed as zero noise limits as considered in \cite{CY05}.
Moreover, the statistical stability of SRB measures in dynamical systems and the continuous dependence of the entropy of the SRB measures 
has been considered in \cite{CY05, DL08,ACF10b, VV10,CV13, CV16}, 
among others.

In parallel to the developments of partially hyperbolic diffeomorphisms, there have been important contributions to the
study of (non-singular) endomorphisms, i.e., local diffeomorphisms. In the case that
the endomorphisms admit no stable bundle, the geometrical constructions of Markov structures for endomorphisms with non-uniform
expansion due to \cite{ALP05} provide weak conditions for the existence of SRB measures and 
estimates on their decay of correlations.
The situation is substantially more complicated in the case of endomorphisms displaying contracting behavior, as e.g. 
strongly dissipative endomorphisms that arise in the context of bifurcation of homoclinic tangencies associated 
to periodic points of diffeomorphisms (see \cite{MoraV} and references therein). 
In the mid seventies, Przytycki~\cite{Prz76} extended the notion of uniform hyperbolicity to the context of endomorphisms and studied Anosov endomorphisms. Here, due to the non-invertibility of the dynamics, the existence of an invariant unstable subbundle for
uniformly hyperbolic basic pieces needs to be replaced by the existence of positively invariant cone-fields on which vectors are uniformly expanded 
by all positive iterates (we refer the reader to Subsection~\ref{sec:UH} for more details). 
In \cite{QZ95,QZ02} the authors constructed SRB measures
for Axiom A attractors of endomorphisms,
obtaining these as equilibrium states for the geometric potential 
(defined by means of the natural extension). A characterization of SRB measures for uniformly hyperbolic
endomorphisms can also be given in terms of dimensional characteristics of the stable
manifold (\cite{UW04}). 
The thermodynamic formalism of hyperbolic basic pieces for endomorphisms
had the contribution of Mihailescu and Urbanski~\cite{Mih10a,MU} 
that introduced and constructed inverse SRB measures for hyperbolic attractors of endomorphisms.
Among the difficulties that arise when dealing with non-invertible hyperbolic dynamics one should refer 
the absence of contraction along inverse branches and the fact that unstable manifolds depend on 
entire pre-orbits, 
thus unstable manifolds may have a complicated geometrical structure 
(cf. \cite{QXZ09,UW04}). 

The ergodic theory of partially hyperbolic endomorphisms is more incomplete. 
In the context of surface endomorphisms, a major contribution is due to Tsujii \cite{Tsu05}, which proved
that for $r\ge 19$, $C^r$-generic partially hyperbolic endomorphisms 
with an unstable cone field admit finitely many SRB measures whose ergodic basins of attraction cover Lebesgue 
almost every point in the manifold (the regularity can be dropped to $r\ge 2$ on the space of non-singular 
endomorphisms). While it remains unknown if these SRB measures are (generically) hyperbolic, an extension to higher 
dimension should present significant difficulties as discussed by the author (cf.~\cite[page 43]{Tsu05}).
The strategy developed by Tsujii for surface endomorphisms was more recently adapted to deal with open sets of three-dimensional 
partially hyperbolic and dynamically coherent diffeomorphisms (see \cite{Bor15} for the precise statement).

Our purpose here is to contribute to the ergodic theory of partially hyperbolic attractors for endomorphisms 
and to present open conditions under which SRB measures exist (actually are hyperbolic and physical measures) 
and are statistically stable.   The open class of partially hyperbolic attractors considered here is complementary to 
the one studied in \cite{Tsu05} (see Section~\ref{sec:future} for a discussion on the different assumptions of partial hyperbolicity) . 
We consider attractors for 
$C^{1+\alpha}$ ($\alpha>0$) non-singular endomorphisms on compact Riemannian manifolds that admit uniform
contraction along a well defined invariant stable subbundle and preserve a positively invariant cone-field.
Under a non-uniform cone-expansion property on a positive Lebesgue measure set of points, we prove the endomorphism 
admits an hyperbolic SRB measure. Moreover, 
if Lebesgue almost every point satisfies the latter condition then there are finitely many SRB measures 
(it is unique if the attractor is transitive) whose basins cover Lebesgue almost every point in the topological 
basin of the attractor.
In the trivial case that the subbundle $E^{s}$ is trivial we recover results obtained in \cite{ABV00} for non-uniform 
expanding maps.
In the case the subbundle $E^{s}$ is non-trivial and there exists uniform expansion along the cone direction we recover
the setting of Axiom A attractors for endomorphisms considered in \cite{QXZ09}.
One should remark that, while SRB measures are measures that admit absolutely continuous disintegrations 
along Pesin's unstable foliation, the simple existence of unstable manifolds is not a direct consequence of
our assumptions. For that reason, some of the main difficulties arise from the fact that
the construction of unstable manifolds (obtained by Pesin's unstable manifold theorem, provided a hyperbolic
invariant measure) is parallel to the construction of the SRB measure. 
For that reason the construction of the SRB measures, as accumulation points of Lebesgue measure over
disks tangent to the cone field, must be followed closely by a precise track of the accumulation disks in terms of 
the space of pre-orbits. This is crucial to conclude later the SRB property.

Here we use the geometric ingredients developed in the construction of SRB measures to prove that 
 SRB measures depend continuously (with respect to the weak$^*$ topology) with perturbations of the endomorphism (see Section~\ref{sec:Main} for the precise statement). This proof is inspired by \cite{Vas07}.
Moreover, the continuous dependence of SRB measures on the dynamics together with Pesin's entropy formula and Oseledets theorem allow
us to to show that the entropy of the SRB measure is given by an expression involving the stable Jacobian and varies continuously with the
endomorphism. To the best of our knowledge it seems these results are new even in the case of attractors for Axiom A  endomorphisms.
 Finally, we exhibit open conditions that define partially hyperbolic non-singular endomorphisms for which our results apply. 
These include endomorphisms Derived from Anosov (inspired by the examples of Ma\~n\'e studied
in \cite{Car93,ABV00}) obtained from Anosov endomorphisms by means of local bifurcations 
(e.g. pitchfork or Hopf bifurcations). 

This paper is organized as follows. In Section~\ref{sec:Main} we describe the classes of 
partially hyperbolic endomorphisms considered here, we state our main results and give a glance on
the proofs. In Section~\ref{sec:examples} we give some applications. In Section~\ref{sec:Preliminaries} we discuss some concepts to the full understanding of our arguments, as the notion of hyperbolicity for endomorphism, the concept of SRB measures, natural extensions and the characterization of hyperbolicity via cone fields. In Section~\ref{sec:existence_uniqueness} we give the proof of the finiteness (uniqueness when
the dynamics is transitive) of SRB measures. In Section~\ref{sec:stability} we prove the statistical stability of the SRB measures, while the proof of the continuity of the entropy of SRB measures appears in Section~\ref{subsec:open}.
Finally, in Section~\ref{sec:future} we relate our results with others for partially hyperbolic endomorphisms and discuss some possible future
directions.

\section{Main results}\label{sec:Main}

\subsection{Setting}
Throughout, let $M$ be a compact connected Riemannian $d$-dimensional manifold.
Assume that $f:M\rightarrow M$ is a $C^{1+\alpha}$, $\alpha>0$, local diffeomorphism 
and that  $\Lambda\subset M$ a compact positively $f$-invariant subset.
Let $U\supset\Lambda$ be an open set such that $f(\overline U) \subset U$ and assume
$$
\Lambda= \Lambda_f:=\bigcap_{n\geq0}f^{n}(\overline{U}),
$$
where $\overline{U}$ denotes the closure of the set $U$.
Assume that there are a continuous splitting of the tangent bundle
$T_{U}M=E^{s}\oplus F$ (where $F$ is not necessarily $Df$-invariant)
and constants $a, c>0$ and $0<\lambda<1$ such that:
\begin{enumerate}
\item [(H1)]$Df(x)\cdot E_{x}^{s}=E_{f(x)}^{s}$
for every $x\in U$;\label{enu:invstabledirection}
\item [(H2)]$\|Df^{n}(x)|_{E^{s}_x}\|\leq\lambda^{n}$ for every $x\in U$ and $n\in\mathbb{N}$;\label{enu:contstabledirection}
\item [(H3)] the cone field $U\ni x\mapsto C_{a}(x):=\{ u+v \in E^s_x\oplus F_x : \|u\| \le a\|v\| \}$ satisfies:
	\begin{itemize}
	\item $Df(x) (C_{a}(x))\subseteq C_{a}(f(x))$ for every $x\in U$, and 
	\item there is a positive Lebesgue measure set $H\subset U$ 
	so that: 
	\begin{equation}\label{eq:nuexpansion}
	\limsup_{n\to\infty}\frac{1}{n}\sum_{j=0}^{n-1}\log\|(Df(f^{j}(x))|_{C_{a}(f^{j}(x))})^{-1}\|\leq-2c<0
	\end{equation}
	 for every $x\in H$\emph{\label{enu:conenue}} (the expression $\|(Df(f^{j}(x))|_{C_{a}(f^{j}(x))})^{-1}\|$ is made precise at Subsection~	\ref{sec:conehyperbolicity});
	 \end{itemize}
\item [(H4)]$\left\Vert Df(x)\, v\right\Vert \,\left\Vert Df(x)^{-1} w\right\Vert \leq\lambda\,\left\Vert v\right\Vert \,\left\Vert w\right\Vert $ for every $v\in E_{x}^{s}$, all $w\in Df(x) ( C_{a}(x))$
and $x\in U$.\label{enu:domination}
\end{enumerate}

Given a local diffeomorphism 
as above, we will say that $\Lambda$ is an \emph{attractor},
$E^{s}$ is a \emph{uniformly contracting bundle} and that $C_{a}$ is a
 \emph{non-uniformly expanding cone field} (or that $C_{a}$ exhibits non-uniform expansion). 
 We will denote 
 $d_{s}=\dim(E^{s})$
and $d_{u}:=\dim(F)=\dim(C_{a})$.  

\subsection{Statements}

Our first result concerns the existence and finiteness of SRB measures for partially hyperbolic endomorphisms 
(we refer the reader to Section~\ref{sec:defSRB} for the definition of SRB measures).

\begin{maintheorem}\label{thm:TEO1} 
Let $f$ be a $C^{1+\alpha}$ local diffeomorphism satisfying (H1) - (H4). Then there are finitely 
many SRB measures for $f$ whose basins of attraction cover $H$,
Lebesgue mod zero. Moreover, these measures are hyperbolic and, if $Leb(U\setminus H)=0$ then there are finitely many 
SRB measures for $f$ in $U$. Furthermore,
if $f\mid_\Lambda$ is transitive then there is a unique SRB measure for $f$ in $U$.
\end{maintheorem}

Some comments are in order.
SRB measures are ergodic and 
have absolute continuous disintegration with respect to Lebesgue measure along unstable manifolds 
(cf. Subsection~\ref{sec:defSRB}).
It is not hard to check that the existence of the subbundle $E^{s}$ 
is a $C^1$-open condition, as it is equivalent to the existence of a stable cone field, and that whenever
it is trivial then \eqref{eq:nuexpansion} corresponds to the usual notion of non-uniform expansion
in \cite{ABV00}.
Condition (H3) implies the existence of a disk 
$D\subset U$ which is tangent to the cone field so
that $Leb_{D}(H)>0$ (here $Leb_{D}$ stands for the 
Lebesgue measure on $D$),
which is a sufficient condition for the existence of SRB measures. 
Nevertheless, the full strength of (H3) 
is used to guarantee the finiteness of the SRB measures whose basin have non-trivial intersection 
with the set $H$ in Theorem~\ref{thm:TEO1}. 
In the case of surfaces, an upper bound on the number of SRB measures can be given in terms of the number of 
homoclinic classes (cf. ~\cite{MT16,M17}).
Let $\mathcal{M}(M)$ denote the space of Borelian probability measures on $M$
endowed with the weak$^*$ topology. 

Our second result concerns the continuous dependence of the SRB measures in terms of the dynamics, the so called 
statistical stability of the SRB measures, in the weak$^*$ topology. For that, we assume that a family of local diffeomorphisms satisfy conditions (H1)-(H4) with uniform constants.
Denote by $End^{1+\alpha}(M)$ the set of $C^{1+\alpha}$ non-singular endomorphisms on $M$ (i.e. local diffeomorphisms) 
endowed with the $C^{1+\alpha}$ topology. 
Assume that $\Lambda$ is an attractor for $f$ and that $U\supset\Lambda$
an open neighborhood of $\Lambda$ such that $\cap_{n\geq0}f^{n}(\overline{U})=\Lambda$.
For every $g \in End^{1+\alpha}(M)$ that is $C^1$-close to $f$ consider $\Lambda_{g}:=\bigcap_{n\geq0}g^{n}(\overline{U})$.
Assume that there are an open neighborhood $\mathcal{V}\subset End^{1+\alpha}(M)$ of $f$,
constants $\lambda\in(0,1)$, $c>0$ and a family of cone fields $U\ni x\rightarrow C(x)$ 
of constant dimension $0<d_{u}\leq dim(M)$,
such that (H1)-(H4) holds with uniform constants on $U$: for every $g\in\mathcal{V}$ there exists 
a continuous splitting $T_{U}M=E^{s}(g)\oplus F(g)$ so that
\begin{enumerate}
\item [(H1)]$Dg(x)\cdot E_{x}^{s}(g)=E_{g(x)}^{s}(g)$
for every $x\in U$;
\item [(H2)]$\|Dg^{n}(x)|_{E^{s}_x(g)}\|\leq\lambda^{n}$ for every $x\in U$ and $n\in\mathbb{N}$
\item [(H3)] the cone field $U\ni x\mapsto C(x)$ contains the subspace $F(g)$, and is such that 
$Dg(x)  (C(x))\subseteq C(g(x))$
for every $x\in U$, and 
\begin{equation*}
\limsup_{n\to\infty}\frac{1}{n}\sum_{j=0}^{n-1}\log\|(Dg(g^{j}(x))|_{C(g^{j}(x))})^{-1}\|\leq-2c<0
\end{equation*}
for Lebesgue almost every $x\in U$; and
\item [(H4)]$\left\Vert Dg(x)\cdot v\right\Vert \, \left\Vert Dg(x)^{-1}\cdot w\right\Vert \leq\lambda\, \left\Vert v\right\Vert \cdot\left\Vert w\right\Vert $ for every $v\in E_{x}^{s}(g)$, all $w\in Dg(x)(C(x))$
and $x\in U$.\label{enu:domination}
\end{enumerate}
We will say that $\mathcal{V}$ is an \emph{open set of partially hyperbolic $C^{1+\alpha}$-local diffeomorphisms with
uniform constants}. For each $g\in\mathcal{V}$, we denote by $H_{g}$
the subset of $U$ that consists of points satisfying the non-uniform cone-hyperbolicity condition in 
(H3). 

\begin{maintheorem}\label{thm:TEO2} 
Given $\alpha>0$,  let $\mathcal{U}$
be an open set of $C^{1+\alpha}$ local diffeomorphisms so that every $g\in\mathcal{U}$
satisfies (H1) - (H4) with uniform constants. If $(g_{n})_{n\in\mathbb{N}}$
is a sequence in $\mathcal{U}$ converging to $g\in\mathcal{U}$ and, for every $n\in \mathbb N$, $\mu_n$ is
an SRB measure for $g_n$ on $\Lambda_{g_n}$ then 
every accumulation point $\mu$ of $(\mu_{n})_n$ is a convex combination of the SRB
measures for $g$ on $\Lambda_g$. In particular, if every $g\in\mathcal{U}$ is transitive
and $\mu_g$ denotes the unique SRB measure for $g$ then 
\[
\begin{array}{ccc}
\mathcal{U} & \rightarrow & \mathcal{M}(\Lambda_g) \subset \mathcal{M}(M)\\
g & \mapsto & \mu_{g}
\end{array}
\]
is continuous with respect to the weak$^*$ topology.
\end{maintheorem}

Using the statistical stability of SRB measures we also prove the continuous dependence of the entropy function associated to the SRB measure.

\begin{maincorollary}\label{cor:entropy-cont}
Let $\cU$ be a open set of transitive $C^{1+\alpha}$ local diffeomorphisms that satisfy (H1)-(H4)
with uniform constants. If $\mu_f$ denotes the unique SRB measure for $f\in \cU$ then the entropy function
\[
\begin{array}{rccc}
H: & \mathcal{U} & \rightarrow & \mathbb R^+\\
& f & \mapsto & h_{\mu_{f}}(f)
\end{array}
\]
is continuous, where $h_{\mu_{f}}(f)$ denotes the entropy of $\mu_f$ with respect to $f$.
\end{maincorollary}

It is clear that an attractor of an Axiom A endomorphism satisfies the assumptions of Theorems~\ref{thm:TEO1}
and ~\ref{thm:TEO2}. In particular we obtain the following direct consequences that we did not find in the literature:

\begin{maincorollary}
Let $f$ be an $C^{1+\alpha}$ local diffeomorphism 
and $U\subset M$ be such that
$\Lambda_f=\bigcap_{n\ge 0} f^n(U)$ is a transitive hyperbolic attractor for $f$. 
Assume that $\mathcal{U}$ is a $C^{1+\alpha}$-open neighborhood of $f$ such that for any 
$\mathcal{U} \ni g \mapsto \Lambda_g$ denotes the analytic continuation of the hyperbolic attractor. 
Then $f\mid_{\Lambda_f}$ is statistically stable: if $(f_k)_k$ is a sequence of local diffeomorphisms in $\cU$
so that $f_k \to f$ as $k\to\infty$ (in the $C^{1+\alpha}$-topology) then $\mu_{f_k} \to \mu_f$ in the weak$^*$ topology.
\end{maincorollary}

In the remaining of this section we exhibit $C^1$-open sets of $C^{1+\alpha}$ 
partially hyperbolic non-singular endomorphisms for which it is checkable that 
assumptions (H1)-(H4) hold with uniform constants.
Let $\mathcal{U}\subset End^{1+\alpha}(M)$ be an open set of local diffeomorphisms on $M$
for which there are constants $\lambda\in(0,1)$, $c>0$ and a family of cone fields $U\ni x\rightarrow C(x)$ 
of constant dimension $0<d_{u}\leq dim(M)$,
and there is an open region $\mathcal{O}\subset M$ and a partition $\left\{ V_{1},...,V_{p},V_{p+1},...,V_{p+q}\right\}$ of $M$,
and $\sigma>q$, such that for every $g\in\mathcal{U}$ there exists 
a continuous splitting $T_{U}M=E(g)\oplus F(g)$ and $a>0$ so that
\begin{enumerate}
\item[(A1)] the family of cone fields $x\mapsto C_{a}^{-}(x)$ of width $a>0$ centered at $E(g)$ satisfies\label{(A1)}
\begin{enumerate}
\item $Dg(x)^{-1} ( C_{a}^{-}(g(x)))\subsetneq C_{a}^{-}(x)$
for every $x\in M$;
\item there exists $\lambda_{s}>1$ such that $\|Dg(x)^{-1}\cdot v\| > \lambda_{s}\|v\|$
for every $v\in C_{a}^{-}(g(x))$ and $x\in M$;
\end{enumerate}
\item[(A2)] there exists $L:M\rightarrow(0,\infty)$, determined by (A4)(b), in such a way that \label{(A2)}
\begin{enumerate}
\item $Dg(x) ( C(x))\subsetneq C(g(x))$ for every $x\in M;$
\item $\|Dg(x)\cdot v\| > L(x)\|v\|$ for every $v\in C(x)$ and for every $x\in M$;
\end{enumerate}
\item[(A3)] there exists $\sigma>1$ so that $\det\left|Df(x)|_E\right|>\sigma$ for every subspace $ E $ 
of dimension $ \dim(F) $ contained in $ C(x) $, for every $x\in M$; \label{(A3)}
\item[(A4)] \label{(A4)}
the set $\mathcal{O}$ is contained in $\cup_{j=1}^{q}V_{p+j}$  
and 
\begin{enumerate}
\item for every disk $D$ tangent to the cone field $C(\cdot)$, the image $g(D \cap V_i) \cap V_j$ has at most
	one connected component for all $i,j\in \{1, 2,\dots, p+q\}$,
\item $L(x)\geq\lambda_{u}>1$ for every $x\in M\backslash\mathcal{O}$ and $L(x)\geq L$ for every 
		$x\in\mathcal{O}$;
\end{enumerate}
where the constants $L$ and $\gamma$ are determined by \eqref{eqL1domin}, \eqref{eq:determineL} and 	
Lemma~\ref{lem:freq_visit}.
\end{enumerate}
It is clear from the definition that the previous conditions are $C^1$-open and that whenever $L>1$ (or equivalently $\mathcal O =\emptyset$) 
corresponds to the hyperbolic setting. 
Endomorphisms satisfying (A1)-(A4) with uniform constants arise naturally in the context of local bifurcations of hyperbolic endomorphisms
(cf Section~\ref{sec:examples}).
In Section~\ref{subsec:open} we prove that the previous class of endomorphisms satisfies the assumptions of Theorems~\ref{thm:TEO1} 
and ~\ref{thm:TEO2} and, consequently:

\begin{maincorollary}\label{thm:open}
Every $C^{1+\alpha}$ local diffeomorphism $f\in \mathcal U$ has a finite number of hyperbolic SRB measures, whose basins of attraction cover Lebesgue almost every points in $U$. Moreover, every $f\in \cU$ is statistically stable.
\end{maincorollary}

\subsection{Overview of the construction and stability of SRB measures}\label{subsec:overview}

We consider local diffeomorphisms that admit an invariant uniformly contracting direction
and a non-uniformly expanding cone field centered on a complementary
direction. 
Let us describe the main differences between the strategy for the construction of SRB measures, inspired from \cite{ABV00},  
in the current non-invertible setting.
First, the 
partial hyperbolicity assumptions (H1), (H2) and (H4) guarantee that small disks tangent to the cone
field $C_a$ are preserved under iteration of the dynamics. By (H3), there exists a disk $D$ 
that is tangent  to the cone field $C_a$ and so that \eqref{eq:nuexpansion} holds for a positive Lebesgue measure set in $D$. This is possible because $H\subset U$ has Lebesgue positive measure.
In particular, one can expect SRB measures to arise from the ergodic components of the 
accumulation points of the C\`esaro averages
\begin{equation}\label{eq:munD}
\mu_{n}=\frac{1}{n}\sum_{j=0}^{n-1}f_{*}^{j}Leb_{D},
\end{equation}
where $Leb_{D}$ is the Lebesgue measure on the disk $D$, inherited by the Riemannian structure on $D$.

In opposition to ~\cite{ABV00}, our assumptions do not imply the subbundle $F$ to be $Df$-invariant
even if $f$ happens to be invertible. For that reason, we use the non-uniform expansion along the cone
field $C_a$  (given by (H3)) to prove the existence of positive frequency of cone-hyperbolic times for every points in $H\cap D$. 
This condition 
assures that
for every $x\in H\cap D$ 
there are infinitely many values of $n$ (with positive frequency)  so that 
$$
\Vert Df^{n}(x)\cdot v\Vert\geq e^{cn}\|v\|
	\quad \text{for {\bf all} }\, v\in C_a(f^{j}(x)) 
$$
(cf. Lemma~\ref{le:unifexp-cht}). The previous uniformity is crucial because unstable directions (to be defined
 \emph{a posteriori} at almost every point by means of Oseledets theorem) 
will be contained in the cone field $C_a$ but cannot be determined a priori. 
This will be also important to prove that these SRB measures are hyperbolic measures. 
Moreover, it guarantees that if $\mathcal{D}_{n}$ denotes a suitable family of disks in $D$ that are 
expanded by $f^{n}$ and $\nu$ is an accumulation point of the measures 
$
\frac{1}{n}\sum_{j=0}^{n-1}f_{*}^{j}Leb_{\mathcal{D}_{j}}
$
then its support is contained in a union of disks 
obtained as accumulations of disks of uniform size $f^{n}(D_n)$, for $D_n\in \mathcal{D}_{n}$, 
as $n\to\infty$.

A second key difference between our setting and the context of partially hyperbolic diffeomorphisms is that, 
due to the possible non-invertibility of the dynamics, the iterates of a disk tangent to the cone $C_a$ 
may have self intersections and stop being a submanifold. In general the accumulation disks $\Delta$
of the family of disks $f^{n}(\mathcal D_{n})$ tangent to $C_a$  may have self-intersections. 
We overcome this difficulty by a selection procedure on the space of pre-orbits (which also involve 
the lifting of a reference measure to the natural extension) in order to prove that  
almost every orbit has a well defined disk, tangent to the cone field $C_a$, with backward contraction by
concatenation of all corresponding inverse branches. 
This construction involves a careful selection of pre-orbits so that accumulation disks are contained in
unstable manifolds parametrized by elements of the natural extension.
As a consequence, any invariant measure supported on these disks will be
hyperbolic since it will have only positive Lyapunov exponents, uniformly bounded away from zero, in the direction complementary 
to the stable bundle.
Moreover, the previous selection process will allow to define unstable manifolds at almost everywhere point and,
consequently, to obtain an almost everywhere defined invariant splitting (cf. Remark \ref{rmk:saturatedS}).  
In particular, invariant stable manifolds are well defined almost everywhere  (see Corollary on page 37 of \cite{PS}
or \cite[Theorems V.6.4 and V.6.5]{QXZ09}). 
Moreover, hypothesis (H1)-(H2) imply the H\"older continuity of the stable bundle and the absolute continuity of the stable foliation (cf.
\cite[Subsections V.7 and V.8]{QXZ09}).
Altogether, the latter ensures, by a Hopf-like argument,
that there are accumulation points of the sequence of probability measures $\mu_n$ given by ~\eqref{eq:munD} 
(hence $f$-invariant) for which some ergodic component is a physical measure for the local diffeomorphism 
$f$. The proof of the SRB property involves a careful and technical choice of partitions adapted to unstable disks on the natural extension that resemble an  inverse Markov tower. Their finiteness follows from the fact that the basin of every SRB measure 
contains an open set with a definite proportion (with respect to Lebesgue) of the phase space and all
SRB measures can be obtained by the previous procedure. 
This completes the guideline for proof of the existence and finiteness of SRB measures.

The proof of the stability of the SRB measures relies on the fact that, in the case that assumptions (H1)-(H4)
can be taken with uniform constants, one can prove that the size of unstable disks tangent to the cone field 
can be proved to be uniform. This allows us to prove that unstable disks and the (not necessarily invariant) absolutely
continuous invariant measures obtained by the previous limiting procedure vary continuously with the local diffeomorphism.
Finally, the continuity of the entropy of the SRB measures is a byproduct of the statistical stability together with Pesin's entropy
formula for local diffeomorphisms (see Section~\ref{subsec:open}).

\section{Some robust classes of partially hyperbolic local diffeomorphisms\label{sec:examples}}

In this section we discuss some applications of our results in relation with some natural bifurcations of
expanding and hyperbolic endomorphisms.

\subsection{Non-uniformly expansion and mostly expanding partially hyperbolic diffeomorphisms }

In \cite{ABV00}, the authors constructed SRB measures for two robust classes of non-uniformly
hyperbolic maps: 
\smallskip

\noindent (I) $C^{1+\al}$-non-singular endomorphisms $f$ ($\al>0$) on a compact Riemannian $M$
for which
\begin{equation}\label{eq:ABVM}
\limsup_{n\to\infty}\frac{1}{n}\sum_{j=0}^{n-1}\log\| Df(f^{j}(x))^{-1} \| \leq-2c<0
\end{equation}
for a positive Lebesgue measure set of points $x\in M$; and 

\smallskip
\noindent  (II) $C^{1+\al}$-diffeomorphisms ($\al>0$) that exhibit 
a $Df$-invariant dominated splitting $TM=E^s\oplus E^c$ so that 
\begin{equation}\label{eq:ABVEc}
\limsup_{n\to\infty}\frac{1}{n}\sum_{j=0}^{n-1}\log\| (Df(f^{j}(x))\mid_{E^c_{f^j(x)}})^{-1} \| \leq-2c<0
\end{equation}
for a positive Lebesgue measure set of points $x\in M$.
\smallskip

The class of non-singular endomorphisms considered in (I) fits our setting (corresponding to the case where the subbundle $E^s$ is trivial and 
in which case conditions ~\eqref{eq:nuexpansion} and ~\eqref{eq:ABVM} coincide).
 In the case that $f$ is a partially hyperbolic diffeomorphism as in (II), condition ~\eqref{eq:ABVEc}
seems weaker than our non-uniform hyperbolicity assumption in ~\eqref{eq:nuexpansion}.
However, this is not the case. Indeed,  the existence of a dominated splitting implies that
there exists  $a>0$ small so that the cone field $C^+_a$ of width $a$ around $E^c$ is positively $Df$-invariant. 
By compactness of $\Lambda$, continuity and $Df$-invariance of the central subbundle $E^c$, one can reduce $a>0$ if necessary in order to guarantee that
$$
\| (Df(f^{j}(x))\mid_{{C^+_a}{(f^j(x)})})^{-1} \| 
	\le e^c \; \| (Df(f^{j}(x))\mid_{E^c_{f^j(x)}})^{-1} \|
$$
for every $x\in \Lambda$ and, consequently,
$\limsup_{n\to\infty} \frac{1}{n}\sum_{j=0}^{n-1} \log\| (Df(f^{j}(x))\mid_{{C^+_a}{f^j(x)}})^{-1} \| \leq-c $. 
This proves that such a partially hyperbolic diffeomorphism $f$ satisfies  ~\eqref{eq:nuexpansion} if and only if  
~\eqref{eq:ABVM} holds (possibly with 
different constants). In particular, our results apply to the class of non-uniformly hyperbolic maps considered in \cite{ABV00}
and Corollary~\ref{cor:entropy-cont} provides an alternative proof of the continuity of the SRB measures.

\subsection{Endomorphisms derived from Anosov}

Dynamical systems in the isotopy class of uniformly hyperbolic ones have been 
intensively studied in the last decades. First examples of $C^1$-robustly transitive 
non-Anosov diffeomorphisms were considered by  Ma\~n\'e 
and their SRB measures were 
studied in \cite{Car93,Cas04}. 
In this subsection we illustrate how some partial hyperbolic and non-hyperbolic endomorphisms 
can be obtained, by local perturbations,  in the isotopy class of hyperbolic endomorphisms.
In particular, it is enough to show such perturbations can be done in such a way the endomorphism satisfies
the requirements of Subsection~\ref{subsec:open}.
There are non-hyperbolic topologically mixing partially hyperbolic endomorphisms on $\mathbb T^2$ ~\cite{Sumi}.
We make the construction of the examples in dimension $3$ for simplicity although similar
statements hold in higher dimension. 
In what follows we give examples of dynamical systems with hyperbolic periodic points with different index but we
could also consider  the case of existence of periodic points with an indifferent direction. 

\subsubsection{Partially hyperbolic endomorphisms obtained by pitchfork bifurcations}
\label{exa:pitchfork} 
Consider $M=\mathbb{T}^{3}$, let $g:M\rightarrow M$ be
a linear Anosov endomorphism induced by a hyperbolic matrix
$A\in \mathcal M_{3\times 3}(\mathbb Z)$ displaying three real eigenvalues and 
such that $TM=E^{s}\oplus E^{u}$ is the $Dg$-invariant splitting, 
where $\dim E^{u}=2$. For instance, take a matrix of the form
$$A=
\left(
\begin{array}{ccc}
n & 1 & 0 \\
1 & 1 & 0 \\
0 & 0 & 2 
\end{array}
\right)
$$
for an integer $n \ge 2$. The map $g$ is a special Anosov endomorphism, meaning that the unstable space 
does not depend on the pre-orbits and it admits a finest dominated splitting $E^s \oplus E^u \oplus E^{uu}$.
If $p\in M$ is a fixed point for $g$ 
and $\delta>0$ is sufficiently small, 
one can write $g$ on the ball $B(p,\delta)$ (in terms of local coordinates in $E_{p}^{s}\oplus E_{p}^{u}$) by
$
g(x,y,z)=(f(x), h(y,z)), 
$
where $x\in E_{p}^{s}$, $(y,z)\in E_{p}^{u}$, $f$ is a contraction on $E_{p}^{s}$ and
$g$ is an expansion on $E^u_p$.
Let $\lambda_{2},\lambda_{3}\in\mathbb{R}$ be the eigenvalues of
$Dg(p)\mid_{E_{p}^{u}}$. Suppose that $\left|\lambda_{2}\right|\geq\left|\lambda_{3}\right|>1$.
Consider an isotopy $\left[0,1\right]\ni t\mapsto h_{t}$ satisfying:
(i) $h_{0}=h$;
(ii) for every $t\in\left[0,1\right]$ the diffeomorphism $h_{t}$ has a fixed point $p_{t}$ (continuation
of $p$) which, without loss of generality, we assume to coincide with $p$;
(iii) $h_{t}: E_{t,p}^{c}\rightarrow E_{t,p}^{c}$
is a $C^{1+\alpha}$ map, where $E_{t,p}^{c}:=E_{p}^{u}$
for every $t\in\left[0,1\right]$;
(iv) the eigenvalues of $Dh_{1}$ at $p$ are $\lambda_{2}$ and $\rho\in\mathbb{R}$
with $\left|\rho\right|<1$ (determined by ~\eqref{rhoeq1} and ~\eqref{eqdefrho})
(v) if $g_t(x,y,z)=(f(x), h_t(y,z))$ then $g_{t}\mid_{M\backslash B(p,\delta)}=g\mid_{M\backslash B(p,\delta)}$
for every $t\in [0,1]$.
Reducing $\delta>0$ if necessary, we may assume that $g_{t}\mid_{B(p,\delta)}$
is injective for every $t\in\left[0,1\right]$.

\begin{figure}[htb]
\includegraphics[scale=.6]{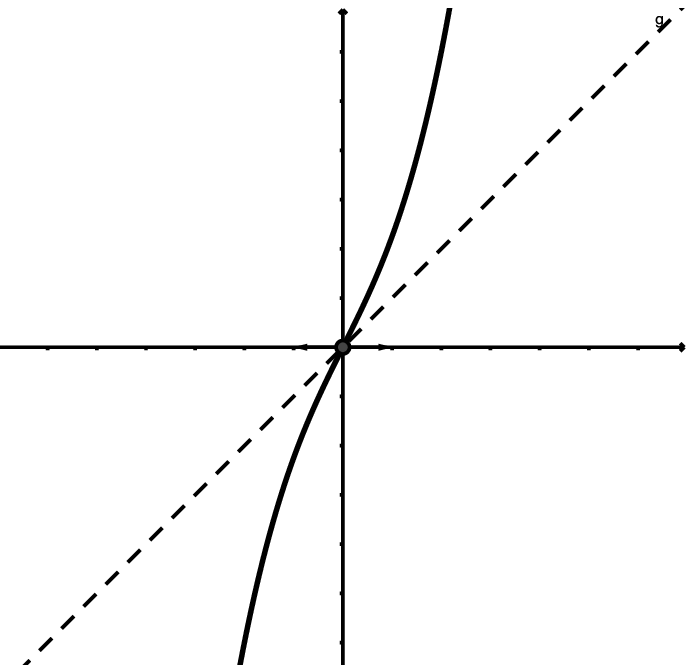} \quad \includegraphics[scale=.745]{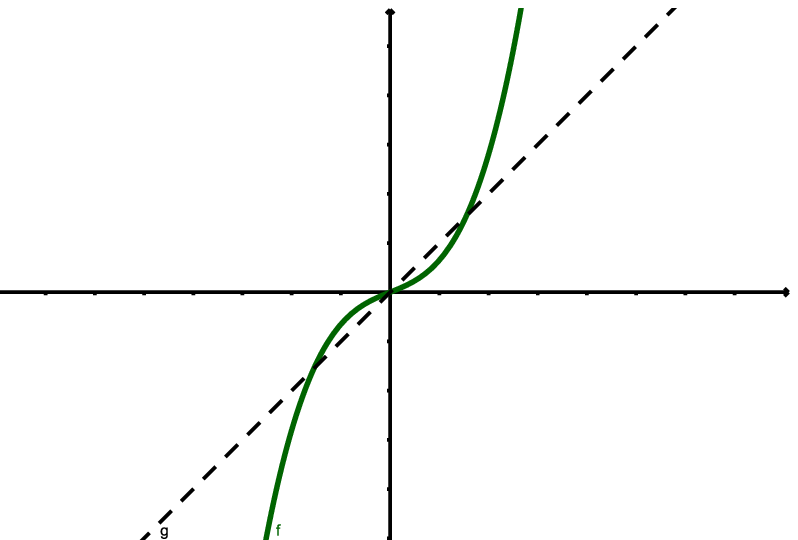}
\caption{Pitchfork bifurcation in dimension one: repelling fixed point (on the left) an attracting fixed point whose 
topological boundary of attraction is formed by two repelling fixed points (on the right)}
\end{figure}

Roughly, the fixed point $p$ goes through a pitchfork bifurcation along a one dimensional 
subspace contained in the unstable subspace associated to the original dynamics 
on the open set $\mathcal{O}=B(p,\delta)$. 
Given $a\in(0,1)$ and the splitting $TM=E_{t}^{s}\oplus E_{t}^{c}$ consider the families of cone fields
\[
\mathcal{C}_{a}^{s}(x):=\left\{ v=v_{s}\oplus v_{c}:\ \left\Vert v_{c}\right\Vert \leq a \left\Vert v_{s}\right\Vert \right\} 
	\quad\text{and}\quad
\mathcal{C}_{a}^{u}(x):=\left\{ v=v_{s}\oplus v_{c}:\ \left\Vert v_{s}\right\Vert \leq  a  \left\Vert v_{c}\right\Vert \right\}
\]
and assume without loss of generality that $\|(u,v,w)\|=|u|+|v|+|w|$ for $(u,v,w) \in E^s\oplus E^u \oplus E^{uu}$ 
(local coordinates for $g$). If $a>0$ is small then the families of 
cone fields are $Dg_1$-invariant. Indeed,
one can write $v=v_{s}\oplus v_{c} \in\mathcal{C}^{s}(g_{1}(p))$, with $v_{*}\in E_{1,g_1(p)}^{*}$,
$*\in\{s,c\}$. 
Assume $\rho$ is such that 
\begin{equation}\label{rhoeq1}
\left|\lambda_{1}\cdot\rho^{-1}\right| < 1.
\end{equation}
Then
$
\left\Vert Dg_{1}(g_1(p))^{-1}\cdot v_{c}\right\Vert  
	  \leq 
	 	\left|\rho\right|^{-1}\left\Vert v_{c}\right\Vert 
	  \leq\left|\rho\right|^{-1} a \left\Vert v_{s}\right\Vert  
	 	\leq\left|\lambda_{1}\right|  \left|\rho\right|^{-1}\cdot a \left\Vert Dg_{1}(p)^{-1}\cdot v_{s}\right\Vert,
$
which proves that $Dg_{1}(p)^{-1}\cdot \mathcal{C}_{a}^{s}(g_1(p)) \subset \mathcal{C}_{|\lambda_1 \rho^{-1}| a}^{s}(p)$.
Analogously, if $v\in\mathcal{C}_{1}^{u}(p)$ then 
$\left\Vert Dg_{1}(p)\cdot v_{s}\right\Vert \leq\left|\lambda_{2}^{-1} \lambda_{1}\right| a \left\Vert Dg_{1}(p)\cdot v_{c}\right\Vert$ and, consequently, 
$Dg_{1}(p)\cdot \mathcal{C}_{a}^{u}(p) \subset \mathcal{C}_{|\lambda_1\lambda_2^{-1}|a}^{u}(g_1(p))$.
By continuity, since $\delta$ is assumed to be small, we get the invariance of the cone fields for all points on the ball $B(p,\delta)$.
Now observe that $g_{1}\mid_{B(p,\delta)}$ expands volume along $E_{1,x}^{c}$ as it coincides with 
$g\mid_{B(p,\delta)}$. Moreover, if $\rho$ is such that 
\begin{equation}\label{eqdefrho}
|\lambda_2\cdot \rho |>1
\end{equation}
then $|\det(Dg_1(p)\mid_{E_{1,p}^{c}})|=|\lambda_2|\, |\rho|>1$
and, using that $a>0$ is small and $x\mapsto \det(Dg_1(x)\mid_{E_{1,x}^{c}})$ is continuous, 
we conclude that  $\inf_{x\in M} |\det(Dg_1(x)\mid_{E_{1,x}^{c}})| >1$ and that the same property holds
for the Jacobian along disks tangent to the cone field.
Simple computations show that $\left\Vert Dg_{1}(p)\cdot v\right\Vert \geq(1- a )\cdot\rho \left\Vert v\right\Vert $
for $v\in\mathcal{C}_{a}^{u}(p)$
and, 
by continuity of the derivative, we get that
$\left\Vert Dg_{1}(x)\cdot v\right\Vert \geq L \left\Vert v\right\Vert $
for every $x\in B(p,\delta)$ and $v\in C_{a}^{u}(x)$ with $L=L(\rho)$ close to $(1-a) \rho$, hence close to 1.
If $x\notin B(p,\delta)$, then $g_{1}\mid_{M\backslash B(p,\delta)}=g\mid_{M\backslash B(p,\delta)}$
and so $\left\Vert Dg_{1}(x)\cdot v\right\Vert \geq(1- a )\cdot\lambda_{3}\left\Vert v\right\Vert$ for every 
 $v\in\mathcal{C}_{a}^{u}(x)$.
Finally, since $\delta$ is assumed to be small, any partition $\left\{ V_{1},V_{2},\dots,V_{k}, V_{k+1}\right\} $ of $\mathbb T^3$ 
such that $V_{k+1}=B(p,\delta)$ and each $V_i$ contains a ball of radius $\delta$ and is contained in a ball of radius $2\delta$ satisfies the requirements of (A4)(a)
injectivity domains for $f$ (here $q=1$ and $\sigma>1$). 
By Corollary~\ref{thm:open}, the partially hyperbolic endomorphism $g_{1}$ has a finite number of SRB measures and it is statistically stable.
Moreover, if $g$ is robustly transitive (as e.g. in \cite{Sumi}) then Corollary~\ref{cor:entropy-cont} assures
that the entropy map $f \mapsto h_{\mu_f}(f)$ is continuous in a $C^{1+\alpha}$-small neighborhood of $g$.

\subsubsection{Partially hyperbolic attractors obtained through Hopf bifurcations}
Consider $M=\mathbb{T}^{3}$ and $g:M\rightarrow M$
a linear Anosov endomorphism of class $C^{1+\alpha}$.
Let $p\in M$ be a fixed point and let $T_{p}M=E_{p}^{s}\oplus E_{p}^{u}$
be the hyperbolic splitting where $\dim E_{p}^{u}=2$.
We can mimic the previous construction, considering a Hopf's bifurcation
through the fixed point $p$ instead of a pitchfork bifurcation. Assume that $Dg$ has two conjugated
eigenvalues whose absolute value is $\rho>1$. Consider an arc $\left[-1,1\right]\ni t\rightarrow g_{t}$
of $C^{1+\alpha}$ endomorphisms so that:
(i) for every $t\in\left[-1,1\right]$, $g_{t}$ is an endomorphism of class
$C^{1+\alpha}$ with a fixed point $p_{t}$ (we will assume that $p_{t}=p$
for every $t\in\left[-1,1\right]$)
(ii) $g_{-1}=g$,
(iii) for every $t<0$, $g_{t}$ is an Anosov endomorphism,
(iv) if $E_{t,p}^{c}$ is the continuation of the unstable space $E_{-1,p}^{u}$
to $t\in\left[-1,1\right]$ then $D g_{0}\mid_{E_{0,p}^{c}}$ has eigenvalues of absolute
value equal to $1$,
(v) the Hopf's bifurcation occurs for $t=0$, meaning that for $t>0$ sufficiently
small we have that the eigenvalues of $Dg_{t}(p)\mid_{E_{t}^{c}}$
direction are complex conjugate with absolute value is less than
$1$, there is an $g_t$-invariant repelling circle $C_{t} \subset B(p,\delta)$ around $p$ and $p$ is an attracting
fixed point, and 
(vi)  $g_{t}\mid_{M\backslash B(p,\delta)}=g\mid_{M\backslash B(p,\delta)}$
for every $t\in\left[-1,1\right]$.
An argument similar as the one used in the previous example guarantees that (A1)-(A4) are satisfied by
$g_1\mid_{\Lambda}$, where $\Lambda= M\setminus B(p)$ and $B(p)$ denotes the topological basin of
attraction of $p$. So, there exists a finite number of SRB measures
for $g_1\mid_{\Lambda}$ and that $g_1\mid_{\Lambda}$ is statistically stable.

\section{Preliminaries}\label{sec:Preliminaries}

The present section is devoted to some preliminary discussion on the notion of uniform and non-uniform
hyperbolicity for local diffeomorphisms. The reader acquainted with these notions may skip the
reading of this section and decide to return to it when necessary.

\subsection{Cone hyperbolicity and cone-hyperbolic times}\label{sec:conehyperbolicity}

Here we define non-uniform expansion along cone fields and define cone-hyperbolic times,
which extend the concept of hyperbolic times from \cite[Definition 2.6]{ABV00}.
Let $V$ and $W$ be vector spaces. Denote by $\mathcal{L}(V,W)$
the vector space of linear maps from $V$ to $W$. If $E\subset V$
is a cone in $V$ we endow the vector space 
$
\mathcal{L}_{E}(V,W):=\left\{ T\mid_{E}:\ T\in\mathcal{L}(V,W)\right\} 
$
with the norm $\|T|_{E}\|=\sup_{0\neq v\in E}\frac{\|T\cdot v\|}{\|v\|}$.
Given a non-singular endomorphism $f: M \to M$, a $Df$-invariant and convex cone field $C$ and $x\in M$, denote by
$(Df(x)|_{C(x)})^{-1}$ the map 
$$
Df(x)^{-1}\mid_{Df(x) (C(x))} : Df(x) (C(x)) \to C(x) \subset T_x M.
$$
Given a normed vector space $V$ such that $V=E\oplus F$ and $a>0$, the \emph{cone of width $a>0$ centered at $E$ }
is defined by $C_{a}:=\left\{  v=v_{1}\oplus v_{2} \in E\oplus F \colon \|v_{2}\|\leq a\|v_{1}\|\right\} $.
The \emph{dimension of the cone field $C_{a}$}, denoted by $\dim(C_{a})$, is the dimension of $E$.
In particular, if $TM=E\oplus F$ is a continuous splitting of the tangent bundle, the \emph{cone field of width $a>0$ centered on $E$}
is the continuous map $x\mapsto C_{a}(x)$ that assigns to each $x\in M$ the cone of width $a>0$ centered at $E_{x}$.
Finally, a submanifold $D\subset M$ is \emph{tangent to
the cone field $x\mapsto C(x)$ }if $\dim(D)=\dim(C)$
and $T_{x}D\subset C(x)$, for every $x\in D$.
As unstable subspaces will be almost everywhere defined  
\emph{a posteriori} and will be contained in the cone field, some hyperbolicity will be required along 
cone fields. Let $\Lambda\subset M$ be a compact positively invariant subset of
$M$ and $\Lambda\ni x\mapsto C(x)$ be a $Df$-invariant cone field on $\Lambda$.

\begin{definition}
Let $M$ be a compact Riemannian manifold $M$. Let $f:M\rightarrow M$
be a local diffeomorphism and $c>0$. We say that $n\in\mathbb{N}$
is a \emph{$c$-cone-hyperbolic time} for $x\in M$ (with respect
to $C$) if $Df(f^{j}(x)) ( C(f^{j}(x)))\subset C(f^{j+1}(x))$
for every $0\leq j\leq n-1$ and 
\begin{equation}
\prod_{j=n-k}^{^{n-1}}\left\Vert (Df(f^{j}(x))|_{C(f^{j}(x))})^{-1}\right\Vert \leq e^{-ck}\label{eq:hyptimes}
	\quad \text{ for every $1\leq k\leq n$.}
\end{equation}
\label{def:hyptime}
\end{definition}

It is easy to see that if $n$ and $m$ are $c$-cone hyperbolic times, say $n<m$ then $n-m$ is a $c$-cone hyperbolic 
time for $f^{n}(x)$. One key property of cone-hyperbolic times is as follows:

\begin{lemma}\label{le:unifexp-cht}
Let $M$ be a compact Riemannian manifold, $f$ be a local diffeomorphism on $M$
 and $c>0$. If $n$ is a $c$-cone-hyperbolic time for $x\in M$ then $\Vert Df^{n-j}(f^{j}(x))\cdot v\Vert\geq e^{c(n-j)}\|v\|$
for every $v\in C(f^{j}(x))$ and $0\le j \le n-1$.
\end{lemma}

\begin{proof}
Fix $0\le j\le n-1$ and $v\in C(f^{j}(x))$, and set $w=Df^{n-j}(f^{j}(x))\cdot v$. 
By definition of cone-hyperbolic times one has $Df(f^{\ell}(x)) ( C(f^{\ell}(x)))\subset C(f^{\ell+1}(x))$
for every $0\leq \ell\leq n-1$.
In particular $Df^{\ell-j}(f^{j}(x))\cdot v\in Df(f^{\ell-1}(x))( C(f^{\ell-1}(x)))\subset C(f^{\ell}(x))$,
for every $j+1\leq \ell\leq n-1$.
As $n$ is a $c$-cone-hyperbolic time for $x$,
\begin{align*}
\Vert v\Vert & =\Vert Df(f^{j}(x))^{-1}\circ Df(f^{j+1}(x))^{-1}\circ\cdots\circ Df(f^{n-1}(x))^{-1}\cdot w\Vert\\
 & \leq\prod_{k=j}^{n-1}\|(Df(f^{k}(x))|_{C(f^{k}(x))})^{-1}\|\cdot\|w\|
 	\le e^{-c(n-j)} \|w\|,
\end{align*}
which guarantees that $\Vert Df^{n-j}(f^{j}(x))\cdot v\Vert\geq e^{c(n-j)}\|v\|$ and proves the lemma.
\end{proof}

\subsection{Natural extension\label{subsec:naturalextension}}

Let $(M,d)$ be a compact metric space and $f:M\rightarrow M$
be a continuous map. The natural extension of $M$ by $f$
is the set
$
M^{f}:=\left\{ \hat{x}=(x_{-j})_{j\in\mathbb{N}}:x_{-j}\in M\mbox{ and }f(x_{-j})=x_{-j+1},\mbox{ for every }j\in\mathbb{N}\right\} .
$
We can induce a metric on $M^{f}$ from the metric $d$ by setting
$
\hat{d}(\hat{x},\hat{y}):=\sum_{j\in\mathbb{N}}2^{-j}d(x_{-j},y_{-j}).
$
The space $M^{f}$ endowed with the metric $\hat{d}$ is a compact metric space and the topology
induced by $\hat{d}$ is equivalent to the topology induced by $M^{\mathbb{N}}$ 
endowed with Tychonov's product topology. The projection
$\pi:M^{f}\rightarrow M$ given by $\pi(\hat{x}):=x_{0}$ is a
continuous map. The lift of $f$ is the map $\hat{f}:M^{f}\rightarrow M^{f}$
given by
$
\hat{f}(\hat{x}):=(\dots,x_{-n},\dots,x_{-1},x_{0},f(x_{0})),
$
and it is clear that $f\circ\pi=\pi\circ\hat{f}$. 
Finally, for each $\hat{x}\in M^{f}$ we take $T_{\hat{x}}M^{f}:=T_{\pi(\hat{x})}M$
and set 
\[
\begin{array}{rccc}
D\hat{f}(\hat{x}):  & T_{\hat{x}}M^{f} & \rightarrow & T_{\hat{f}(\hat{x})} M^{f} \\
 & v & \mapsto & Df(\pi(\hat{x})).
\end{array}
\]
If $\Lambda\subset M$ is a compact positively invariant, that is,
$f(\Lambda)\subset\Lambda$, the natural extension
of $\Lambda$ by $f$ is the set of pre-orbits that lie on $\Lambda$, that is, 
$
\Lambda^{f}:=\{ \hat{x}=(x_{-j})_{j\in\mathbb{N}}:\ x_{-j}\in\Lambda\mbox{ and }f(x_{-j})=x_{-j+1}\mbox{ for every }j\in\mathbb{N}\} .
$
The projection $\pi$ induces a continuous bijection between $\hat{f}$-invariant probability measures and $f$-invariant probability measures: assigns 
every $\hat{\mu}$ in $M^{f}$ the push forward 
$\pi_{*}\hat{\mu}=\mu$ (see e.g. \cite{QXZ09} for more details).

\subsection{Uniform and non-uniform hyperbolicity for endomorphisms}\label{sec:UH}

\subsubsection{Uniform hyperbolicity for endomorphisms}

The notion of uniform hyperbolicity for endomorphisms was introduced in \cite{Prz76}.
If $f$ is a non-singular endomorphism for each $x\in M$ there is a neighborhood $V_{x}$ of $x$ 
such that $f\mid_{V_{x}}:V_{x}\rightarrow f(V_{x})$ admits a inverse $f^{-1}:f(V_{x})\rightarrow V_{x}$
(whose derivative will be denoted simply by $Df(x)^{-1}$).
We say a positively $f$-invariant set $\Lambda\subset M$ is \emph{uniformly hyperbolic} if 
there are families of cone fields $\Lambda\ni x\mapsto C^{s}(x)$ and $\Lambda\ni x\mapsto C^{u}(x)$
of constant dimension satisfying: (i) $Df(x)^{-1} ( C^{s}(f(x)))\subsetneq C^{s}(x)$
and $Df(x) ( C^{u}(x))\subsetneq C^{u}(f(x))$
for every $x\in\Lambda$; and (ii)
there exists $\sigma>1$ such that $\|Df(x)^{-1}\cdot v\|\geq\sigma\|v\|$
for every $v\in C^{s}(f(x))$ and $\|Df(x)\cdot w\|\geq\sigma\|w\|$
for every $w\in C^{u}(x)$.
The existence of a contracting cone field $C^{s}$ in item (i) assures that the 
stable subbundle $E^s$, defined by 
$E_{x}^{s}:=\bigcap_{n\in\mathbb{N}}Df^{n}(x)^{-1}( C^{s}(f^{n}(x)))$
for every $x\in \Lambda$, is $Df$-invariant and uniformly contracting. 
Uniform hyperbolicity can be used to define invariant splittings by means of the natural extension of $f$, that is
the (compact) set $M^f$ of preimages of $f$. Indeed, given a pre-orbit $\hat x =(x_{-n})_{n\in\mathbb{N}}\in M^f$ of some point $x=x_0$, the expansion along the unstable
cone field implies that the set
$E_{\hat x}^{u}:=\bigcap_{n\in\mathbb{N}}Df^{n}(x_{-n})( C^{u}(x_{-n}))$,
is a subspace (see e.g  Proposition~\ref{prop:hatnuishyp}). Moreover, 
$Df(x)\cdot E_{\hat x}^{u}=E_{\hat y}^{u}$,
whenever $\hat y = (\dots, x_{-2}, x_{-1}, x_0, f(x_0))$.
Clearly, $E_{\hat x}^{u}$ defines a subspace that is uniformly contracting by backward iterates along the pre-orbit 
$\hat x=(x_{-n})_{n\in\mathbb{N}}$ of $x_0$.
In particular, for each $\hat x\in M^f$, $T_{\hat x} M^f=E_{x}^{s}\oplus E_{\hat x }^{u}$ defines a $D\hat f$-invariant 
splitting on $TM^f:=\bigcup_{\hat x\in M^f} T_{\hat x} M^f$.
We remark that, in general, the natural
extension admits no differential structure.

\subsubsection{Lyapunov exponents and non-uniform hyperbolicity}
We need the following:
\begin{proposition}\cite[Proposition I.3.5]{QXZ09}
\label{prop:oseledet_nat_ext} 
Suppose that $M$ is a compact Riemannian
manifold. Let $f:M\rightarrow M$ be a $C^{1}$ map preserving a probability measure $\mu$. There is a Borelian
set $\hat{\Delta}\subset M^{f}$ with $\hat{f}(\hat{\Delta})=\hat{\Delta}$
and $\hat{\mu}(\hat{\Delta})=1$ satisfying that for every
$\hat{x}\in\hat{\Delta}$ there is a splitting 
$T_{\hat{x}}M^{f}=E_{1}(\hat{x})\oplus\dots\oplus E_{r(\hat{x})}(\hat{x})$
and numbers
$+\infty>\lambda_{1}(\hat{x})>\lambda_{2}(\hat{x})>\dots>\lambda_{r(\hat{x})}(\hat{x})>-\infty$ (Lyapunov exponents) and multiplicities $m_{i}(\hat{x})$ for $1\leq i\leq r(\hat{x})$
such that:
\begin{enumerate}
\item $D\hat{f}(\hat{f}^{n}(\hat{x})):T_{\hat{f}^{n}(\hat{x})} M^{f}\rightarrow T_{\hat{f}^{n+1}(\hat{x})} M^{f}$
is a linear isomorphism for every $n\in\mathbb{Z}$;
\item the functions $r:\hat{\Delta}\rightarrow\mathbb{N}$, $m_{i}:\hat{\Delta}\rightarrow\mathbb{N}$ and 
$\lambda_{i}:\hat{\Delta}\rightarrow\mathbb{R}$ are $\hat{f}$-invariant ;
\item $\dim(E_{i}(\hat{x}))=m_{i}(\hat{x})$
for every $1\leq i\leq r(\hat{x})\le \dim M $;
\item the splitting is $D\hat{f}$-invariant, that is, $D\hat{f}(\hat{x})\cdot E_{i}(\hat{x})=E_{i}(\hat{f}(\hat{x}))$,
for every $1\leq i\leq r(\hat{x})$;
\item $\lim_{n\to\pm\infty}\frac{1}{n}\log\|D\hat{f}^{n}(\hat{x})\cdot u\|=\lambda_{i}(\hat{x})$
for every $u\in E_{i}(\hat{x})\backslash\left\{ 0\right\} $
and for every $1\leq i\leq r(\hat{x})$;
\item if
$
\rho_{1}(\hat{x})\geq\rho_{2}(\hat{x})\geq\dots\geq\rho_{d} (\hat{x})
$
 represent the numbers $\lambda_{i}(\hat{x})$ repeated
$m_{i}(\hat{x})$ times for each $1\leq i\leq r(\hat{x})$
and $\left\{ u_{1},u_{2},\dots,u_{d}\right\} $is a basis for $T_{\hat{x}}M^{f}$
satisfying
$
\lim_{n\to\pm\infty}\frac{1}{n}\log\|D\hat{f}^{n}(\hat{x})\cdot u_{i}\|=\rho_{i}(\hat{x})
$
for every $1\leq i\leq d$, then for any subsets $P,Q\subset\left\{ 1,2,\dots,d\right\} $ with
$P\cap Q=\emptyset$ one has
$
\lim_{n\to\infty}\frac{1}{n}\log\angle(D\hat{f}^{n}(\hat{x})\cdot E_{P},D\hat{f}^{n}(\hat{x})\cdot E_{Q})=0
$
where $E_{P}$ and $E_{Q}$ are, respectively, the subspaces generated
by $\left\{ u_{i}\right\} _{i\in P}$ and $\left\{ u_{i}\right\} _{i\in Q}$.
\end{enumerate}
Moreover, if $\hat\mu$ is ergodic then the previous functions are constant $\hat\mu$-almost everywhere.
\end{proposition}

Given an $f$-invariant probability measure $\mu$, we say that $\mu$ is \emph{hyperbolic} if it has no zero Lyapunov exponents. 
Associated to this formulation of the concept of Lyapunov exponents
we have the existence of unstable manifolds for almost every point with respect to some invariant
measure with some positive Lyapunov exponents. Consider the subspaces 
$
E_{\hat{x}}^{u}:=\bigoplus_{\lambda_{i}(\hat{x})>0}E_{i}(\hat{x})\text{ and }E_{\hat{x}}^{cs}:=\bigoplus_{\lambda_{i}(\hat{x})\leq0}E_{i}(\hat{x}).
$

\begin{proposition}\cite[Proposition V.4.4]{QXZ09}
\label{prop:unstable_manifold} There is a countable number of compact
subsets $(\hat{\Delta}_{i})_{i\in\mathbb{N}}$, of $M^{f}$ with $\bigcup_{i\in\mathbb{N}}\hat{\Delta}_{i}\subset\hat{\Delta}$
and $\hat{\mu}(\hat{\Delta}\backslash\bigcup_{i\in\mathbb{N}}\hat{\Delta}_{i})=0$
such that:
\begin{enumerate}
\item for each $\hat{\Delta}_{i}$ there is $k_{i}\in\mathbb{N}$ satisfying
$\dim(E^{u}(\hat{x}))=k_{i}$ for every $\hat{x}\in\hat{\Delta}_{i}$ and the subspaces
$E_{\hat{x}}^{u}$ and $E_{\hat{x}}^{cs}$ depend continuously on
$\hat{x}\in\hat{\Delta}_{i}$;
\item for any $\hat{\Delta}_{i}$ there is a family of $C^{1}$ embedded
$k_{i}$-dimensional disks $\left\{ W_{loc}^{u}(\hat{x})\right\} _{\hat{x}\in\hat{\Delta}_{i}}$
in $M$ and real numbers $\lambda_{i}$, $\vep\ll\lambda_{i}$,
$r_{i}<1$, $\gamma_{i}$, $\alpha_{i}$ and $\beta_{i}$ such that
the following properties hold for each $\hat{x}\in\hat{\Delta}_{i}$:
\begin{enumerate}
\item there is a $C^{1}$ map $h_{\hat{x}}:O_{\hat{x}}\rightarrow E_{\hat{x}}^{cs}$,
where $O_{\hat{x}}$ is an open set of $E_{\hat{x}}^{u}$ which contains
$\left\{ v\in E_{\hat{x}}^{u}:\ \|v\|<\alpha_{i}\right\} $ satisfying: (i)
$h_{\hat{x}}(0)=0$ and $Dh_{\hat{x}}(0)=0$; (ii) Lip$(h_{\hat{x}})\leq\beta_{i}$ and $Lip(Dh_{\hat{x}}(\cdot))\leq\beta_{i}$;
and (iii) $W_{loc}^{u}(\hat{x})=\exp_{x_{0}}(graph(h_{\hat{x}}))$.
\item for any $y_{0}\in W_{loc}^{u}(\hat{x})$ there is a unique
$\hat{y}\in M^{f}$ such that $\pi(\hat{y})=y_{0}$ and
$$
dist(x_{-n},y_{-n})\leq \min\{ r_{i}e^{-\vep_{i}n}, \gamma_{i}e^{-\lambda_{i}n}dist(x_{0},y_{0})\}, \; \text{ for every }n\in\mathbb{N},
$$
\item if
$
\hat{W}_{loc}^{u}(\hat{x}):=\big\{ \hat{y}\in M^{f}:\ y_{-n}\in W_{loc}^{u}(\hat{f}^{-n}(\hat{x})),\text{ for every }n\in\mathbb{N}
\big\}$
then $\pi:\hat{W}_{loc}^{u}(\hat{x})\rightarrow W_{loc}^{u}(\hat{x})$
is bijective and $\hat{f}^{-n}(\hat{W}_{loc}^{u}(\hat{x}))\subset\hat{W}_{loc}^{u}(\hat{f}^{-n}(\hat{x}))$.
\item for any $\hat{y},\hat{z}\in\hat{W}_{loc}^{u}(\hat{x})$ it holds 
$
dist_{\hat{f}^{-n}(\hat{x})}^{u}(y_{-n},z_{-n})\leq\gamma_{i}e^{-\lambda_{i}n}dist_{\hat{x}}^{u}(y_{0},z_{0})
$
for every $n\in\mathbb{N}$, where $dist_{\hat{x}}^{u}$ is the distance
along $W_{loc}^{u}(\hat{x})$.
\end{enumerate}
\end{enumerate}
\end{proposition}

\subsection{SRB measures}

The notion of SRB measure for endomorphisms depends intrinsically on the 
existence of unstable manifolds, whose geometry and construction are more involved than 
for diffeomorphisms as unstable manifolds may have self intersections (hence do not generate an invariant foliation). 

\subsubsection{Rokhlin's disintegration theorem}

Let $(X,\mathcal{A},\mu)$ be a probability space and $\mathcal{P}$
be a partition of $X$. We say that $\mathcal{P}$ is a measurable
partition if there is a 
sequence of countable partitions
of $X$, $(\mathcal{P}_{j})_{j\in\mathbb{N}}$, such that
$\mathcal{P}=\bigvee_{j=0}^{\infty}\mathcal{P}_{j}\ \mod0.$
Consider the projection $p:X\rightarrow\mathcal{P}$ that associates
to each $x\in X$ the atom $\mathcal{P}(x)$ that contains $x$.
We say that a subset $\mathcal{Q}\subset\mathcal{P}$ is measurable
if, and only if $p^{-1}(\mathcal{Q})$ is measurable. It is easy to see that the family of all measurable sets of $\mathcal{P}$
is a $\sigma$-algebra of $\mathcal{P}$. Set $\tilde{\mu}:=p_{*}\mu$.
\begin{definition}
Let $(X,\mathcal{A},\mu)$ be a probability space and let
$\mathcal{P}$ be a partition of $X$. A 
\emph{disintegration of $\mu$ with respect to $\mathcal{P}$} is a family $(\mu_{P})_{P\in\mathcal{P}}$
of probability measures on $X$ so that: 
\begin{enumerate}
\item $\mu_{P}(P)=1$ for $\tilde{\mu}$ a.e. $P\in\mathcal{P}$;
\item for every measurable set $E\subset X$ the map $\mathcal{P}\ni P\mapsto\mu_{P}(E)$ is measurable;
\item $\mu(E)=\int\mu_{P}(E) \, d\tilde{\mu}(P)$ for every measurable set $E\subset X$.
\end{enumerate}
\end{definition}
Rokhlin's theorem \cite{Ro49} guarantees that probability measures can be disintegrated with respect to
measurable partitions, and that if
$(\mu_{P})_{P\in\mathcal{P}}$and $(\mu'_{P})_{ P\in\mathcal{P}} $
are disintegrations of the measure $\mu$ with respect to $\mathcal{P}$
then $\mu_{P}=\mu'_{P}$ for $\tilde{\mu}$ almost every $P\in\mathcal{P}$. 

\subsubsection{SRB property for endomorphisms}\label{sec:defSRB}

In order to define SRB measures we need the notion of partitions (on $M^{f}$)
adapted to unstable manifolds of $(f,\mu)$.

\begin{definition}
\label{def:wupartition}Let $\mu$ be an $f$-invariant probability
measure with at least one positive Lyapunov exponent at $\mu$ almost
every point. A measurable partition $\mathcal{P}$ of $M^{f}$ is
\emph{subordinated to unstable manifolds} if for $\hat{\mu}$
a.e. $\hat{x}\in M^{f}$
(1) $\pi|_{\mathcal{P}(\hat{x})}:\mathcal{P}(\hat{x})\rightarrow\pi(\mathcal{P}(\hat{x}))$
is bijective, and
(2) 
there is a submanifold $W_{\hat{x}}$ with dimension $k(\hat{x})$
in $M$ such that $W_{\hat{x}}\subset W^{u}(\hat{x})$, $\pi(\mathcal{P}(\hat{x}))\subset W_{\hat{x}}$
and $\pi(\mathcal{P}(\hat{x}))$ contains an open neighborhood of $x_{0}$ in $W_{\hat{x}}$,
where $\mathcal{P}(\hat{x})$ denotes the element of the partition $\cP$ that contains $\hat x$ and
$k(\hat x)$ denotes the number of positive Lyapunov exponents at $\hat x$.
\end{definition}

We are now in a position to define the SRB property for invariant measures.
\begin{definition}
\label{def:srbproperty}We say that an $f$-invariant and ergodic probability
measure $\mu$ is an \emph{SRB measure} if  it has at least one positive Lyapunov exponent 
almost everywhere and for every partition $\mathcal{P}$ subordinated
to unstable manifolds one has 
$
\pi_{*}\hat{\mu}_{\mathcal{P}(\hat{x})}\ll Leb_{\,W_{\hat{x}}}
	\;\text{for $\hat{\mu}$ a.e. $\hat{x}\in M^{f}$},
$
where $(\hat{\mu}_{\mathcal{P}(\hat{x})})_{\hat{x}\in M^{f}} $
is a disintegration of $\hat{\mu}$ with respect to $\mathcal{P}$.
\end{definition}

\section{Existence and Uniqueness of SRB measures \label{sec:existence_uniqueness}}

The present section is devoted to the proof of Theorem~\ref{thm:TEO1}. In Subsections \ref{subsec:diskgeometry} and
\ref{sec:expandingstructure} we describe the geometry of the iterates of submanifolds that are tangent to the non-nonuniformly expanding cone field and construct families of disks with non-uniform hyperbolicity.
Afterwards, the core of the proof is, after building a (non-invariant) hyperbolic measure $\nu$ on $M$, to
construct a suitable lift $\hat \nu$ of $\nu$ to the natural extension and prove its absolute continuity along 
unstable disks (see Subsections \ref{sec:hatnuconstruction} and \ref{sec:srbproperty}). Finally, in 
Subsection \ref{sec:ergodicityfinitiness} we use these lifted measures prove the existence and finiteness of SRB measures
on the partially hyperbolic attractor.

\subsection{Geometry of disks which are tangent to $C_{a}$} \label{subsec:diskgeometry}

In this subsection we collect a result on the regularity disks tangent to cone fields from \cite{ABV00}. 
Consider $D$ a disk in $M$ which is tangent to the cone field $C_{a}$ and
assume that $D$ satisfies $Leb_{D}(H)>0$, where $H$
is the subset given by the hypothesis (H3). For $\vep>0$ and
$x\in M$ denote by
$
T_{x}M(\vep):=\left\{ v\in T_{x}M:\ \|v\|\leq\vep\right\} 
$
the ball of radius $\vep$-centered on $0\in T_{x}M$. Fix $\delta_{0}>0$
such that the exponential map restricted to $T_{x}M(\delta_{0})$
is a diffeomorphism onto its image for every $x\in M$, and set $B(x,\vep):=\exp_{x}(T_{x}M(\vep))$
for every $0<\vep\leq\delta_{0}$.
As $D$ is tangent to the cone field there exists $0<\delta\leq\delta_{0}$ such that if $y\in B(x,\delta)\cap D$
then there exists a unique linear map $A_{x}(y):T_{x}D\rightarrow E_{x}^{s}$
whose graph is parallel to $T_{y}D$. Moreover, by compactness of $M$,
$\delta$ can be taken independent of $x$. Thus one can express the H\"older variation
of the tangent bundle in these local coordinates as follows:
given $C>0$ and $0<\alpha<1$ we say that the tangent bundle $TD$
of $D$ is \emph{$(C,\alpha)$-H\"older} if
$
\|A_{x}(y)\|\leq C\cdot dist_{D}(x,y)^{\alpha},
$
for every $y\in B(x,\delta_{0})\cap D$ and $x\in D$. 
Moreover, the $\alpha$-H\"older constant of the
tangent bundle $TD$ is given by the number:
$
\kappa(TD,\alpha):=\inf\left\{ C>0:\ TD\mbox{ is }(C,\alpha)\mbox{-H\"older}\right\}.
$
Since our assumptions imply that disks tangent to the cone field are preserved under positive iteration by the 
dynamics and assumption (H4) guarantees domination, the following control of the distortion along disks tangent to the cone field
is a consequence of the arguments of \cite[Proposition 2.2 and Corollary 2.4]{ABV00}:

\begin{proposition}\label{prop:iteratesoftgbundleholder}
There are $\alpha\in(0,1)$, $\beta\in(0,1)$
and $C_{0}>0$ such that if $D$ is a submanifold of $M$which is
tangent to the cone field $C_{a}$ and if $f(D)$ is also
a submanifold of $M$ then $\kappa(Tf(N),\alpha)\leq\beta\cdot\kappa(TN,\alpha)+C_{0}$.
As a consequence, there is $L_{1}>0$ such that if $D$ is a submanifold of $M$ which
is tangent to $C_{a}$ and if $k\in\mathbb{N}$ is such that $f^{j}(D)$
is a submanifold of $M$ for every $1\leq j\leq k$ then the map
$$
\begin{array}{rccc}
J_k: & f^k(D) & \rightarrow & \mathbb{R}\\
     & x      & \mapsto     & \log\left| \det( Df(x)|_{T_xf^k(D)})\right|,
\end{array}
$$is $(L_{1},\alpha)$-H\"older continuous. Moreover, the constant
$L_{1}$ depends only on $f$.
\end{proposition}

\subsection{Non-uniform expansion along disks tangent to the cone-field} \label{sec:expandingstructure}

Hypothesis (H3) implies that there exists a disk $D\subset U$ tangent to $C_{a}$
and a subset $H\subset D$ 
with $Leb_{D}(H)>0$ satisfying
that:
\begin{equation}
\limsup_{n\to\infty}\frac{1}{n}\sum_{j=0}^{n-1}\log\|(Df(f^{j}(x))|_{C_{a}(f^{j}(x))})^{-1}\|\leq-2c<0,\label{eq:limsupnegative}
\end{equation}
for every $x\in H$. A direct computation (which we leave as a simple exercise to the reader) 
shows that Pliss' lemma, obtained independently by Liao, (see e.g. \cite[Lema IV.11.8]{Mane87}) together with \eqref{eq:limsupnegative}  
is enough to deduce the following:

\begin{lemma}
For all $x\in H$ and $n$ sufficiently large there is $\theta>0$
(that depends only on $f$ and $c$) and a sequence $1\leq n_{1}(x)<\cdots<n_{l}(x)\leq n$
of $c$-cone-hyperbolic times for $x$ with $l\geq\theta n.$\label{lem:densposth}
\end{lemma}

Since $M$ is compact and $f$ is a non-singular endomorphism, there is
$\delta_{1}>0$ such that all inverse branches are well defined on 
the ball $B(x,\delta_{1})$ for every $x\in M$. More precisely, given
$x\in M$ and $y\in f^{-1}(\left\{ x\right\} )$ there is a well defined 
diffeomorphism $f_{y}^{-1}:B(x,\delta_{1})\rightarrow V_{y}$
onto an open neighborhood $V_{y}$ of $y$ satisfying
$f\circ f_{y}^{-1}=id_{B(x,\delta_{1})}$. Reducing $\delta_{1}$ if necessary, we may assume 
that $V_{z}\cap V_{y}=\emptyset$ for every distinct $y,z\in f^{-1}(\left\{ x\right\} )$.
Actually, we have the following uniform continuity of the inverse branches:

\begin{lemma} \label{lem:continuityinversebranche}
For any $\varepsilon>0$ there exists $0<\delta=\delta(\varepsilon)<\delta_1$
such that $f_{y}^{-1}(B(x,\delta))\subset B(y,\varepsilon)$
for every $x\in M$ and $y\in f^{-1}(\left\{ x\right\} )$.
\end{lemma}

\begin{proof}
Fix $\varepsilon>0$. Given $x\in M$ and $y\in f^{-1}(\left\{ x\right\} )$,
there is $0<\delta_{y}<\delta_{1}$ such that $dist_{M}(f_{y}^{-1}(z),y)<\varepsilon/2$ whenever 
$dist_{M}(z,x)<\delta_{y}$.
As $M$ is a compact connected manifold then 
$\delta'(x):=\min\left\{ \delta_{y}:y\in f^{-1}(\left\{ x\right\} )\right\} >0$ for every $x\in M$.
By compactness, the open covering $( B(x,\delta'(x)/8))_{x\in M}$
of $M$ admits a finite subcover $(B(x_{j},\delta'(x_{j})/8))_{1\leq j\leq m}$.

We claim that $\delta:=\frac12\min\left\{ \delta'(x_{j}):1\leq j\leq m\right\}>0$ satisfies the requirements
of the lemma. Indeed, take $x\in M$, $y\in f^{-1}(\left\{ x\right\} )$
and $z\in B(x,\delta)$, and let $1\leq j\leq m$ be such that $z\in B(x_{j},\delta'(x_{j})/8)$.
Hence $dist_{M}(x,x_{j}) <\delta'(x_{j})$. 
Moreover, taking $u=f_{y}^{-1}(x_{j}) \in f^{-1}(\{x_j\})$ the maps 
$f_{y}^{-1}$ and $f_{u}^{-1}$ coincide at $B(x,\delta_{1})\cap B(x_{j},\delta_{1})$.
Since $z$ and $x$ belong to this intersection then
$
dist_{M}(f_{y}^{-1}(z),y)=dist_{M}(f_{y}^{-1}(z),f_{y}^{-1}(x))=dist_{M}(f_{u}^{-1}(z),f_{u}^{-1}(x))<\varepsilon,
$
which proves the lemma. 
\end{proof}

\begin{lemma}\label{lem:contofderivative}
If $c>0$ is given by \eqref{eq:limsupnegative} there exists $\delta_{2}>0$ (depending on $c$) so that for every $x\in\Lambda$, every $y\in B(x,\delta_{2})\cap\Lambda$ and every $z\in f^{-1}(\left\{ x\right\} )$ it hold that: 
\begin{equation}\label{eq:contcone}
\|Df(f_{z}^{-1}(y))^{-1}v\|\leq e^{\frac{c}{2}}\|(Df(f_{z}^{-1}(x))|_{C_{a}(f_{z}^{-1}(x))})^{-1} \|\|v\|
	\quad\text{for every $v\in C_{a}(y).$}
\end{equation} 
\end{lemma}

\begin{proof}
The map $\Lambda\ni x\mapsto\|(Df(x)|_{C_{a}(x)})^{-1}\|$
is continuous, hence uniformly continuous. 
So, given $\eta>0$ there is $\xi>0$ such
that if $dist_{M}(x,y)<\xi$ then 
$
|\|(Df(x)|_{C_{a}(x)})^{-1}\|-\|(Df(y)|_{C_{a}(y)})^{-1}\||<\eta.
$
Moreover, as $f$ is a non-singular endomorphism and $\overline{U}$ is compact
then $D:=\min_{x\in\overline{U}}\|(Df(x)|_{C_{a}(x)})^{-1}\|>0$. 
Take $\eta=(e^{\frac{c}{2}}-1)D$, let $\xi=\xi(\eta.c)>0$ be given by the latter uniform continuity
and let $\delta_{2}=\delta(\xi)>0$ be the radius given by Lemma~\ref{lem:continuityinversebranche}.
Thus, if $x,y\in\Lambda$ are such that $dist_{M}(x,y)<\delta_{2}$ and $z\in f^{-1}(\{x\})$ then 
$dist_{M}(x_{1},y_{1})<\xi$, where $x_{1}:=f_{z}^{-1}(x)=z$ and $y_{1}:=f_{z}^{-1}(y)$.
Therefore
$
\left|\|(Df(x_{1})|_{C_{a}(x_{1})})^{-1}\|-\|(Df(y_{1})|_{C_{a}(y_{1})})^{-1}\|\right|
<(e^{\frac{c}{2}}-1)D<(e^{\frac{c}{2}}-1) \|(Df(x_{1})|_{C_{a}(x_{1})})^{-1}\|
$
and so $\|(Df(y_{1})|_{C_{a}(y_{1})})^{-1}\| < e^{\frac{c}{2}}\|(Df(x_{1})|_{C_{a}(x_{1})})^{-1}\|.$
Then, 
$
\|Df(y_{1})^{-1}\cdot v\| 
  \le e^{\frac{c}{2}}\|(Df(x_{1})|_{C_{a}(x_{1})})^{-1}\|\cdot\|v\|
$
or, equivalently, $\|Df(f_{z}^{-1}(y))^{-1}v\|\leq e^{\frac{c}{2}}\|(Df(f_{z}^{-1}(x))|_{C_{a}(f_{z}^{-1}(x))})^{-1} \|\|v\|$, for every 
$v\in C_{a}(f(y_{1}))=C_{a}(y)$.
This proves the lemma.
\end{proof}

For a submanifold $D\subset M$ we will denote by $B_{D}(\cdot,\varepsilon)$
the ball of radius $\varepsilon>0$ in $D$ with the distance
induced by the Riemannian metric on $D$. 
Using the continuity of the parallel transport 
and the H\"older control of
the variation of the tangent bundle of disks tangent to $C_{a}$ (recall Proposition~\ref{prop:iteratesoftgbundleholder}),
if $D$ is a $C^{1}$ disk embedded in $U$ tangent to $C_{a}$ then there exists $\xi>0$ such that $B_{D}(x,\xi)=exp_{x}(graf(\psi_{x}))$
for every $x\in D$, where $\psi_{x}: B(0,\xi)\subset T_{x}D \to E_{x}^{s}$ is a Lipschitz map and $Lip(\psi_{x})\leq\varepsilon_{0}$ for some
is $\varepsilon_{0}>0$ depending only on $a>0$.
For that reason we will say that
$\Delta(q,\delta)\subset D$ is a disk of radius $\delta>0$
tangent to the cone field if $\Delta(q,\delta)=\exp_{q}(graf(\psi))$
where the graph of $\psi: B(0,\delta)\subset T_{q}D \to E_{x}^{s}$ is contained in $C_{a}(q)$.
Note that, there is a constant $R>0$ (depending only on the cone field) such that 
$B_{D}(q,R\delta) \subset \Delta(q,\delta) \subset B_{D}(q,2R\delta)$.
In what follows we show that cone-hyperbolic times imply a growth to a large scale 
along some subdisk of $D$. 
If $D$ is a submanifold of $M$ then we denote by $dist_{D}$ the
Riemannian distance on $D$ 
then $dist_{M}\leq dist_{D}$.

\begin{proposition}\label{prop:contraidist}
Let $D$ be a $C^{1}$ disk embedded in $U$ that is tangent to $C_{a}$. There is $\delta>0$ such that if $n$ is a 
$c$-cone-hyperbolic time for $x\in D$ and
$d(x,\partial D)>8\delta$ 
then there is an open neighborhood $D(x,n,8\delta)$ of $x$ in $D$ diffeomorphic 
to a $C^1$-disk $\Delta(f^{n}(x),8\delta)$ by $f^{n}$. Moreover 
$$
dist_{f^{n-k}(D(x,n,8\delta))}(f^{n-k}(x),f^{n-k}(y))\leq e^{-\frac{c}{2}k}dist_{\Delta(f^{n}(x),8\delta)}(f^{n}(x),f^{n}(y))
$$
for every $y\in D(x,n,8\delta)$ and $1\leq k\leq n.$
\end{proposition}

\begin{proof}
The first part in the proof is to assure that there exists a uniform radius $\delta$ so that 
the image $f(\tilde D)$ of any disk $\tilde D$ of radius $8\delta$ tangent to the cone field is a 
disk tangent to the cone field with inner diameter at least $8\delta$, and the second part of 
the proof is to assure the backward contraction property.
The first part of the argument adapts the classical ideas that the existence of the invariant cone field and domination assumptions (H3) and (H4) imply 
that the classical graph transform map is well defined(see e.g. \cite[Appendix III]{Sh1987} or \cite[Section~7]{BP07}).
Although the subbundle $F$ is not necessarily $Df$-invariant, 
we can use the tangent spaces to $f^j(D)$ (which are contained in the cone field) as reference space to define the graph transform. 
We will make this precise in what follows.

By compactness of $M$, there exist $\delta_{0},\delta_{1}>0$ such that $exp_{y}:B(0,\delta_{0})\subset T_{y}M\rightarrow exp_{y}(B(0,\delta_{0}))$
and $f|_{B(y,\delta_{1})}:B(y,\delta_{1})\rightarrow f(B(y,\delta_{1}))$
are diffeomorphisms for every $y\in M$. Moreover, by uniform continuity
of $f$, there is $\delta_{3}>0$ such that if $dist_{M}(y,z)<\delta_{3}$
then $dist_{M}(f(y),f(z))<\delta_{1}$.
Fix $\xi>0$ so that if  $\tilde D$ is a disk tangent to the cone field then $B_{\tilde D}(x,\xi)=exp_{x}(graf(\psi_{x}))$, $\psi_{x}: B(0,\xi)\subset T_{x}\tilde D 
\to E_{x}^{s}$ and set $8\delta:=\frac{1}{2}\min\left\{ \xi,\delta_{0},\delta_{1},\delta_{3}\right\}$.
By the choice of $\delta$,  the map
$
F_{x}:=\exp_{f(x)}^{-1}\circ f\circ\exp_{x}:B(0,8\delta)\subset T_{x}M\rightarrow T_{f(x)}M
$
is a diffeomorphism onto its image for every $x\in M$.

Fix a $C^{1}$ disk $D$ tangent to $C_{a}$ and $x\in D$ such that $d(x,\partial D)>8\delta$.
Then we claim that if $B_{D}(x,8\delta):=\exp_{x}(graf(\psi_{x}))$ with
$\psi_{x}:T_{x}D(8\delta)\subset T_{x}D\rightarrow E_{x}^{s}$ Lipschitz then the image $F_{x}(graf(\psi_{x}))$
is the graph of a Lipschitz map from an open disk in $T_{f(x)}f(D)$ to $E_{f(x)}^{s}$ whose domain contains a ball
of radius $8\delta$ around $0 \in T_{f(x)} f(D)$.
The choice of $8\delta<\delta_1$ guarantees that
$f(B_{D}(x,8\delta))$ is a submanifold of $M$ tangent to $C_a$.

\medskip
\noindent {\bf Claim:} \emph{
$F_{x}(graf(\psi_{x}))=graf(\psi_{f(x)})$ for some Lipschitz map $\psi_{f(x)}:T_{f(x)}f(D)(8\delta)\subset T_{f(x)}f(D)\rightarrow E_{f(x)}^{s}$.}

\begin{proof}[Proof of the claim]
Write
$
F_{x}(u+v)=Df(x)(u+v)+r_{x}(u+v)+t_{x}(u+v)
$
where the error terms $r_{x}:T_{x}D(8\delta)\times E_{x}^{s}(8\delta)\rightarrow T_{f(x)}f(D)$
and $t_{x}:T_{x}D(8\delta)\times E_{x}^{s}(8\delta)\rightarrow E_{f(x)}^{s}$
are $C^{1+\alpha}$ maps. 
Hypothesis (H3)-(H4) assure that the angle between $T_{x}D$ and $E^s_x$ is bounded away from zero. Therefore, the norm of the projection 
$p_{x}:T_{f(x)}M \to T_{f(x)}f(D)$ parallel to $E^s_{f(x)}$ is uniformly bounded away from zero.
Reduce $\delta>0$ if necessary so that for every $x\in M$ the error terms $r_{x}$ and $t_{x}$ are uniformly small (see e.g.
\cite[Section 7.1]{BP07}) so that 
$Lip(r_{x})\leq Lip\left[(p_{x}\circ Df(x)\circ(id,\psi_{x}))^{-1}\right]^{-1}$.
Now, using the invariance of $E^{s}$ and that $Df(x)\cdot T_{x}D=T_{f(x)}f(D)$, we  get
$
Lip [ p_{x}\circ F_{x}\circ(id,\psi_{x})-p_{x}\circ Df(x)\circ(id,\psi_{x}) ]\leq Lip(r_{x}) 
	\le Lip\left[(p_{x}\circ Df(x)\circ(id,\psi_{x}))^{-1}\right]^{-1}.
$
Theorem I.2 in \cite{Sh1987} assures that the Lipschitz map $p_{x}\circ F_{x}\circ(id,\psi_{x})$
is a homeomorphism onto its image, which guarantees that 
 $F_x(graf(\psi_{x}))$ is a graph of a Lipschitz map $\psi_{f(x)} : T_{f(x)}f(D)(8\delta)\subset T_{f(x)}f(D) \to E_{f(x)}^{s}$.
\end{proof}

As the previous estimates are uniform (independ of the disk $D$) then $f^j(D(x,n,8\delta))=F^j_x(graf(\psi^j_{x}))$
for every $0\le j \le n$, where $F_{x}^{j}=exp_{f^{j}(x)}^{-1}\circ f^{j}\circ exp_{x}$ and $D(x,n,8\delta)$ denotes the set of
points $y\in D$ so that $d(f^i(x), f^i(y))<8\delta$ for every $0\le i \le n$.
Finally, the proof of the backward contraction property and that $f^n(D(x,n,8\delta))$ is a disk of radius $8\delta$ tangent to $C$
is similar to~\cite[ Lemma 2.7]{ABV00}, the difference being that 
hyperbolic times are considered along the cone-field (these guarantee the necessary 
control of the curvature of the disks) and that backward contraction is given along the suitable
concatenation of inverse branches.
\end{proof}

We observe the constant $\delta>0$ in Proposition~\ref{prop:contraidist} is independent of $x\in M$.
The disks $D(x,n,8\delta)\subset D$ are called \emph{hyperbolic pre-disks} and their 
images $\Delta(f^{n}(x),8\delta)$ are called \emph{$n$-hyperbolic disks}.
We need the following bounded distortion property for the Jacobian of $f$ along disks tangent to
the invariant cone field.

\begin{proposition}\cite[Proposition 2.8]{ABV00}
Suppose $D$ is a $C^{1}$ disk embedded in $U$, $x\in D$ and $n$
a $c$-cone hyperbolic time for $x$.There is $C_{1}>0$ such that
\[
C_{1}^{-1}\leq\dfrac{|det\ Df^{n}(y)|_{T_{y}D(x,n,8\delta)}|}{|det\ Df^{n}(z)|_{T_{z}D(x,n,8\delta)}|}\leq C_{1}
\]
for every $y,z\in D(x,n,8\delta)$, where $D(x,n,8\delta)$ is the hyperbolic
pre-disk associated to $x$.\label{prop:distorcionhyptime}
\end{proposition}

Fix $\delta>0$ given by Proposition \ref{prop:contraidist} and
$H=H(c) \subset U$ be the set of points with infinitely many $c$-cone hyperbolic times
and such that $Leb_{D}(H)>0$, guaranteed by hypothesis (H3). Reducing $\delta$, if necessary,
we can assume that 
$
B:=B(D,c,8\delta):=\left\{ x\in D\cap H:\ dist_{D}(x,\partial D)\geq8\delta\right\} 
$
satisfies $Leb_{D}(B)>0$. For each $n\in\mathbb{N}$ define
$H_{n}:=\left\{ x\in B:\ n\mbox{ is a }c\mbox{-cone hyperbolic time for }x\right\} .$

\begin{proposition} \cite[Proposition 3.3]{ABV00}\label{prop:diskformeasure}
There exists $\tau>0$ and for each $n\in\mathbb{N}$
there is a finite set $H_{n}^{*} \subset H_{n}$ such that the hyperbolic
pre-disks $(D(x,n,\delta))_{x\in H_{n}^{*}}$ are contained in $D$, pairwise disjoint and their union 
$
\mathscr{D}_{n}:=\bigcup_{x\in H_{n}^{*}}D(x,n,\delta)
$
 satisfies $\Leb_{D}(\mathscr{D}_{n}\cap H_{n})\geq\tau Leb_{D}(H_{n})$.
\end{proposition}

\subsection{Construction of (non-invariant) hyperbolic measures\label{sec:nuconstruction}}

Let $\mathcal{B}(\Lambda)$ be the Borel
$\sigma$-algebra in $\Lambda$. We will say that a measure $\nu:(\Lambda,\mathcal{B}(\Lambda))\rightarrow\left[0,+\infty\right]$
is a subprobability if $\nu(\Lambda)\leq1$. We
denote by $\mathcal{S}(\Lambda)$ the set of all Borelian
subprobabilities which, by the Schauder's Theorem, is a compact space (in the weak{*} topology). 
Define, for each $n\in\mathbb{N}$,
\[
\mu_{n}:=\frac{1}{n}\sum_{j=0}^{n-1}f_{*}^{j}Leb_{D},
\]
where 
$Leb_{D}$ is the 
Lebesgue measure in $D$ induced by its Riemannian structure. It is easy to check that
any accumulation point $\mu$ of $(\mu_{n})_{n}$ is an $f$-invariant measure.
Now define for each $n\in\mathbb{N}$
\begin{equation}
\nu_{n}:=\frac{1}{n}\sum_{j=0}^{n-1}f_{*}^{j}Leb_{\mathscr{D}_{j}} \label{eq:hyp_subprobability_n}
\end{equation}
where $Leb_{\mathscr{D}_{j}}:=\sum_{x\in H_{j}^{*}}Leb_{D(x,j,\delta)}$,
and write $\mu_{n}=\nu_{n}+\eta_{n}$.
If $\mu$ is an accumulation point of $(\mu_{n})_{n\in\mathbb{N}}$
in the weak{*} topology then $\mu=\nu+\eta$, where $\nu$ and $\eta$
are subprobabilities obtained as an accumulation points of $(\nu_{n})_{n\in\mathbb{N}}$
and $(\eta_{n})_{n\in\mathbb{N}}$, respectively. The
following result guarantees such a measure $\nu$ has positive mass.

\begin{lemma}
There is $\alpha>0$ such that $\nu_{n}\mbox{\ensuremath{(H)\geq\alpha}}$
for every large $n$. Consequently, if $\nu$ is an accumulation
point of $(\nu_{n})_{n\in\mathbb{N}}$ then $\nu(H)\geq\alpha$.\label{prop:nupositiva}
\end{lemma}

\begin{proof}
The proof is identical to \cite[Proposition 3.5]{ABV00}.
\end{proof}

We now describe the support of the measure $\nu$.
By construction, the measure $\nu_{n}$ is supported on the union of disks 
$\cup_{j=0}^{n-1}\Delta_{j}$, with $\Delta_{j}:=\cup_{x\in H_{j}^{*}} \Delta(f^{j}x,\delta)$, which \emph{may have} self intersections.
In particular any accumulation point  $\nu$ of $(\nu_{n})_{n\in\mathbb{N}}$ is such that 
$supp(\nu)\subseteq\cap_{n\in\mathbb{N}}\overline{\cup_{j=0}^{n-1}\Delta_{j}}$. 
The next result asserts that $supp(\nu)$ is contained in a union of unstable manifolds of uniform size.

\begin{proposition}
\label{lem:suppdisk} Given $y\in\supp(\nu)$ exist $\hat{x}=(x_{-n})_{n\in\mathbb{N}}\in M^{f}$
such that $y$ belongs to a disk $\Delta(\hat{x})$ of
radius $\delta>0$ accumulated by $n_{j}$-hyperbolic disks $\Delta(f^{n_{j}}(x),\delta)$
with $j\to\infty$, and satisfying that
\begin{enumerate}
\item the inverse branch $f_{x_{-n}}^{-n}$ is well defined on $\Delta(\hat{x})$
for every $n\geq0$;
\item if for each $y\in\Delta(\hat{x})$ one takes $y_{-n}:=f_{x_{-n}}^{-n}(y)$ then
$
dist_{M}(y_{-n},x_{-n})\leq e^{-\frac{c}{2}\cdot n}\delta\mbox{ for every }n\geq0.
$
\end{enumerate}
\end{proposition}

\begin{proof}
The idea is that any point in $\cap_{n\in\mathbb{N}}\overline{\cup_{j=0}^{n-1}\Delta_{j}}$
is accumulated by $j$-hyperbolic disks with $j$ going to $\infty$ and that, using the contraction of 
$j$-hyperbolic disk along past $j$ iterates, one can choose a full past trajectory (pre-orbit) along which
we observe backward contraction at all iterates.

Given $y\in\supp(\nu)$, there is a sequence $(y_{j})_{j\in\mathbb{N}}$
with $y_{j}\in\Delta(f^{n_{j}}(x_{n_{j}}),\delta)$
and $x_{n_{j}}\in H_{n_{j}}^{*}$ such that $\lim_{j\to\infty}y_{j}=y$.
Choosing a subsequence if necessary, we can assume that $\lim_{j\to\infty}f^{n_{j}}(x_{n_{j}})=x$.
As one can write $\Delta(f^{n_{j}}(x_{n_{j}}),\delta)$
as the image of the exponential map of the graph of a map $\psi_{f^{n_{j}}(x_{n_{j}})}:T_{f^{n_{j}}(x_{n_{j}})}D(\delta)\rightarrow E_{f^{n_{j}}(x_{j})}^{s}$ with $Lip(\psi_{f^{n_{j}}(x_{n_{j}})})<\varepsilon_{0}$.
Moreover, by parallel transport, we can identify the disk $\Delta(f^{n_{j}}(x_{n_{j}}),\delta)$
with the graph of a map $g_{n_{j}}:T_{x}(f^{n_{j}}(D))(\delta)\rightarrow E_{x}^{s}$ in $T_xM$.
Since $Lip$$(\psi_{f^{n_{j}}(x_{n_{j}})})<\varepsilon_{0}$
for any $j$ and the parallel transport is a linear isomorphism 
then $(g_{n_{j}})_{j\in\mathbb{N}}$ is a sequence of uniformly bounded and equicontinuous maps which, 
by Arzel\`a-Ascoli's theorem, admits a convergent subsequence $\lim_{k\to\infty}g_{n_{j_{k}}}=g$.
In particular, $\Delta(x)=exp_{x}(graf(g))$ is a $C^{1}$ disk of radius $\delta>0$ tangent to the cone field and obtained as a limit
of the family of disks $(\Delta(f^{n_{j}}(x_{n_{j}}),\delta)) _{j\in\mathbb{N}}$.
For notational simplicity assume, without loss of generality, that
\[
\Delta(x):=\lim_{j\to\infty}\Delta(f^{j}(x_{j}),\delta)
	\qquad\text{and}\qquad
	\lim_{j\to\infty}f^{j}(x_{j})=x.
\]
Since $\Delta(f^{j}(x_{j}),\delta)\subset B(f^{j}(x_{j}),\delta)$ then there exists $j_{0}\in\mathbb{N}$ such that 
$\Delta(f^{j}(x_{j}),\delta)\subset B(x,\delta_{1})$ for every $j\geq j_{0}$. 
As every point has $d=\text{deg}(f)\ge 1$ preimages, 
there exists $x_{-1}\in f^{-1}(\left\{ x\right\} )$ such that $f_{x_{-1}}^{-1}(B(x,\delta))\supset f^{j_{k}-1}(D(x_{j_{k}},j_{k},\delta))$
for some subsequence $(j_{k})_{k\in\mathbb{N}}$, with
$j_{k}\to\infty$ when $k\to\infty$ (see Figure \ref{fig:choose_inverse_branche}
below).
\begin{figure}[htb]
\centering
\includegraphics[scale=.7]{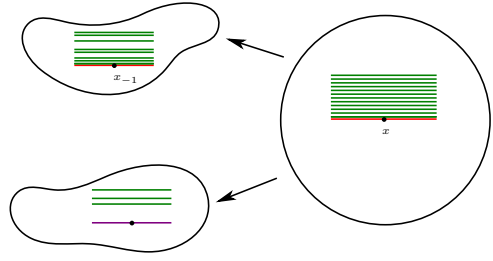}
\caption{Construction of local unstable manifolds through selection of inverse branches\label{fig:choose_inverse_branche}}
\end{figure}
Moreover,
\begin{align*}
dist_{M}\,(f_{x_{-1}}^{-1}(f^{j_{k}}(y_{j_{k}})),f_{x_{-1}}^{-1}(f^{j_{k}}(x_{j_{k}})))
	& \le  dist_{f^{j_{k}-1}(D(x_{j_{k}},j_{k},\delta))}(f^{j_{k}-1}(y_{j_{k}}),f^{j_{k}-1}(x_{j_{k}})) \\
	& \leq e^{-\frac{c}{2}}dist_{\Delta(f^{j_{k}}(x_{j_{k}}),\delta)}(f^{j_{k}}(y_{j_{k}}),f^{j_{k}}(x_{j_{k}})),
\end{align*}
for every $y_{j_{k}}\in D(x_{j_{k}},j_{k},\delta)$ and for
all $k\in\mathbb{N}$ (cf. Proposition \ref{prop:contraidist})
So, if $y\in\Delta(x)$ and $y=\lim_{j\to\infty}y_{j_{k}}$
with $y_{j_{k}}\in\Delta(f^{j_{k}}(x_{j_{k}}),\delta)$
then $y_{-1}:=f_{x_{-1}}^{-1}(y)$ satisfies $y_{-1}=\lim_{k\to\infty}f_{x_{-1}}^{-1}(y_{j_{k}})$
and
$dist_{M}(y_{-1},x_{-1})\leq e^{-\frac{c}{2}}\delta.$
By considering a subsequence of $(x_{j_{k}})_{k\in\mathbb{N}}$
we may assume $\lim_{k\to\infty}f^{k}(x_{k})=x$ and $f^{k-1}(D(x_{k},k,\delta))\subset B(x_{-1},\delta)$
for every large  $k\ge1$. Then, there exists $x_{-2}\in f^{-1}(\left\{ x_{-1}\right\} )\subset f^{-2}(\left\{ x\right\} )$
such that $f_{x_{-2}}^{-1}(B(x_{-1},\delta))\supset f^{k_{s}-2}(D(x_{k_{s}},k_{s},\delta))$
for some subsequence $k_{s}$ of $k$, where $k_{s}\to\infty$ when
$s\to\infty$. The previous argument shows that
$dist_{M}(y_{-2},x_{-2})\leq e^{-\frac{c}{2}\cdot2}\delta<\delta.$
Proceeding recursively, we conclude that there exists $\hat{x}\in\pi^{-1}(\left\{ x\right\} )$
such that the inverse branches $f_{x_{-n}}^{-n}$ are well defined on $\Delta(x)$
and that, for every $y\in\Delta(x)$, 
$
dist_{M}(y_{-n},x_{-n})\leq e^{-\frac{c}{2}\cdot n}\delta
\; \text{for every $n\in\mathbb{N}$},
$
where $y_{-n}=f_{x_{-n}}^{-n}(y)$.
The disk $\Delta(\hat{x}):=\Delta(x)$ is as desired.
\end{proof}

The previous result says that each point of the support of $\nu$
is contained in a disk $\Delta(\hat{x})$, where $\hat{x} \in M^f$
parameterizes the inverse branches along which we observe backward contraction 
at the disk.
Let us denote by $\hat{H}_{\infty}$ the set of all elements of $M^{f}$
obtained in this way. Although $\nu$ is not necessarily invariant it is supported on true
unstable disks. 

\subsection{Lift of non necessarily invariant measures to $M^{f}$\label{sec:hatnuconstruction}}

In this section we will construct a lift
for $\nu$, that is, a measure $\hat{\nu}$ on $M^f$ such that $\pi_{*}\hat{\nu}=\nu$.
Although the uniqueness of the lift holds only 
for invariant measures (cf. \cite[Proposition I.3.1]{QXZ09}), our method
of the construction of this measure will allow us to verify the absolute
continuity of $\hat{\nu}$ along unstable sets.
As observed in Section \ref{subsec:naturalextension}, in general the natural extension 
is not a manifold. Nevertheless, in the case of non-singular endomorphisms it can be locally
identified by suitable charts. 
We will use the following:

\begin{proposition}
\label{prop:uniforminversebranches}Let $M$ be a compact connected
Riemannian manifold and let $f:M\rightarrow M$ be a non-singular endomorphism.
There exists $\rho>0$ such that for every $x\in M$ and $n\in\mathbb{N}$
there is a family of pairwise disjoint open sets $\mathcal{J}_{n}(x):=\left\{ V_{\lambda}:\lambda\in I_{n}(x)\right\} $
satisfying:
\begin{enumerate}
\item $f^{-n}(B(x,\rho)):=\bigcup_{\lambda\in I_{n}(x)}V_{\lambda}$
\item $f^{n}|_{V_{\lambda}}:V_{\lambda}\rightarrow B(x,\rho)$
is a diffeomorphism, for every $\lambda\in I_{n}(x)$.
\end{enumerate}
\end{proposition}
\begin{proof}
\cite[Theorem 6.5.1]{AH94}.
\end{proof}

Recall that $\#f^{-1}(x)=d$ for every $x\in M$ and so
the set $I_{n}(x)$ given by the previous proposition satisfies that $\#I_{n}(x)=\#f^{-n}(\left\{ x\right\} )=d^{n}$
for every $n\geq1$ (hence it is independent of $x$). 
Moreover, given $V_{\lambda}\in\mathcal{J}_{n}(x)$ we have that $f^{j}(V_{\lambda})\in\mathcal{\mathcal{J}}_{n-k}(x)$,
for every $1\leq k\leq n-1$.
This motivates the following coding. 
We fix $I_{n}=I_{n}(x)=\left\{ 1,2,...,d\right\} ^{n}$
for every $n\in\mathbb{N}$ and establish an enumeration for $\mathcal{J}_{n}(x)$
that will allow us to identify the images of the elements of $\mathcal{J}_{n}(x)$
in $\mathcal{J}_{k}(x)$, for $1\leq k\leq n-1$. For $n=1$
write $\mathcal{J}_{1}(x):=\left\{ V_{1},...,V_{d}\right\} $.
For $n=2$ one can write $\mathcal{J}_{2}(x):=\left\{ V_{i_{2}i_{1}}:(i_{2},i_{1})\in I_{2}\right\} $
where $f(V_{i_{2}i_{1}})=V_{i_{1}}\in\mathcal{J}_{1}(x)$.
By induction, having defined
\[
\mathcal{J}_{n-1}(x):=\left\{ V_{i_{n-1}i_{n-2}...i_{1}}:(i_{n-1},...,i_{2},i_{1})\in I_{n-1}\right\} .
\]
one can write $\mathcal{J}_{n}(x):=\left\{ V_{i_{n} i_{n-1}...i_{1}}:(i_{n},...,i_{1})\in I_{n}\right\} $
where each $V_{i_{n}i_{n-1}...i_{1}}$ is determined by $f(V_{i_{n}i_{n-1}...i_{1}})=V_{i_{n-1}...i_{1}}\in\mathcal{J}_{n-1}(x)$.
This establishes an enumeration for the pre-images of a ball of radius $\rho$.
Furthermore, recalling that $\pi:M^{f}\rightarrow M$ is the projection in the first coordinate, 
given 
$x\in M$ one can identify $\pi^{-1}(B(x,\rho))$ with a product of $B(x,\rho)$ by
a fixed Cantor set. Indeed, given $x\in M$ and $\Gamma:=\left\{ 1,...,d\right\} ^{\mathbb{N}}$, the map
\begin{equation}\label{lem:localstructureMf}
\begin{array}{rccc}
{\varphi}_x: & B(x,\rho)\times \Gamma &\rightarrow &\pi^{-1}(B(x,\rho))\subset M^f\\
           & (z,(i_{n})_{ n\in\mathbb{N} }) &\mapsto & ( (f^n|_{V_{i_n...i_1}})^{-1}(z))_{n\in\mathbb{N}}
\end{array}
\end{equation}
is a homeomorphism (see \cite[Theorem 6.5.1]{AH94})
and 
it satisfies $\pi\circ\varphi_{x}(z,c)=z$ for every $(z,c)\in B(x,\rho)\times\Gamma$. 
\medskip

We make use of the previous structure  
to lift the Lebesgue measure restricted to hyperbolic pre-disks $D(x,n,\delta)$ to $M^f$. 
Fix an arbitrary probability measure $\mathds{P}$ on $\Gamma$,
reduce $\delta$ (if necessary) so that $D(x,n,\delta)\subset B(x,\rho)$ for every $x\in H_{n}$ and $n\in\mathbb{N}$
and consider the measure
\begin{equation}
\hat{\nu}_{n}:=\frac{1}{n}\sum_{j=0}^{n-1}\hat{f}_{*}^{j}\hat{m}_{j},\label{eq:lift_of_nun}
\end{equation}
in $\pi^{-1}(D(x,j,\delta))\subset M^{f}$, where $\hat{m}_{j}:=\sum_{x\in H_{j}^{*}}\hat{m}_{j,x}$ and 
for each $j\in\mathbb{N}$ and $x\in H_{j}^{*}$ 
\begin{equation}
\hat{m}_{j,x}:=(\varphi_{x})_{*}\left[Leb_{D(x,j,\delta)}\times\mathds{P}\right].\label{eq:lift_leb_on_prehypdisk}
\end{equation} 
It is not hard to check that $\pi_{*}\hat{\nu}_{n}=\nu_{n}$ 
and that if $\lim_{k\to\infty}\hat{\nu}_{n_{k}}=\hat{\nu}$ then 
$\nu=\pi_{*}\hat{\nu}$ is an accumulation point of $(\nu_{n})_{n\in\mathbb{N}}$.
Moreover, any accumulation point of the sequence $(\nu_n)_n$ can be obtained 
as a limit of the push-forwards by $\pi$ of elements of the sequence $(\hat\nu_n)_n$.
Finally, using that 
$supp(\hat{\nu}_{n})\subset\bigcup_{j=0}^{n-1}\bigcup_{x\in H_{j}^{*}}\hat{f}^{j}(\pi^{-1}(D(x,j,\delta)))$
we conclude that 
$
\supp(\hat{\nu})\subset
	\bigcap_{n\in\mathbb{N}}\overline{\bigcup_{j=0}^{n-1}\bigcup_{x\in H_{j}^{*}}\hat{f}^{j}(\pi^{-1}(D(x,j,\delta)))}
$
for any accumulation point $\hat{\nu}$ of $(\hat\nu_n)_n$.
By Proposition \ref{lem:suppdisk}, given $y\in\supp(\nu)$ there
is a disk $\Delta(\hat{x})$ which is contractive along
the pre-orbit $\hat{x}\in M^{f}$. 
We proceed to prove $\hat{\nu}$ is supported on a union of unstable manifolds. 
For that, if $\Delta(\hat{x})$ ($\hat{x}\in\hat{H}_{\infty}$) is a disk given by Proposition \ref{lem:suppdisk} 
then there exists a sequence of points
$(z_{j_{k}})_{k\in\mathbb{N}}$ satisfying  $z_{j_{k}}\in H_{j_{k}}^{*}$
for every $k\in\mathbb{N}$ 
such that $\lim_{k\to\infty}\Delta(f^{j_{k}}(z_{j_{k}}),\delta)=\Delta(\hat{x})$.
Since $\lim_{k\to\infty}f^{j_{k}-n}(z_{j_{k}})=x_{-n}$ for every $n\in\mathbb{N}$ (where the limit is taken over $j_{k}\geq n$)
then it is natural to define:
\begin{align}
\hat{\Delta}(\hat{x}) & :=\left\{ \hat{y}\in M^{f}:\ \hat{y}=\lim_{s\to\infty}\hat{y}_{j_{k_{s}}},\text{ for some subsequence }(\hat y_{j_{k_{s}}})_{s\in\mathbb{N}}\right. \nonumber \\ 
	& \left.\text{where }\hat{y}_{j_{k_{s}}}\in\hat{f}^{j_{k_{s}}}(\pi^{-1}(D(z_{j_{k_{s}}},j_{k_{s}},\delta)))\text{ and }(z_{j_{k_{s}}})_{\!s\in\mathbb{N}}\!\subset(z_{j_{k}})_{\!k\in\mathbb{N}}\right\} .\label{eq:def_lift_deltax}
\end{align}
Roughly, the set $\hat{\Delta}(\hat{x})$ is formed by all points obtained as accumulation points of a sequence $(\hat y_{j_{k}})_{k\in\mathbb{N}}$ where $\hat y_{j_{k}}\in\hat{f}^{j_{k}}(\pi^{-1}(D(z_{j_{k}},j_{k},\delta)))$ for every $k\ge 1$.
As before the support of $\hat\nu$ is contained in the union of this family of sets. 
Moreover, we have the following:

\begin{lemma}
Given $\hat{y}\in supp(\hat{\nu})$, there exists $\hat{x}\in M^{f}$
and $\hat{\Delta}(\hat{x})\subset M^{f}$ (given by  \eqref{eq:def_lift_deltax}) 
such that $\hat{y}\in\hat{\Delta}(\hat{x})$ and $\pi|_{\hat{\Delta}(\hat{x})}:\hat{\Delta}(\hat{x})\rightarrow\Delta(\hat{x})$
is a continuous bijection. \label{prop:bijectiveprojectionondisks}
\end{lemma}

\begin{proof}
Let $\Delta(\hat{x})=\lim_{k\to\infty}\Delta(f^{j_{k}}(z_{j_{k}}),\delta)$ be given by Proposition \ref{lem:suppdisk}
and $\hat{\Delta}(\hat{x})$ be given by \eqref{eq:def_lift_deltax}.
The continuity of $\pi$ implies that $\pi(\hat{\Delta}(\hat{x}))\subset\Delta(\hat{x})$.
In fact, if $\hat{y}\in\hat{\Delta}(\hat{x})$ then $\hat{y}=\lim_{s\to\infty}\hat{y}_{j_{k_{s}}}$
where $\hat{y}_{j_{k_{s}}}\in\hat{f}^{j_{k_{s}}}(\pi^{-1}(D(z_{j_{k_{s}}},j_{k_{s}},\delta)))$.
So, $\pi(\hat{y})=\lim_{s\to\infty}\pi(\hat{y}_{j_{k_{s}}})$.
Moreover, as $\pi(\hat{y}_{j_{k_{s}}})\in\Delta(f^{j_{k_{s}}}(z_{j_{k_{s}}}),\delta)$
and $\Delta(\hat{x})=\lim_{k\to\infty}\Delta(f^{j_{k}}(z_{j_{k}}),\delta)$, we conclude  that
 $\pi(\hat{y})\in\Delta(\hat{x})$. 
The inclusion $\Delta(\hat{x})\subset\pi(\hat{\Delta}(\hat{x}))$ is analogous.
Indeed, if $y\in\Delta(\hat{x})$
satisfies that $y=\lim_{k\to\infty}y_{j_{k}}$ then, choosing $\hat{y}_{j_{k}}$
such that $\pi(\hat{y}_{j_{k}})=y_{j_{k}}$ and $\hat{y}_{j_{k}}\in\hat{f}^{j_{k}}(\pi^{-1}(D(z_{j_{k}},j_{k},\delta)))$, 
the compactness of $M$ assures that there exists some convergent subsequence $(\hat{y}_{j_{k_{s}}})_{s\in\mathbb{N}}$
to some point $\hat y$. 
Moreover $\hat{y}\in\hat{\Delta}(\hat{x})$ and, since $\pi$ is continuous,  $\pi(\hat{y})=y$.
We will show that $\pi\mid_{\hat{\Delta}(\hat{x})}$
is injective. We claim that if $\hat{y}\in\hat{\Delta}(\hat{x})$
then 
$
dist(y_{-n},x_{-n})\leq e^{-\frac{c}{2}n}\delta,\mbox{ for all }n\geq1.
$
In fact, given $n\geq1$, Proposition \ref{lem:suppdisk} implies that 
$f_{x_{-n}}^{-n}(\Delta(f^{j_{k}}(z_{j_{k}}),\delta))=f^{j_{k}-n}(D(z_{j_{k}},j_{k},\delta))$
for every $j_{k}\geq n$. Then $\hat{y}\in\hat{\Delta}(x)$.
Therefore
$y_{-n}=\lim_{s\to\infty}f_{x_{-n}}^{-n}(y_{j_{k_{s}}}),$
because $y_{-n}=\pi\circ\hat{f}^{-n}(\hat{y})$ and $f_{x_{-n}}^{-n}(f^{j_{k_{s}}}(z_{j_{k_{s}}}))=\pi\circ\hat{f}^{-n}(\hat{y}_{j_{k_{s}}})$. Moreover, by Proposition \ref{lem:suppdisk}, 
$
dist(y_{-n},x_{-n})\leq e^{-\frac{c}{2}n}\delta
$
for every $n\geq1$.
This guarantees that if $\pi(\hat{y})=\pi(\hat{u})$
then $\hat{y}=\hat{u}$ (because points in $\hat{\Delta}(\hat{x})$
are determined by their images by the inverse branches $f_{x_{-n}}^{-n}$, for
all $n\in\mathbb{N}$). Therefore, $\pi\mid_{\hat{\Delta}(\hat{x})}$
is injective and, consequently, $\pi\mid_{\hat{\Delta}(\hat{x})}:\hat{\Delta}(\hat{x})\rightarrow\Delta(\hat{x})$
is a bijection.
\end{proof}

Recall $\hat{H}_{\infty}$ is the set of points $\hat{x}\in M^{f}$ given by Lemma~\ref{prop:bijectiveprojectionondisks}
and that $\mbox{supp}(\hat{\nu})\subset\bigcup_{\hat{x}\in\hat{H}_{\infty}}\hat{\Delta}(\hat{x})$. For all points in the
latter set we will construct an invariant unstable subbundle, which will coincide with the unstable Oseledets subspace
associated to the SRB measure supported at these points (cf. Proposition~\ref{prop:existsrb}).

\begin{proposition}\label{prop:hatnuishyp}
Given $\hat{x}\in\hat{H}_{\infty}$ and $\hat{y}\in\hat{\Delta}(\hat{x})$
there is a $d_{u}$-dimensional subspace $\mathcal{E}_{\hat{y}}^{u}$
of $T_{\hat{y}}=T_{y_{0}}M$ so that 
\begin{equation}
\|(Df^{m}(y_{-m}))^{-1}\cdot v\|\leq e^{-\frac{c}{2}m}\|v\|
	\; \text{for every $v\in\mathcal{E}_{\hat{y}}^{u}$ and $m\in\mathbb{N}$.}
	\label{eq:contractivespace}
\end{equation}
\end{proposition}
\begin{proof}
Fix $\hat{y}\in\hat{\Delta}(\hat{x})$ and set
$
\mathcal{\mathcal{E}}_{\hat{y}}^{u}:=\bigcap_{n\in\mathbb{N}}Df^{n}(y_{-n}) ( C_{a}(y_{-n})).
$
We will show that $\mathcal{\mathcal{E}}_{\hat{y}}^{u}$ is a subspace
of $T_{y_{0}}M$ whose dimension is $d_{u}$ and satisfies 
\eqref{eq:contractivespace}. 
By Lemma~\ref{prop:bijectiveprojectionondisks}, we know that
$\hat{y}=\lim_{k\to\infty}\hat{y}_{n_{k}}$, for some subsequence
$(\hat{y}_{n_{k}})_{k\in\mathbb{N}}$ where $\hat{y}_{n_{k}}\in\hat{f}^{j}(\pi^{-1}(D(z_{n_{k}},n_{k},\delta)))$
and $z_{n_{k}}\in H_{n_{k}}^{*}$ for every $k\in\mathbb{N}$. Writing
$\hat{y}_{n_{k}}:=(y_{-n}^{k})_{n\in\mathbb{N}}$ and $\hat{y}=(y_{-n})_{n\in\mathbb{N}}$
we have that $\hat{y}=\lim_{k\to\infty}\hat{y}_{n_{k}}$ if, and only
if $\lim_{k\to\infty}y^k_{-j}=y_{-j}$ for every $j\in\mathbb{N}$.
We need the following: 

\medskip
\noindent {\bf Claim:}
\emph{ If $m\in\mathbb{N}$ and $w\in Df^{m}(y_{-m})( C_{a}(y_{-m}))$
then $\left\Vert (Df^{m}(y_{m}))^{-1}\cdot w\right\Vert \leq e^{-\frac{c}{2}m}\left\Vert w\right\Vert $.}

\begin{proof}[Proof of the claim]
In fact, if $k$ is large so that $n_{k}\geq m$, one can use that
$\hat{y}_{n_{k}}\in\hat{f}^{j}(\pi^{-1}(D(z_{n_{k}},n_{k},\delta)))$
to conclude that $y_{-j}^{k}\in f^{n_{k}-j}(D(z_{n_{k}},n_{k},\delta))$
for every $0\leq j\leq n_{k}$. In particular $y_{-j}^{k}\in f^{n_{k}-j}(D(z_{n_{k}},n_{k},\delta))$
for every $0\leq j\leq m$.
Proposition \ref{prop:contraidist} assures 
$
dist_{M}(y_{-j}^{k},f^{n_{k}-j}(z_{n_{k}}))\leq e^{-\frac{c}{2}j}\delta
$
 for every $0\leq j\leq m$. Using the continuity of the derivative,
as in Lemma \ref{lem:contofderivative}, we obtain that
\begin{equation}
\|Df(y_{-j-1}^{k})^{-1}\cdot u_{j}\|\leq e^{\frac{c}2}\|(Df(f^{n_{k}-j-1}(z_{n_{k}}))|_{C_{a}(f^{n_{k}-j-1}(z_{n_{k}}))})^{-1}\|\cdot\|u_{j}\|,\label{eq:inequality}
\end{equation}
for every $u_{j}\in C_{a}(y_{-j}^{k})$ and for every $0\leq j\leq m-1$.
Thus, if $w_{k}\in Df^{m}(y_{-m}^{k}) ( C_{a}(y_{-m}^{k}))$
then
\begin{align*}
\left\Vert (Df^{m}(y_{-m}^{k}))^{-1}\cdot w_{k}\right\Vert  
 	& \leq e^{\frac{c}2 m} \prod_{j=n_{k}-m}^{n_{k}-1}\left\Vert (Df(f^{j}(z_{n_{k}}))|_{C_{a}(f^{j}(z_{n_{k}}))})^{-1}\right\Vert \cdot\left\Vert w_{k}\right\Vert 
	\le e^{-\frac{c}2 m} \left\Vert w_{k}\right\Vert .
\end{align*}
Here we used that $(Df^{j}(y_{-j}^{k}))^{-1}\cdot w_{k}\in C_{a}(y_{-j})$
for every $0\leq j\leq m-1$, because the cone field is invariant and
$w_{k}\in Df^{m}(y_{-m}^{k}) ( C_{a}(y_{-m}^{k}))$, and that  $z_{n_{k}}$ has $n_{k}$ as $c$-cone-hyperbolic time.
Suppose that $w\in Df^{m}(y_{-m}) ( C_{a}(y_{-m}))$.
The continuity of the cone field allow us that we write $w$ as the
limit of a sequence $(w_{k})_{k\in\mathbb{N}}$ where $w_{k}\in Df^{m}(y_{-m}^{k}) (C_{a}(y_{-m}^{k}))$.
Since $f$ is $C^{2}$ and the cone field is continuous then
$
\|(Df^{m}(y_{-m}^{k}))^{-1}\cdot w_{k}\|\leq e^{-\frac{c}{2}m}\|w_{k}\|
$
for every $k$. So, passing to the limit when $k\to\infty$ we have that
$
\|(Df^{m}(y_{-m}))^{-1}\cdot w\|\leq e^{-\frac{c}{2}m}\|w\|,
$
for every $w\in Df^{m}(y_{-m}) (C_{a}(y_{-m}))$.
This proves the claim.
\end{proof}

Now the remaining proof of the proposition, that $\mathcal{E}_{\hat{y}}^{u}$ is a vector subspace, follows a standard route.
For each $n\geq1$ consider the set $\mathscr{S}_{n}$ formed by subspaces of $T_{y_{0}}M$ 
given by
$$
\mathscr{S}_{n}:=\left\{ E\leqslant T_{y_{0}}M:\ E\subset Df^{n}(y_{-n}) (C_{a}(y_{-n})) \mbox{ and }\dim(E)=d_{u}\right\},
$$
where $E\leqslant T_{y_{0}}M$ means that $E$ is a vector subspace
of $T_{y_{0}}M$. 
It is not hard to check that $\mathscr{S}_{n}$ is a compact set (in the Grassmannian topology).
Moreover, $\mathcal{S}_{n}$ is non-empty subset because $F_{n}:=Df^{n}(y_{-n})\cdot F_{y_{-n}}\in\mathscr{S}_{n}$,
for every $n\in\mathbb{N}$, where $F$ is the subbundle (not necessarily $Df$-invariant) used to define the cone field 
$C_{a}$. Furthermore, $\mathscr{S}_{n+1}\subset\mathscr{S}_{n}$
for every $n\in\mathbb{N}$: if $E:=Df^{n+1}(y_{-n-1})\cdot Z\in\mathscr{S}_{n+1}$
then $E:=Df^{n}(y_{-n})\cdot\left[Df(y_{-n-1})\cdot Z\right]\in\mathscr{S}_{n}$.
Therefore $\cap_{n\in\mathbb{N}}\mathscr{S}_{n}\neq\emptyset$ and
$G\in\cap_{n\in\mathbb{N}}\mathscr{S}_{n}$ if, and only if $G\subset\mathcal{E}_{\hat{y}}^{u}$. 
Hence, in order to conclude the proof of the proposition we are left to show
that $\#\cap_{n\in\mathbb{N}}\mathscr{S}_{n}=1$.
Suppose there exist $G\neq G^{'}\in\cap_{n\in\mathbb{N}}\mathscr{S}_{n}$ and 
take $v\in G\backslash G^{'}$. On the one hand
$
\|\left[Df^{n}(y_{-n})\right]^{-1}\cdot v\|\leq e^{-\frac{c}{2}n}\|v\|,
$
for every $n\in\mathbb{N}$. On the other hand, writting $v=v_{s}+v^{'}$
where $v_{s}\in E_{y_{0}}^{s}\backslash\left\{ 0\right\} $ and $v^{'}\in G^{'}$,
we get that, for every $n\in\mathbb{N}$,
\begin{align*}
\|\left[Df^{n}(y_{-n})\right]^{-1}\cdot v\| & \geq\|\left[Df^{n}(y_{-n})\right]^{-1}\cdot v_{s}\|-\|\left[Df^{n}(y_{-n})\right]^{-1}\cdot v^{'}\| 
\geq\lambda^{-n}\cdot\|v_{s}\|-e^{-\frac{c}{2}n}\|v^{'}\|,
\end{align*}
which can hold if and only if $\|v_{s}\|=0$.
This proves that $v\in G^{'}$, leading to a contradiction.
Thus, 
$\bigcap_{n\in\mathbb{N}}\mathcal{S}_{n}$
contains a unique vector subspace $\mathcal{E}_{\hat{y}}^{u}$. This completes the proof of the proposition.
\end{proof}

\begin{remark}\label{rmk:saturatedS}
Proposition \ref{prop:hatnuishyp} implies that for every $\hat{x}\in\hat{H}_{\infty}$
and $\hat{y}\in\hat{\Delta}(\hat{x})$ there exists 
a unique splitting $T_{\hat{y}}=E_{y_{0}}^{s}\oplus\mathcal{E}_{\hat{y}}^{u}$
with the backward contraction property \eqref{eq:contractivespace}. Given $\hat{y}\in\bigcup_{\hat{x}\in\hat{H}_{\infty}}\hat{\Delta}(\hat{x})$, the vector subpace
$$
\mathcal{E}_{\hat{f}^{-1}(\hat{y})}^{u}:=\bigcap_{n\in\mathbb{N}}Df^{n}(y_{-1-n}) ( C(y_{-1-n}))
$$
is clearly an unstable space in $T_{\hat{f}^{-1}(\hat{y})}=T_{y_{-1}}M$.
Thus, for every $\hat{x}\in\hat{S}_{\infty}:=\bigcup_{n\geq0}\hat{f}^{-n}(\hat{\Delta}_{\infty})$
there is a $Df$-invariant splitting $T_{\hat{x}}=E_{x_{0}}^{s}\oplus\mathcal{E}_{\hat{x}}^{u}$,
such that 
$\|D\hat{f}^{n}(\hat{x})|_{E_{x_{0}}^{s}}\| \leq\lambda^{n}<1$ and
$\|D\hat{f}^{-n}(\hat{x})|_{\mathcal{E}_{\hat{x}}^{u}}\|  \leq e^{-\frac{c}{2}n}<1$
for every $n\in\mathbb{N}$. By \cite{QXZ09} we have the
existence of local stable manifolds 
$\left\{ W_{loc}^{s}(\hat{x})\right\} _{\hat{x}\in\hat{S}_{\infty}}$. 
\end{remark}

The following lemma is useful to distinguish unstable disks and to prove 
that some ergodic component of the accumulation points of the measures defined by ~\eqref{eq:munD}
have the SRB property.

\begin{lemma}
There is $\varepsilon>0$ such that if $\hat{x},\hat{y}\in\hat{H}_{\infty}$
and $dist(x_{0},y_{0})<\varepsilon$ then either ${\Delta}(\hat{x})={\Delta}(\hat{y})$
or ${\Delta}(\hat{x})\cap {\Delta}(\hat{y})=\emptyset$.\label{prop:disjointmanifolds}
\end{lemma}

\begin{proof}
Fix $z\in M$ and $\vep>0$ small so that $B(z,\vep) \times \Gamma$
is contained in the domain of the homeomorphism $\varphi_z$ given by equation \eqref{lem:localstructureMf}. Consider the set
$
H_{z}:=\{ \hat{x}\in\hat{H}_{\infty}:\ \Delta(\hat{x})\cap B(z,\vep)\neq\emptyset\} .
$
We claim that for any $\hat{x},\hat{y}\in H_{z}$ the sets
$\hat{\Delta}(\hat{x})\cap\pi^{-1}(B(z,\vep))$ and $\hat{\Delta}(\hat{y})\cap\pi^{-1}(B(z,\vep))$
either coincide or are disjoint.
Indeed, since $B(z,\vep) \times \Gamma$ is contained in the domain of the homeomorphism $\varphi_{z}$,
$
\hat{\Delta}(\hat{x})\cap\pi^{-1}(B(z,\vep))=\varphi_{z}((\Delta(\hat{x})\cap B(z,\vep))\times\left\{ \xi\right\} )
$
and
$
\hat{\Delta}(\hat{y})\cap\pi^{-1}(B(z,\vep))=\varphi_{z}((\Delta(\hat{y})\cap B(z,\vep))\times\left\{ \zeta\right\} )
$
for some $\xi,\zeta \in \Gamma$.
In particular, if $\xi\neq\zeta$ then $(\hat{\Delta}(\hat{x})\cap\pi^{-1}(B(z,\vep)))\bigcap(\hat{\Delta}(\hat{y})\cap\pi^{-1}(B(z,\vep)))=\emptyset$. 
Otherwise, $\hat{x}$ and $\hat{y}$ are pre-orbits that follow the same pre-orbit of $z$ and,
by \cite[Proposition VII.2.1]{QXZ09}, either
$
\hat{\Delta}(\hat{x})\cap\pi^{-1}(B(z,\vep))=\hat{\Delta}(\hat{y})\cap\pi^{-1}(B(z,\vep))
$
or
$
(\hat{\Delta}(\hat{x})\cap\pi^{-1}(B(z,\vep)))\bigcap(\hat{\Delta}(\hat{y})\cap\pi^{-1}(B(z,\vep)))=\emptyset.
$
\end{proof}

\subsection{Absolute continuity\label{sec:srbproperty}}

A main difference between our setting and the diffeomorphism case is the fact that the family of unstable disks $\left\{ \Delta(\hat{x})\right\} _{\hat{x}\in\hat{H}_{\infty}}$
is not pairwise disjoint. Indeed, there may have infinitely many unstable disks
passing through a single point.
 By Lemma \ref{prop:disjointmanifolds},
there is $\varepsilon\in(0,1)$ such that if $dist(x_{0},y_{0})<\varepsilon$
then either $\Delta(\hat{x})\cap\Delta(\hat{y})=\emptyset$
or $\Delta(\hat{x})=\Delta(\hat{y})$. Intuitively, this means that in balls of radius $\vep>0$ 
we only observe disjoint disks (cf. Figure \ref{fig:intersections}).

\begin{figure}[htb]
\centering
\includegraphics[scale=.6]{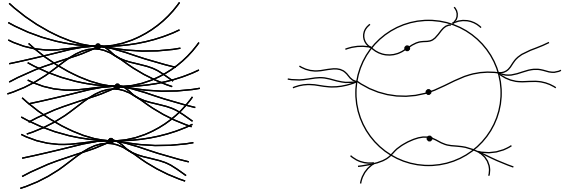}
\caption{Intersection of arbitrary unstable disks (on the left) and intersection of unstable disks 
only for close pre-orbits (on the right).\label{fig:intersections}}
\end{figure}

Recall that the set $\hat{H}_{\infty} \subset M^{f}$ was defined by Proposition \ref{lem:suppdisk}:
every $\hat{x}\in \hat{H}_{\infty}$ is an accumulation point of a sequence $(f^{n_{j}}(z_{n_{j}}))_{j\in\mathbb{N}}$,
where $z_{n_{j}}\in H_{n_{j}}$ and $n_{j}\to\infty$ when $j\to\infty$.
Moreover, $\pi\mid_{\hat{\Delta}(\hat{x})}: \hat{\Delta}(\hat{x}) \to \Delta(\hat{x})$ is a bijection.
Consider the foliated cylinder
\begin{equation}
C_{r}(\hat{x}):=\bigcup_{y\in\Delta(\hat{x})}W_{r}^{s}(y),\label{eq:box_stable_foliated}
\end{equation}
where $W_{r}^{s}(y)$ is the connected component of $W_{loc}^{s}(y)\cap B(y,r)$
that contains $y$. We will construct a covering of the support of
$\hat{\nu}$ by sets of the form $C_r(\hat{x})$.
We will show that in each set of this covering that has $\hat{\nu}$-positive
measure we can find a partition in unstable disks and prove
that $\pi_{*}\hat{\nu}_{\gamma}\ll Leb_{W_{loc}^{u}(\hat{z})}$,
for almost every $\gamma$ with respect to such partition. More details are as follows.

Let $\delta>0$ be the size of the unstable disks obtained in Proposition \ref{lem:suppdisk} and $0<r<\delta$ be a lower bound for the size of local stable
manifolds. Assume without loss of generality
that $C_{r}(\hat{x})\subset B(x_{0},\varepsilon)$. 
Using the homeomorphism given by expression ~\eqref{lem:localstructureMf}
we can identify $\pi^{-1}(C_{r}(\hat{x}))$
with $C_{r}(\hat{x})\times\Gamma$.
Fix $\varepsilon_{1}>0$ and consider $\left\{ \Delta_{\hat{x},l}\right\} _{l=1}^{n}$
a finite open covering of $\overline{\Delta(\hat{x})}$ by balls in such
way that the intersection of the sub-cylinders
\begin{equation}
C_{r}(\hat{x},l):=\bigcup_{y\in\Delta_{\hat{x},l}}W_{r}^{s}(y)\label{eq:cutted_box_stable_foliated}
\end{equation}
with any disk $\Delta$ that is tangent to the cone field $C_{a}$
has diameter smaller than $\varepsilon_{1}$ in $\Delta$. Clearly,
$\pi^{-1}(C_{r}(\hat{x},l))$ is identified
with $C_{r}(\hat{x},l)\times\Gamma$.
Let $H^{s}:C_{r}(\hat{x})\rightarrow\Delta(\hat{x})$
be the projection obtained via holonomy along stable manifolds given by $H^{s}(y):=W_{r}^{s}(y)\pitchfork\Delta(\hat{x})$
for every $y\in C_{r}(\hat{x})$. 

We say that $\Delta(f^{j}(z),\delta)$
\emph{crosses} $C_{r}(\hat{x},l)$ if $H^{s}|_{\Delta(f^{j}z,\delta)\cap C_{r}(\hat{x},l)}$
is a diffeomorphism onto $\Delta_{\hat{x},l}$. Given $(\hat{x},l)$, fixed,
by some abuse of notation, we will denote by $\Delta(f^{j}(z),\delta)$
the intersection of the disk $\Delta(f^{j}(z),\delta)$
with $C_{r}(\hat{x},l)$. The choice of $\varepsilon_{1}$
assures that any disk $\Delta(f^{j}(z),\delta)$
that intersects $C_{r}(\hat{x},l)$ crosses $C_{r}(\hat{x},l)$.
Define for each $j\in\mathbb{N}$ 
\begin{align}
\hat{\mathcal{K}}_{j}(\hat{x},l) & :=\left\{ \hat{f}^{j}(\pi^{-1}(D(z,j,\delta))):\Delta(f^{j}z,\delta)\mbox{ crosses }C_{r}(\hat{x},l)\right\} \label{eq:j_hyp_disk}
\end{align}
and
\begin{equation}
\hat{\mathcal{K}}_{\infty}(\hat{x},l):=\left\{ \hat{\Delta}(\hat{z}):\ \Delta(\hat{z})\mbox{ crosses }C_{r}(\hat{x},l)\right\} .\label{eq:inf_hyp_disk}
\end{equation}
\begin{figure}[htb]
\centering
\includegraphics[scale=0.6]{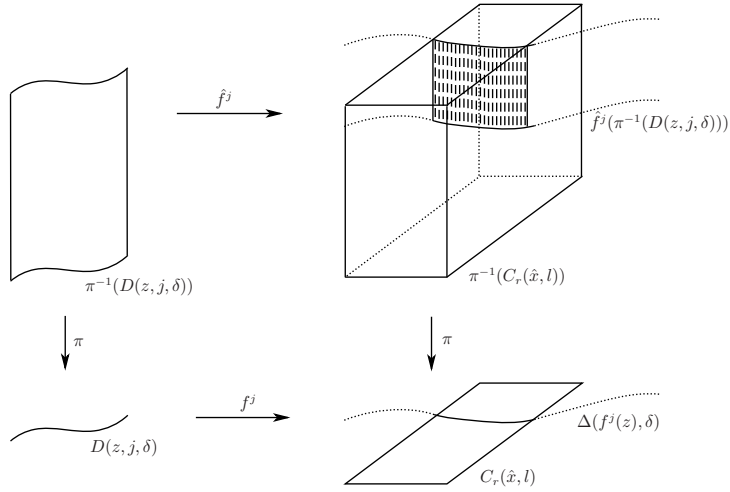}
\caption{Elements in $\mathcal{\hat{K}}_{j}(\hat{x},l)$}
\end{figure}
Observe that, since $\left\{ D(z,j,\delta)\right\} _{z\in H_{j}^{*}}$
is a pairwise disjoint family of disks (cf. Proposition \ref{prop:diskformeasure}),
then the family $\left\{ \pi^{-1}(D(z,j,\delta))\right\} _{z\in H_{j}^{*}}$
is also pairwise disjoint and, consequently, the same holds for the elements of $\hat{\mathcal{K}}_{j}(\hat{x},l)$. 
Denote by $\hat{K}_{j}(\hat{x},l)$ the
union of the elements of $\hat{\mathcal{K}}_{j}(\hat{x},l)$,
for each $0\leq j\leq\infty$.

\begin{remark}
By construction, the support of $\hat{\nu}_{n}|_{\pi^{-1}(C_{r}(\hat{x},l))}$  (recall $\hat{\nu}_{n}$ is the lift of $\nu_{n}$ \eqref{eq:lift_of_nun})
is contained in $\bigcup_{j=0}^{n-1}\hat{K}_{j}(\hat{x},l)$
and the support of any accumulation point $\hat \nu$ of the sequence $(\hat \nu_n)_n$ is contained in
$\hat{\mathcal{K}}_{\infty}(\hat{x},l)$.
\end{remark}

Since we are interested in describing the disks in the support of the measures $\hat\nu_n$ and their accumulation disks 
it will be useful to show that these measures are not concentrated near the boundary of the disks. For that reason,
given $\vep>0$, let $V_{\vep}(\partial\Delta(f^{j}(z),\delta))$
be the $\vep$-neighborhood of the boundary of $\Delta(f^{j}(z),\delta)$,
for $z\in H_{j}^{*}$ and $j\in\mathbb{N}$. Define
$\Delta_{\vep}(f^{j}(z),\delta):=\Delta(f^{j}(z),\delta)\backslash V_{\vep}(\partial\Delta(f^{j}(z),\delta))$
and $D_{\vep}(z,j,\delta):=f_{z}^{-j}(\Delta_{\vep}(f^{j}(z),\delta))$.
Consider the measures
\begin{equation}
\hat{\nu}_{n,\vep}:=\frac{1}{n}\sum_{j=0}^{n-1}\hat{f}_{*}^{j}\hat{m}_{j,\vep}.\label{eq:lift_nunepsilon}
\end{equation}
where $\hat{m}_{j,\vep}:=\sum_{z\in H_{j}^{*}}\hat{m}_{j,z,\vep}$ and
$\hat{m}_{j,z,\vep}:=(\varphi_{z})_{*}(Leb_{D_{\vep}(z,j,\delta)}\times\mathds{P})$.
Then we have the following:

\begin{lemma}
If $\vep>0$ is sufficiently small then $\hat{\nu}_{n,\vep}(M^{f})\geq\frac{\alpha}{2}$
for every $n$ large enough.\label{lem:hatnuposonsubdisk}
\end{lemma}
\begin{proof}
Note that $\hat{\nu}_{n}(M^{f})=\nu_{n}(M)\geq\nu_{n}(H)\geq\alpha$,
for every $n$ sufficiently large (Lemma \ref{prop:nupositiva}).
Thus:
\begin{align*}
\hat{\nu}_{n}(M^{f})-\hat{\nu}_{n,\vep}(M^{f}) 
 & \leq\frac{1}{n}\sum_{j=0}^{n-1}\sum_{z\in H_{j}^{*}}\hat{f}_{*}^{j}\hat{m}_{j,z}(\pi^{-1}(\Delta(f^{j}(z),\delta)\backslash\Delta_{\vep}(f^{j}(z),\delta)))\\
 & =\frac{1}{n}\sum_{j=0}^{n-1}\sum_{z\in H_{j}^{*}}\pi_{*}\hat{f}_{*}^{j}\hat{m}_{j,z}(\Delta(f^{j}(z),\delta)\backslash\Delta_{\vep}(f^{j}(z),\delta))\\
 & \leq\frac{1}{n}\sum_{j=0}^{n-1}\sum_{z\in H_{j}^{*}}f_{*}^{j}Leb_{D(z,j,\delta)}(\Delta(f^{j}(z),\delta)\backslash\Delta_{\vep}(f^{j}(z),\delta)).
\end{align*}
Taking $\vep>0$ sufficiently small, the Lebesgue
measure of the union of the $\vep$-neighborhoods of the boundary
of the disks $\Delta(f^{j}(z),\delta)$ is a
small fraction of the measure of the union of the disks $\Delta(f^{j}(z),\delta)$.
By bounded distortion (cf. Proposition~\ref{prop:distorcionhyptime}) this is also true for $f_{*}^{j}Leb_{D}$.
Therefore 
we may reduce $\vep$, if necessary, so that 
$
\sum_{z\in H_{j}^{*}}f_{*}^{j}Leb_{D(z,j,\delta)}(\Delta(f^{j}(z),\delta)\backslash\Delta_{\vep}(f^{j}(z),\delta))\leq\frac{\alpha}{2},
$
which guarantees that 
$\hat{\nu}_{n,\vep}(M^{f})\geq\hat{\nu}_{n}(M^{f})-\frac{\alpha}{2}\geq\frac{\alpha}{2}$
and proves the lemma.
\end{proof}

As a consequence of the previous lemma we also get:

\begin{lemma}\label{lem:hatnuposincilinder}
Assume that $(\hat{\nu}_{n_{k}})_k$ converges to $\hat{\nu}$.
There exists $(\hat{x},l)$ and 
$\kappa(\hat{x},l)>0$ such that $\hat{\nu}(\hat{K}_{\infty}(\hat{x},l))>\kappa(\hat{x},l)$
and $\hat{\nu}_{n}(\cup_{j=0}^{n-1}\hat{K}_{j}(\hat{x},l))>\kappa(\hat{x},l)$
for every $n$ in the subsequence of $(n_{k})$.
\end{lemma}

\begin{proof}
Up to considering a subsequence of $(n_{k})_{k\in\mathbb{N}}$,
we may assume that $\hat{\nu}_{n_{k},\vep}$ converges to some measure
$\hat{\nu}_{\vep}$ (in the weak{*} topology) so that $\hat{\nu}_{\vep}(M^{f})\geq\frac{\alpha}{2}$ (cf. Lemma \ref{lem:hatnuposonsubdisk}).
Note that $supp(\hat{\nu}_{\vep})\subset\cap_{n\in\mathbb{N}}\overline{\cup_{j=0}^{n-1}\hat{K}_{j,\vep}}$
and this set is covered by the union of the interiors of the boxes
$\pi^{-1}(C_{r}(\hat{x}))$. By compactness,
there is a finite number of $\hat{x}$ such that $supp(\hat{\nu}_{\vep})\subset\cup_{\hat{x}\in\hat{H}_{\infty}}(\pi^{-1}(C_{r}(\hat{x})))$. In consequence, there exists $\hat{x}$ such that $\hat{\nu}_{\vep}(\pi^{-1}(C_{r}(\hat{x})))>0$
and consequently there is $(\hat{x},l)$ and $\kappa(\hat{x},l)>0$
such that $\hat{\nu}_{\vep}(\pi^{-1}(C_{r}(\hat{x},l)))\geq\kappa(\hat{x},l)>0$.
By definition of $C_{r}(\hat{x},l)$,  any
disk $\Delta(f^{j}(z),\delta)$ that intersects
$C_{r}(\hat{x},l)$ crosses $C_{r}(\hat{x},l)$.
This implies that 
$
\hat{\nu}_{n,\vep}(\pi^{-1}(C_{r}(\hat{x},l)))=\hat{\nu}_{n,\vep}(\cup_{j=0}^{n-1}\hat{K}_{j}(\hat{x},l)).
$
Reducing $r>0$ and $\Delta_{\hat{x},l}$ if necessary, we can assume
that the $\hat{\nu}$-measure of the boundary of $\pi^{-1}(C_{r}(\hat{x},l))$
is zero. Since $\hat{\nu}_{\vep}\leq\hat{\nu}$ then
$
\lim_{k\to\infty}\hat{\nu}_{n_{k},\vep}(\pi^{-1}(C_{r}(\hat{x},l)))=\hat{\nu}_{\vep}(\pi^{-1}(C_{r}(\hat{x},l)))\geq\kappa(\hat{x},l).
$
Thus, 
$
\hat{\nu}_{n}(\cup_{j=0}^{n-1}\hat{K}_{j}(\hat{x},l))\geq\hat{\nu}_{n,\vep}(\cup_{j=0}^{n-1}\hat{K}_{j}(\hat{x},l))\geq\kappa(\hat{x},l)
$
for every $n$ in the subsequence $(n_{k})$.
Finally, since we have that
$
\limsup_{k\to\infty}\hat{\nu}_{n_{k}}(\cup_{j=0}^{n-1}\hat{K}_{j}(\hat{x},l))\leq\hat{\nu}(\cup_{j=0}^{n-1}\hat{K}_{j}(\hat{x},l))
$
and $\cap_{n\in\mathbb{N}}\overline{(\cup_{j=0}^{n-1}\hat{K}_{j}(\hat{x},l))}\subset\hat{K}_{\infty}(\hat{x},l)$, 
it follows that $\hat{\nu}(\hat{K}_{\infty}(\hat{x},l))\geq\kappa(\hat{x},l)>0$.
\end{proof}

We now use the local product structure from expression ~\eqref{lem:localstructureMf}
to describe the sets $\hat{f}^{j}(\pi^{-1}(D(z,j,\delta)))\in\hat{\mathcal{K}}_{j}(\hat{x},l)$.
We begin with the following remark.

\begin{remark}
\label{rem:representationofhypdisk}
Fix $\hat{f}^{j}(\pi^{-1}(D(z,j,\delta)))\in\hat{\mathcal{K}}_{j}(\hat{x},l)$.
The set $\pi^{-1}(D(z,j,\delta))$ is identified with
$D(z,j,\delta)\times\Gamma$ via the homeomorphism $\varphi_{z}$.
Since the set $\pi^{-1}(D(z,j,\delta))\subset M^{f}$
consists of pre-orbits of $D(z,j,\delta)$ then $\hat{f}^{j}(\hat{\pi}^{-1}(D(z,j,\delta)))$
consists of the pre-orbits of $\Delta(f^{j}(z),\delta)=f^{j}(D(z,j,\delta))$
that visit the set $D(z,j,\delta)$. In other words, $\hat{f}^{j}(\hat{\pi}^{-1}(D(z,j,\delta)))$
is the set of pre-orbits of $\Delta(f^{j}(z),\delta)$ where the first $j$ pre-images are fixed. 
Hence the set $\hat{f}^{j}(\hat{\pi}^{-1}(D(z,j,\delta)))$
is homeomorphic to $\Delta(f^{j}(z),\delta)\times[x_{-j},...,x_{-1}]_{x_{0}}$
by $\varphi_{x_{0}}$ , where 
\[
[x_{-j},...,x_{-1}]_{x_{0}}:=\left\{ (i_{n})_{n\in\mathbb{N}}\in\Gamma:\ x_{-k}\in V_{i_{k}...i_{1}}(x_{0}),\mbox{ para cada }1\leq k\leq j\right\} 
\]
is a cylinder of length $j$
in $\Gamma$.
Moreover, given $j\in\mathbb{N}$ and $z\in H_{j}^{*}$, if $\Delta(f^{j}(z),\delta)\subset B(x,\rho)$
consider
\[
\left[z,f(z),...,f^{j}(z)\right]_{x}:=\left\{ (i_{n})_{n\in\mathbb{N}}\in\Gamma:\ f^{j-k}(z)\in V_{i_{n}i_{n-1}...i_{1}}(x),\ 0\leq k\leq j\right\}, 
\]
as the set of pre-orbits of $x$ such that $f^{j-k}(z)=f^{-k}(f^{j}(z))$.
Then $\left[z,f(z),...,f^{j}(z)\right]_{x}$ is a cylinder of length $j$ in $\Gamma$ and
the set $\hat{f}^{j}(\pi^{-1}(D(z,j,\delta)))\in\hat{K}_{j}(\hat{x},l)$
is identified with $\Delta(f^{j}(z),\delta)\times\left[z,f(z),...,f^{j}(z)\right]_{x_{0}}$.
\end{remark}

\begin{figure}[htb]
\centering
\includegraphics[scale=0.75]{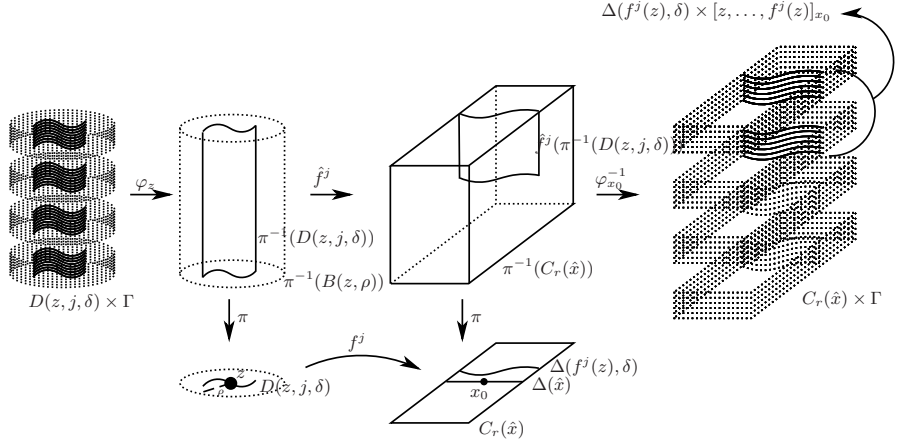}
\caption{Representation of $\hat{f}^{j}(\pi^{-1}(D(z,j,\delta)))$
in $C_{r}(\hat{x})\times\Gamma$}
\end{figure}

In order to prove absolute continuity we construct an auxiliary sequence of partitions
that allow us to study an element $\hat{\nu}_{\hat{\Delta}}$
of the disintegration of $\hat{\nu}$ with respect to $\hat{\mathcal{K}}_{\infty}(\hat{x},l)$
in terms of the sequence of measures $(\hat{\nu}_{n})_{n\in\mathbb{N}}$.
This construction is one of the most technical steps in the proof.
\medskip

Fix a pair $(\hat{x},l)$ such that $\hat{\nu}(\hat{K}_{\infty}(\hat{x},l))>0$
(cf. Lemma \ref{lem:hatnuposincilinder}). 
The elements
of $\hat{\mathcal{K}}_{\infty}(\hat{x},l)$ are obtained
as a limit of $\hat{f}^{j}(\pi^{-1}(D(z,j,\delta)))$,
where $\Delta(f^{j}(z),\delta)=f^{j}(D(z,j,\delta))$
are disks that cross $C_{r}(\hat{x},l)$. 
When no confusion arises, for notational simplicity, we shall omit $(\hat{x},l)$ from the sets 
$\hat{\mathcal{K}}_{j}(\hat{x},l)$, $\hat{K}_{j}(\hat{x},l)$ and $C_{r}(\hat{x},l)$ previously defined.
We denote $\hat{x}=(x_{-n})_{n\in\mathbb{N}}$,
$\hat{\Delta}:=\hat{\Delta}(\hat{x})$ and $\Delta:=\Delta(\hat{x})$. 
Define 
$$
K^{\dagger}:=\bigcup_{j\in{\mathbb{N}\cup\left\{ \infty\right\}} }\hat{K}_{j}\times\left\{ j\right\}. 
$$
Then, an element of $K^{\dagger}$ is a pair $(\hat{a},m)$
where $\hat{a}\in\hat{f}^{m}(\pi^{-1}(D(z,m,\delta)))$
for some $z\in H_{m}^{*}$. Recall that 
$\hat{f}^{m}(\pi^{-1}(D(z,m,\delta)))=\varphi_{x_{0}}(\Delta(f^{m}(z),\delta)\times\left[z,\dots,f^{m}(z)\right]_{x_{0}})$
where $\left[z,\dots,f^{m}(z)\right]_{x_{0}}$ is a cylinder
of length $m$ in $\Gamma=\left\{ 1,\dots,d\right\} ^{\mathbb{N}}$.
Since $\varphi_{x_{0}}$ is a homeomorphism, to each $\hat{a}\in\hat{f}^{m}(\pi^{-1}(D(z,m,\delta)))$
one can associate a unique pair $(a,\lambda)\in\Delta(f^{m}(z),\delta)\times\left[z,\dots,f^{m}(z)\right]_{x_{0}}$
such that $\varphi_{x_{0}}(a,\lambda)=\hat{a}$ and $a=\pi(\hat{a})$.
Having in mind this identification by $\varphi_{x_{0}}$, we will write an element $(\hat{a},m)\in K^{\dagger}$ 
as a triple $(a,\lambda,m)\in\Delta(f^{m}(z),\delta)\times\left[z,\dots,f^{m}(z)\right]_{x_{0}}\times\mathbb{N}$
where $\varphi_{x_{0}}(a,\lambda)=\hat{a}$.

\begin{figure}[htb]
\centering
\includegraphics[scale=0.65]{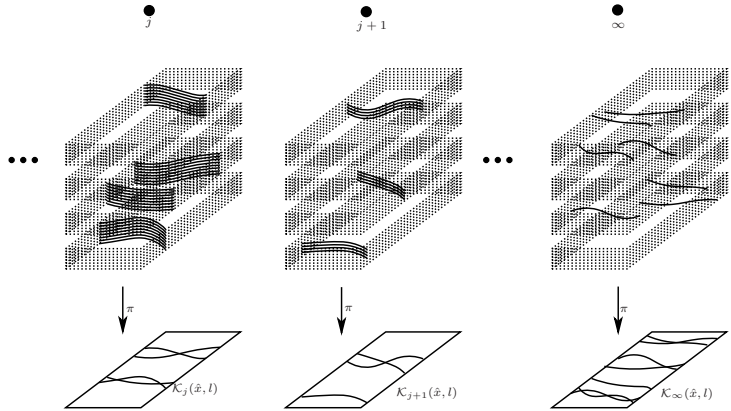}
\caption{The space $K^{\dagger}$}
\end{figure}

In parallel to the construction of $\hat{\nu}_{m}$ 
consider the measure $\xi_{m}^{\dagger}$ in $K^{\dagger}$ given by
\begin{equation}\label{eq:daggerss}
\xi_{m}^{\dagger}(\bigcup_{j\in{\mathbb{N}\cup\left\{ \infty\right\}} }\hat{A}_{j}\times\left\{ j\right\} ):=\frac{1}{m}\sum_{j=0}^{m-1}\hat{f}_{*}^{j}\hat{m}_{j}(\hat{A}_j),
\end{equation}
whenever $\bigcup_{j\in{\mathbb{N}\cup\left\{ \infty\right\}} }\hat{A}_{j}\times\left\{ j\right\} $
is a subset of $K^{\dagger}$ (in particular $\hat{A}_{j}\subset\hat{K}_{j}$ for every $j$).
The measure $\hat{f}_{*}^{j}\hat{m}_{j}$ is supported
in $\hat{K}_{j}$ and
$
\hat{\nu}_{m}(\hat{A})=\hat{\nu}_{m}(\bigcup_{j=0}^{m-1}\hat{A}\cap\hat{K}_{j})=\xi_{m}^{\dagger}(\bigcup_{j=0}^{m-1}(\hat{A}\cap\hat{K}_{j})\times\left\{ j\right\} )
$
for every $m\in\mathbb{N}$ and $\hat{A}\subset\hat{C}_{r}$. This implies that $\xi_{m}^{\dagger}(K^{\dagger})>\kappa>0$
for some subsequence of integers (given by Lemma \ref{lem:hatnuposincilinder})
and that $(\xi_{m}^{\dagger})_{m\in\mathbb{N}}$ admits an accumulation 
that is a positive measure $\xi^{\dagger}$ supported on $\hat{K}_{\infty}\times\left\{ \infty\right\} \subset K^{\dagger}$.
In particular, $\xi^{\dagger}(\hat{B}\times\left\{ \infty\right\} )=\hat{\nu}(\hat{B})$,
for every $\hat{B}\subset\hat{K}_{\infty}$. 

We will show that $\xi^{\dagger}$ has absolutely continuous disintegration
with respect to Lebesgue relatively with the partition
$\{ \hat{\Delta}\times\left\{ \infty\right\} \} _{\hat{\Delta}\in\hat{\mathcal{K}}_{\infty}}$
of $\hat{K}_{\infty}\times\left\{ \infty\right\}$ (meaning that $p_{*}^{\dagger}\xi_{\hat{\Delta}\times\left\{ \infty\right\} }^{\dagger}\ll Leb_{\pi(\hat{\Delta})}$ where $p^{\dagger}:K^{\dagger}\rightarrow\Delta$ is given by 
$p^{\dagger}(a,\lambda,n)=H^{s}(a)$ and $H^{s}$ is the stable holonomy in $C_{r}$)
and use this fact to conclude that the same holds for $\hat{\nu}$. 
For that, we define a sequence of increasing and generating partitions $\left\{ \mathcal{P}_{k}\right\} _{k\in\mathbb{N}}$
for $K^{\dagger}$ that generate
 $\left\{ \hat{\Delta}\times\left\{ \infty\right\} \right\} _{\hat{\Delta}\in\hat{\mathcal{K}}_{\infty}}$.

Recall that we can identify $\pi^{-1}(C_{r})$ with $C_{r}\times\Gamma$
by the homeomorphism $\varphi_{x_{0}}$. Consider $z'\in\Delta$ and
let $W_{r}^{s}(z')$ be the local stable manifold of $z'$.
Take $(\mathcal{W}_{k})_{k\in\mathbb{N}}$ a sequence
of countable partitions of $W_{r}^{s}(z')$ whose diameter
goes to zero when $k$ goes to infinity. Fix the sequence $(\Gamma_{k})_{k\in\mathbb{N}}$
of partitions of $\Gamma$, where each $\Gamma_{k}$ is the partition
by $k$-cylinders in $\Gamma$ (hence 
their diameter goes to zero when $k$ goes to infinity). Thus, $(\mathcal{W}_{k}\times\Gamma_{k})_{k\in\mathbb{N}}$
is an increasing sequence of countable partitions whose diameter goes
to zero when $k$ goes to zero (see Figure \ref{fig:sequence_of_partitions}).
So, $\bigvee_{j=0}^{\infty}\mathcal{W}_{k}\times\Gamma_{k}$ is the
partition of $W_{r}^{s}(z')\times\Gamma$ by singletons.

\begin{figure}[htb]
\includegraphics[scale=0.56]{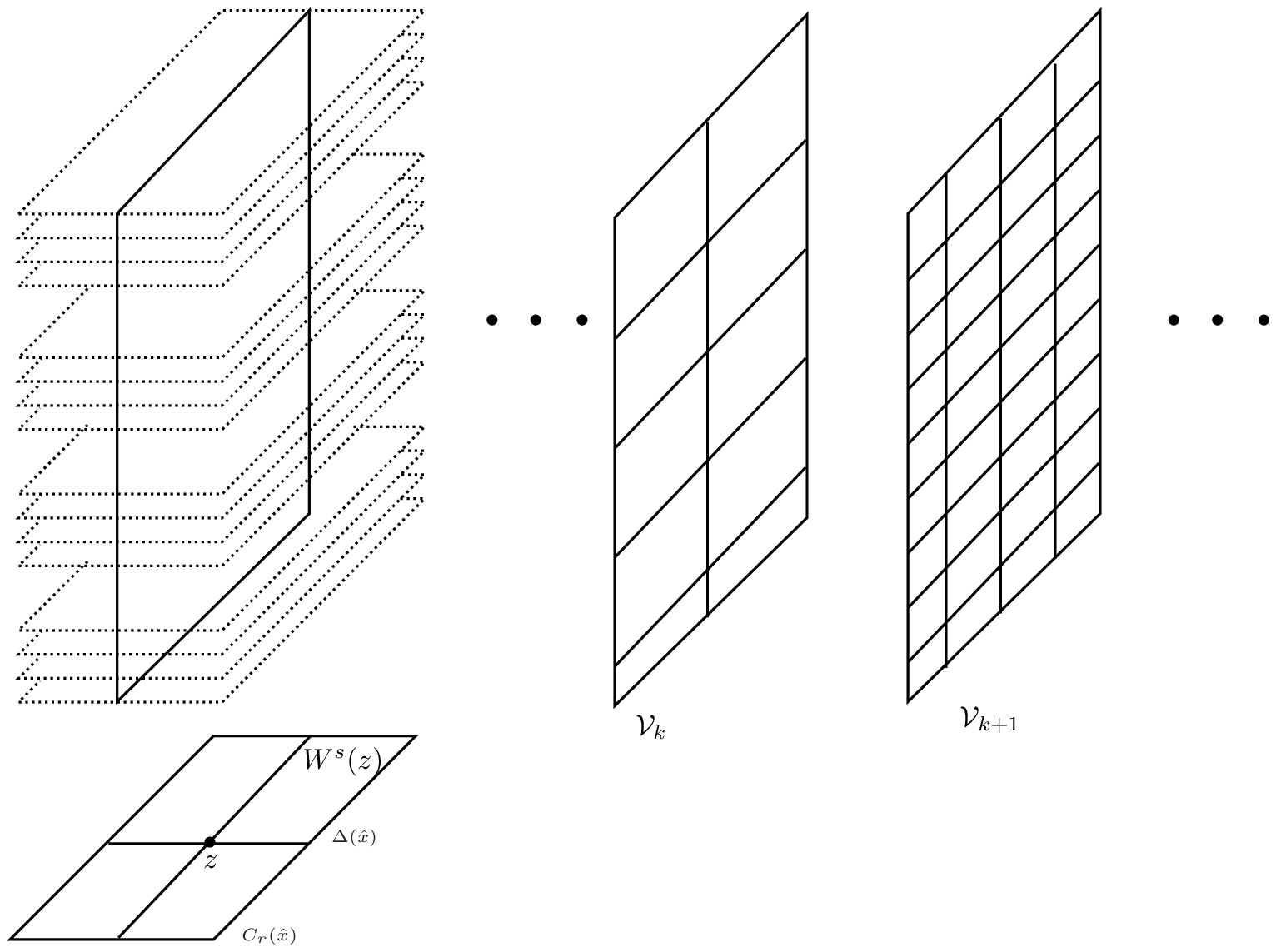} \qquad\qquad 
\includegraphics[scale=0.5]{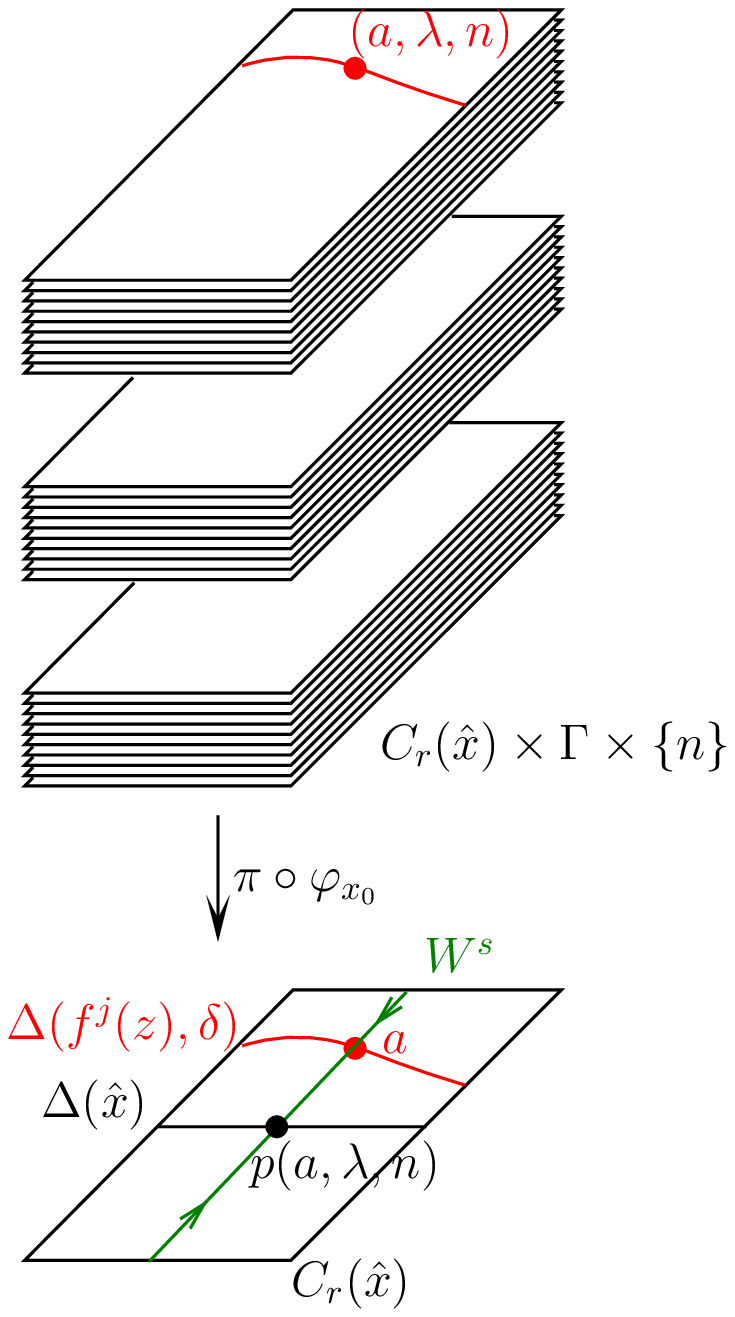}
\caption{\label{fig:sequence_of_partitions} Sequence of partitions $(\mathcal{W}_{k}\times\Gamma_{k})_{k\in\mathbb{N}}$ (on the left)
and projection $p^{\dagger}$ (on the right)}
\end{figure}

For each $k\in\mathbb{N}$, we will identify the elements of $\hat{\mathcal{K}}_{j}$
that intersects the same atom of $\mathcal{W}_{k}\times\Gamma_{k}$
to construct the partition $\mathcal{P}_{k}$. The atoms of $\mathcal{P}_{k}$
will consist of the union of elements in $\hat{\mathcal{K}}_{j}$
that intersect the same atom of $\mathcal{W}_{k}\times\Gamma_{k}$
for $j<k$ fixed or consist of the union of elements in $\hat{\mathcal{K}}_{j}$,
for every $j\geq k$, that intersect the same atom of $\mathcal{W}_{k}\times\Gamma_{k}$.
More precisely, we say that $(a,\lambda,n)\in K^{\dagger}$
and $(b,\sigma,m)\in K^{\dagger}$ belong to the same atom
of $\mathcal{P}_{k}$ if the following conditions are satisfied:
\begin{enumerate}
\item there are $y$ and $z$ such that $a\in\Delta(f^{n}(y),\delta)$
and $b\in\Delta(f^{m}(z),\delta)$ and the disks
$\Delta(f^{m}(y),\delta)$ and $\Delta(f^{m}(z),\delta)$
intersect the same atom of $\mathcal{W}_{k}$; or  \label{enu:partition1}
\item $\lambda$ and $\sigma$ belong to the same atom of $\Gamma_{k}$;\label{enu:partition2}
\item either $\left[m\geq k\text{ and }n\geq k\right]$ or $m=n<k$.\label{enu:partition3}
\end{enumerate}
Observe that, since $\Delta(f^{m}(y),\delta)$
and $\Delta(f^{m}(z),\delta)$ cross $C_{r}$,
 their intersection with $W_{r}^{s}(z')$ is never empty.
Conditions \eqref{enu:partition1},\eqref{enu:partition2} and \eqref{enu:partition3}
define an equivalence relation in $K^{\dagger}$. This implies that
$\mathcal{P}_{k}$ is a partition of $K^{\dagger}$.

We claim that the atoms of $\mathcal{P}_{k}$ consist of unions of sets that are 
obtained as products of hyperbolic disks by cylinders in $\Gamma$ by the corresponding
cone-hyperbolic times (see Figure~\ref{fig:atom_1}).
In fact, fix $(a,\lambda,m)\in K^{\dagger}$. 

Suppose first that $m<k$.
Then $(a,\lambda)\in\Delta(f^{m}(z),\delta)\times\left[z,...,f^{m}(z)\right]_{x_{0}}$
for some $z\in H_{m}^{*}$. Denote by $\mathcal{P}_{k}(a,\lambda,m)$
the atom of $\mathcal{P}_{k}$ that contains $(a,\lambda,m)$.
It is immediate from the definition that 
\[
\Delta(f^{m}(z),\delta)\times(\left[z,...,f^{m}(z)\right]_{x_{0}}\cap\Gamma_{k}(\lambda))\times\left\{ m\right\} \subset\mathcal{P}_{k}^ {}(a,\lambda,m),
\]
where $\Gamma_{k}(\lambda)$ is the atom of the partition
$\Gamma_{k}$ that contains $\lambda$. But $\Gamma_{k}(\lambda)$
is a cylinder of length $k>m$ that has non empty intersection with
the cylinder $\left[z,...,f^{m}(z)\right]_{x_{0}}$ whose
length is $m$. So $\left[z,...,f^{m}(z)\right]_{x_{0}}\supset\Gamma_{k}(\lambda)$
and $\left[z,...,f^{m}(z)\right]_{x_{0}}\cap\Gamma_{k}(\lambda)=\Gamma_{k}(\lambda)$
and, consequently, 
$
\Delta(f^{m}(z),\delta)\times\Gamma_{k}(\lambda)\times\left\{ m\right\} \subset\mathcal{P}_{k}(a,\lambda,m).
$
Now, observe that $(y,\sigma,n)\in\mathcal{P}_{k}(a,\lambda,m)$
if, and only if:
\begin{itemize}
\item $y\in\Delta(f^{n}(w),\delta)$ for some $w\in H_{n}^{*}$
and $\Delta(f^{n}(w),\delta)$ and $\Delta(f^{m}(z),\delta)$
intersect the same atom of $\mathcal{W}_{k}$;
\item $\sigma\in\Gamma_{k}(\lambda)$;
\item $m=n$.
\end{itemize}
Then $\Delta(f^{m}(w),\delta)\times\Gamma_{k}(\lambda)\times\left\{ m\right\} \subset\mathcal{P}_{k}(a,\lambda,m)$.
Therefore, we conclude that, if $m<k$ then 
\[
\mathcal{P}_{k}(a,\lambda,m)=\bigcup_{w}\Delta(f^{m}(w),\delta)\times\Gamma_{k}(\lambda)\times\left\{ m\right\} ,
\]
where the union runs over the set of all $w\in H_{m}^{*}$ such that
$\Delta(f^{m}(w),\delta)$ and $\Delta(f^{m}(z),\delta)$
intersect the same atom of $\mathcal{W}_{k}$ and $\left[w,\dots,f^{m}(w)\right]_{x_{0}}$
has non empty intersection with $\Gamma_{k}(\lambda)$.

Now, suppose that $(a,\lambda,m)\in K^{\dagger}$ with
$m\geq k$. As before,  $(a,\lambda)\in\Delta(f^{m}(z),\delta)\times\left[z,...,f^{m}(z)\right]_{x_{0}}$
for some $z\in H_{m}^{*}$ and it follows from the definition
that 
$
\Delta(f^{m}(z),\delta)\times(\left[z,...,f^{m}(z)\right]_{x_{0}}\cap\Gamma_{k}(\lambda))\times\left\{ m\right\} \subset\mathcal{P}_{k}^ {}(a,\lambda,m).
$
As $m\geq k$ we have that $\left[z,...,f^{m}(z)\right]_{x_{0}}\cap\Gamma_{k}(\lambda)=\left[z,...,f^{m}(z)\right]_{x_{0}}$,
and $\Delta(f^{m}(z),\delta)\times\left[z,...,f^{m}(z)\right]_{x_{0}}\times\left\{ m\right\} \subset\mathcal{P}_{k}(a,\lambda,m)$.
The previous argument shows that 
\[
\mathcal{P}_{k}(a,\lambda,m)=\bigcup_{j\geq k}\bigcup_{w_{j}}\Delta(f^{j}(w_{j}),\delta)\times\left[w_{j},...,f^{j}(w_{j})\right]_{x_{0}}\times\left\{ j\right\} \cup(\bigcup_{\hat{x}}\Delta(\hat{x})\times\left\{ \sigma\right\} \times\left\{ \infty\right\} )
\]
where the union runs over all $w_{j}\in H_{j}^{*}$ such that $\Delta(f^{j}(w_{j}),\delta)$
and $\Delta(f^{m}(z),\delta)$ intersect the
same atom of $\mathcal{W}_{k}$ and $\left[w_{j},...,f^{j}(w_{j})\right]_{x_{0}}$
has non empty intersection with $\Gamma_{k}(\lambda)$, and
over all $\hat{x}\in\hat{H}_{\infty}$ such that $\Delta(\hat{x})$
and $\Delta(f^{m}(z),\delta)$ intersect the
same atom of $\mathcal{W}_{k}$ and $\sigma\in\Gamma_{k}(\lambda)$.
\begin{figure}[htb]
\centering
\includegraphics[scale=0.6]{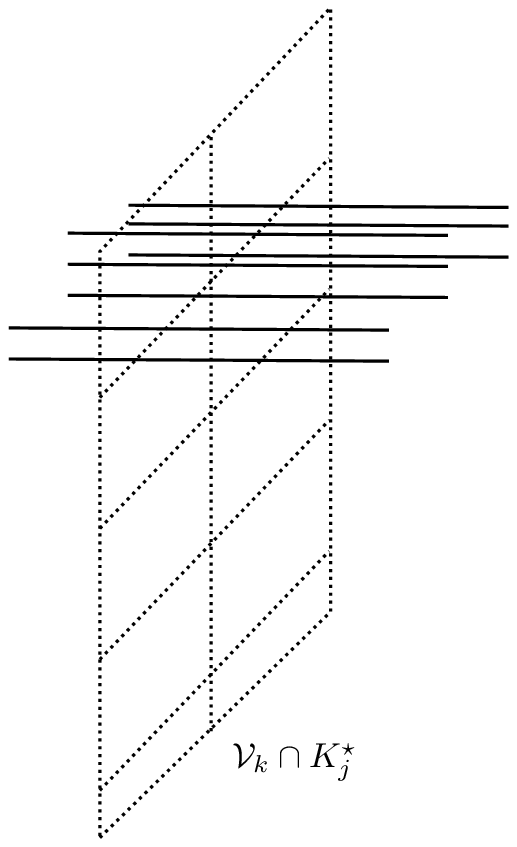}
\hspace{2cm}
\includegraphics[scale=0.5]{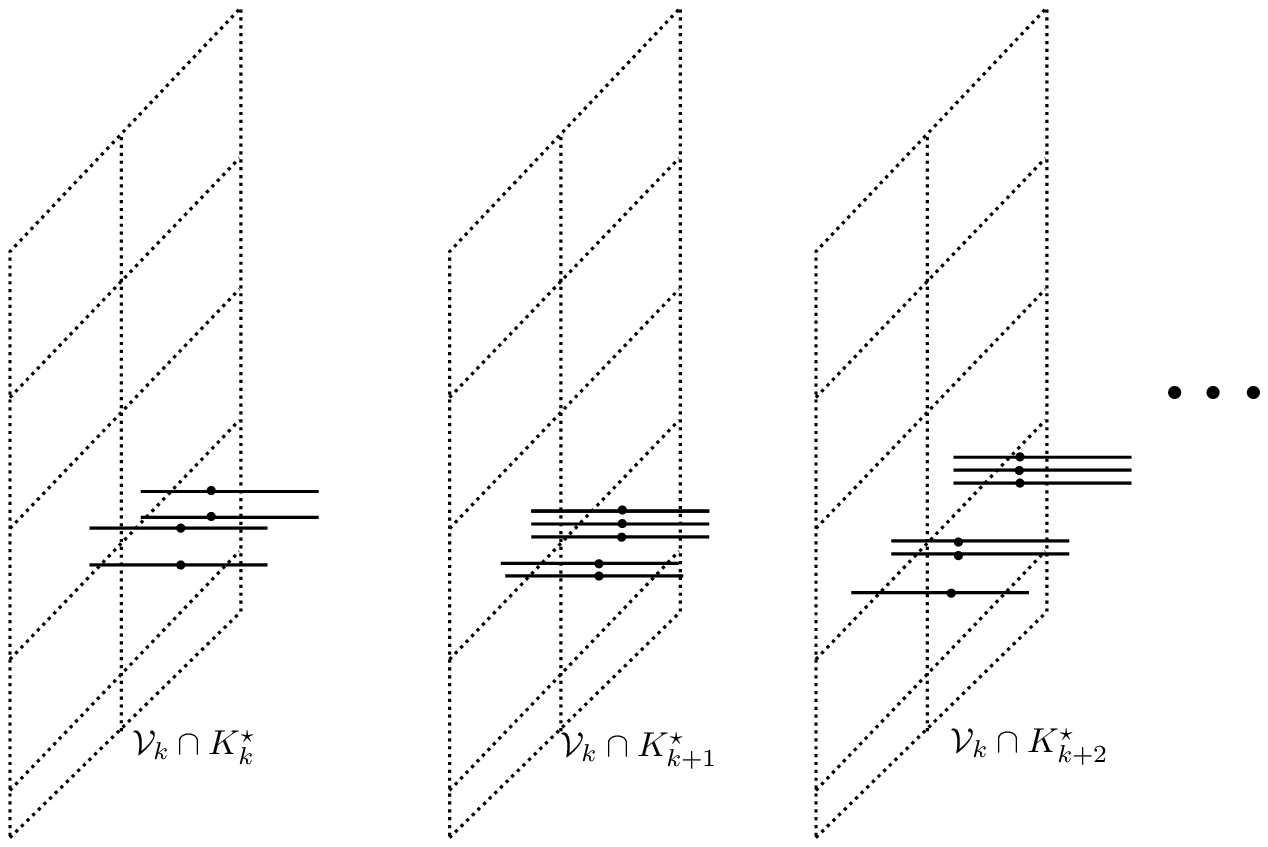}
\caption{Atom of partition $\mathcal{P}_{k}$ that contains $(\hat{u},j)$,
where $j<k$ (on the left) and atom of partition $\mathcal{P}_{k}$ that contains $(\hat{u},j)$,
where$j\geq k$ (on the right)
\label{fig:atom_1}}
\end{figure}
It is clear from the construction that the diameter of $\mathcal{W}_{k}\times\Gamma_{k}$
tends to zero as $k\to\infty$. Moreover
$
\mathcal{P}_{1}(a,\lambda,n)\supset\mathcal{P}_{2}(a,\lambda,n)\supset\dots\supset\mathcal{P}_{k}(a,\lambda,n)\supset\dots
$
and 
\[
\bigcap_{k\in\mathbb{N}}\mathcal{P}_{k}(a,\lambda,n)=\Delta(f^{n}(z),\delta)\times\left\{ \lambda\right\} \times\left\{ n\right\} 
\]
 where $a\in\Delta(f^{n}(z),\delta)$. In particular,
given $(a,\lambda,\infty)$ we have that $\bigcap_{k\in\mathbb{N}}\mathcal{P}_{k}(a,\lambda,\infty)=\Delta(\hat{y})\times\left\{ \lambda\right\} \times\left\{ \infty\right\} $
where $a\in\Delta(\hat{y})$. We are now in a position to state the main result of this section.

\begin{theorem}
There exists $C_{3}>1$ so that the following holds: given $(\hat{x},l)$ given by Lemma~\ref{lem:hatnuposincilinder}
there exists a family of conditional measures $(\hat{\nu}_{\hat{\Delta}})_{\hat{\Delta}\in\hat{\mathcal{K}}_{\infty}(\hat{x},l)}$ of $\hat{\nu}|_{\hat{K}_{\infty}(\hat{x},l)}$ such that $\pi_{*}\hat{\nu}_{\hat{\Delta}}\ll Leb_{\Delta}$, where $\Delta=\pi(\hat{\Delta})$, 
and 
$
\frac{1}{C_{3}}Leb_{\Delta}(B)\leq\pi_{*}\hat{\nu}_{\hat{\Delta}}(B)\leq C_{3}Leb_{\Delta}(B)
$
for every measurable subset $B\subset\Delta$ and for almost every $\hat{\Delta}\in\hat{\mathcal{K}}_{\infty}(\hat{x},l)$.\label{thm:abscontdisint}
\end{theorem}

\begin{proof}
Fix $(\hat{x},l)$ given by Lemma \ref{lem:hatnuposincilinder},
that is, such that $\hat{K}_{\infty}(\hat{x},l)$ has positive
$\hat{\nu}$-measure. Consider the space $K^{\dagger}$ and the sequence
of measures $(\xi_{m}^{\dagger})_{m\in\mathbb{N}}$ given
by ~\eqref{eq:daggerss}.
Consider the projection $p^{\dagger}:K^{\dagger}\rightarrow\Delta$ 
given by $p^{\dagger}(a,\lambda,n)=H^{s}(a)$.
Now, recall that the measures $\xi_{m}^{\dagger}$ involve the push-forward of 
the lift of Lebesgue measure at
hyperbolic pre-disks as in (\ref{eq:lift_leb_on_prehypdisk}) at
page \pageref{eq:lift_leb_on_prehypdisk}, where $\mathds{P}$
is a fixed probability measure in the Cantor set $\Gamma$.
Hence, in order to prove the theorem it is enough to show that there exists 
$C>1$ such that, for any measurable subset $B\subset\Delta$, any
integer $k\geq1$, any $\zeta^{\dagger}\in K^{\dagger}$ and any $n\geq1$
we have:
\begin{equation}\label{eq:lCdaggerS}
C^{-1}\xi_{n}^{\dagger}(\mathcal{P}_{k}(\zeta^{\dagger}))Leb_{\Delta}(B)\leq\xi_{n}^{\dagger}((p^{\dagger})^{-1}(B)\cap\mathcal{P}_{k}(\zeta^{\dagger}))\leq C\xi_{n}^{\dagger}(\mathcal{P}_{k}(\zeta^{\dagger}))Leb_{\Delta}(B).
\end{equation}
Note that the previous estimates are independent of the
choice of the probability measure $\mathds{P}$. 
Fix a measurable subset $B\subset\Delta$ and write $\mathcal{P}_{k}(\zeta^{\dagger})$
as the union of sets $\Delta(f^{j}(z),\delta)\times\Omega(z,k)\times\left\{ j\right\}$, where, either 
$
\Omega(z,k)=\left[z,\dots,f^{j}(z)\right]_{x_{0}}
$
or
$
\Omega(z,k)=\Gamma_{k}(\lambda)
$
with $\zeta^{\dagger}=(a,\lambda,n)$. In both cases, $\Omega(z,k)$
is a cylinder in $\Gamma$. Then

\begin{equation}
\xi_{n}^{\dagger}(\mathcal{P}_{k}(\zeta^{\dagger}))=\frac{1}{n}\sum_{j=0}^{n-1}\sum_{z}\hat{f}_{*}^{j}\hat{m}_{j,z}(\varphi_{x_{0}}(\Delta(f^{j}(z),\delta)\times\Omega(z,k))).\label{eq:measure1}
\end{equation}
where the sum is over all $z$ such $\Delta(f^{j}(z),\delta)\times\Omega(z,k)\times\left\{ j\right\} \subset\mathcal{P}_{k}(\zeta^{\dagger})$. The arguments used in Remark \ref{rem:representationofhypdisk} guarantee that
$
\varphi_{z}^{-1}\circ\hat{f}^{-j}\circ\varphi_{x_{0}}(\Delta(f^{j}(z),\delta)\times\Omega_{x_{0}}(z,k))=D(z,j,\delta)\times\Omega_{z},
$
 where $\Omega_{z}$ still is a cylinder set in $\Gamma$. 
Thus, the summand in \eqref{eq:measure1} is 
\begin{align*}
(Leb_{D(z,j,\delta)}\times\mathds{P})(D(z,j,\delta)\times\Omega_{z}) & =Leb_{D(z,j,\delta)}(D(z,j,\delta))\cdot\mathds{P}(\Omega_{z})\\
 & =f_{*}^{j}Leb_{D(z,j,\delta)}(\Delta(f^{j}(z),\delta))\cdot\mathds{P}(\Omega_{z}).
\end{align*}
 Then
$
\hat{f}_{*}^{j}\hat{m}_{j,z}(\varphi_{x_{0}}(\Delta(f^{j}(z),\delta)\times\Omega_{x_{0}}(z,k)))=f_{*}^{j}Leb_{D(z,j,\delta)}(\Delta(f^{j}(z),\delta))\cdot\mathds{P}(\Omega_{z}).
$
\begin{figure}[htb]
\centering
\includegraphics[scale=0.6]{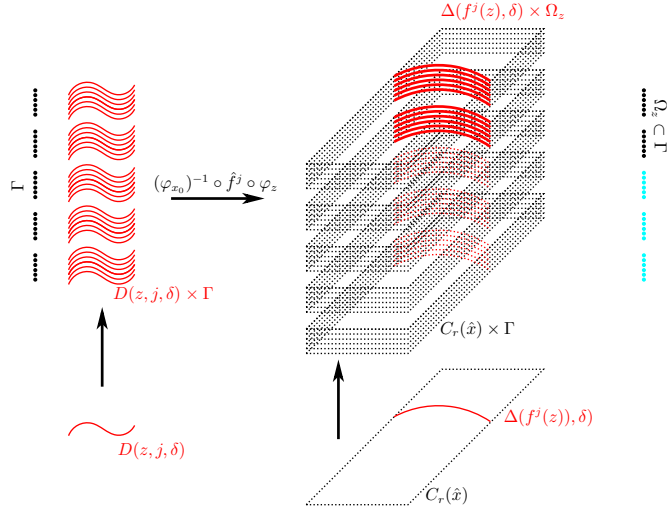}
\caption{Change of cylinders.}
\end{figure}
Denote by $H_{f^{j}(z)}^{s}$ the holonomy map from the hyperbolic disk $\Delta(f^{j}(z),\delta)$ to the disk $\Delta$.
Since $(H_{f^{j}(z)}^{s})^{-1}(B)\subset\Delta(f^{j}(z),\delta)$,
it is not hard to see that the same properties above hold for $(H_{f^{j}(z)}^{s})^{-1}(B)$
instead $\Delta(f^{j}(z),\delta)$ with the
same subsets of $\Gamma$. Then
$
\hat{f}_{*}^{j}\hat{m}_{j,z}(\varphi_{x_{0}}((H_{f^{j}(z)}^{s})^{-1}(B)\times\Omega_{x_{0}}(z,k)))=f_{*}^{j}Leb_{D(z,j,\delta)}((H_{f^{j}(z)}^{s})^{-1}(B))\cdot\mathds{P}(\Omega_{z}).
$
By bounded distortion (recall Proposition \ref{prop:distorcionhyptime})
there exists $C_{1}>1$ such that 
\[
C_{1}^{-1}Leb_{\Delta(f^{j}(z),\delta)}((H^{s})^{-1}(B))\leq f_{*}^{j}Leb_{D(z,j,\delta)}((H^{s})^{-1}(B))\leq C_{1}Leb_{\Delta(f^{j}(z),\delta)}((H^{s})^{-1}(B)).
\]
Since the stable holonomy $H_{f^{j}(z)}^{s}$ is absolutely continuous
with a Jacobian bounded away from zero and infinity (cf. \cite[Theorem V.8.1]{QXZ09}), there exists 
$C_{2}>1$ so that
$
C_{2}^{-1}Leb_{\Delta}(B)\leq Leb_{\Delta(f^{j}(z),\delta)}((H^{s})^{-1}(B))\leq C_{2}Leb_{\Delta}(B)
$
and, consequently, 
$
T_{1}^{-1}T_{2}^{-1}Leb_{\Delta}(B)\leq f_{*}^{j}Leb_{D(z,j,\delta)}((H^{s})^{-1}(B))\leq T_{1}T_{2}Leb_{\Delta}(B).
$
This proves that 
\begin{align*}
C_{*}^{-1}\frac{Leb_{\Delta}(B)}{Leb(\Delta)}  
	& \le \frac{\hat{f}_{*}^{j}\hat{m}_{j,z}(\varphi_{x_{0}}((H_{f^{j}(z)}^{s})^{-1}(B)\times\Omega_{x_{0}}(z,k)))}{\hat{f}_{*}^{j}\hat{m}_{j,z}(\varphi_{x_{0}}(\Delta(f^{j}(z),\delta)\times\Omega_{x_{0}}(z,k)))} \\
	&=\frac{f_{*}^{j}Leb_{D(z,j,\delta)}((H_{f^{j}(z)}^{s})^{-1}(B))}{f_{*}^{j}Leb_{D(z,j,\delta)}(\Delta(f^{j}(z),\delta))} \
	 \leq C_{*}\frac{Leb_{\Delta}(B)}{Leb(\Delta)}
\end{align*}
where $C_{*}=(T_{1}T_{2})^{2}$.
Hence we conclude that 
\begin{alignat*}{1}
\xi_{n}^{\dagger}((p^{\dagger})^{-1}(B)\cap\mathcal{P}_{k}(\zeta^{\dagger})) & =\frac{1}{n}\sum_{j=0}^{n-1}\sum_{z}\hat{f}_{*}^{j}\hat{m}_{j,z}(\varphi_{x_{0}}((H^{s})^{-1}(B)\times\Omega_{x_{0}}(z,k)))\\
 & \leq C_{*}\frac{Leb_{\Delta}(B)}{Leb(\Delta)}\frac{1}{n}\sum_{j=0}^{n-1}\sum_{z}\hat{f}_{*}^{j}\hat{m}_{j,z}(\varphi_{x_{0}}(\Delta(f^{j}(z),\delta)\times\Omega_{x_{0}}(z,k)))\\
 & =C_{*}\frac{Leb_{\Delta}(B)}{Leb(\Delta)}\xi_{n}^{\dagger}(\mathcal{P}_{k}^{\dagger}(\zeta^{\dagger}))
\end{alignat*}
for every $B\subset \Delta$ (and the lower bound is obtained analogously).
Thus, \eqref{eq:lCdaggerS} holds with $C=\frac{C_{*}}{Leb(\Delta)}$.

To conclude the proof of the theorem, note that if $\lim_{k\to\infty}\hat{\nu}_{n_{k}}=\hat{\nu}$
in the weak{*} topology then $\lim_{k\to\infty}\xi_{n_{k}}^{\dagger}=\xi^{\dagger}$ and $supp(\xi^{\dagger})\subset\cup_{\hat{\Delta}\in\hat{K}_{\infty}}\hat{\Delta}\times\left\{ \infty\right\}$.
Moreover,  $\xi^{\dagger}(\Upsilon\times\left\{ \infty\right\} )=\hat{\nu}(\Upsilon)$.
Now, using that $\cap_{k}\mathcal{P}_{k}(\hat{\zeta},\infty)=\hat{\Delta}\times\left\{ \infty\right\} $
where $\hat{\Delta}\times\left\{ \infty\right\} $ is the element containing $\zeta$, and taking $n\to\infty$ at 
\eqref{eq:lCdaggerS} we conclude that 
\[
C^{-1}Leb_{\Delta}(B)\leq\frac{\xi^{\dagger}((p^{\dagger})^{-1}(B)\cap\mathcal{P}_{k}(\hat{\zeta},\infty))}{\xi^{\dagger}(\mathcal{P}_{k}(\hat{\zeta},\infty))}\leq CLeb_{\Delta}(B).
\]
Since $B\subset \Delta$ is arbitrary we conclude the conditional measure of $\xi^{\dagger}$ in $\hat{\Delta}\times\left\{ \infty\right\} $
is absolutely continuous with respect to $Leb_{\Delta}$. Equivalently, 
 $\pi_{*}\hat{\nu}_{\hat{\Delta}}$ is absolutely continuous
with respect to $Leb_{\Delta}$. This finishes the proof of the theorem.
\end{proof}

\subsection{Finiteness and uniqueness of SRB measures\label{sec:ergodicityfinitiness}}

In this subsection we first prove that some ergodic component of the measure 
$
\mu=\lim_{k\to\infty}\frac{1}{n_{k}}\sum_{j=0}^{n_{k}-1}f_{*}^{j}Leb_{D},
$
is an SRB measure. Up to normalization we assume, without loss
of generality, that the limiting measure $\mu$ described above is a probability measure. 
Let $\hat\mu$ denote the unique $\hat f$-invariant lift of $\mu$.
Write  $\mu=\nu+\eta$, where $\nu$ is an accumulation point 
of the sequence of measures $(\nu_n)_n$ defined by
\eqref{eq:hyp_subprobability_n} on page \pageref{eq:hyp_subprobability_n},
and $\eta=\mu-\nu$. By some abuse of notation we assume
$\nu=\lim_{k\to\infty}\frac{1}{n_{k}}\sum_{j=0}^{n_{k}-1}f_{*}^{j}Leb_{\mathcal{D}_{j}}$ and assume
$\hat\nu$ is the lift of $\nu$ built along Subsection~\ref{sec:hatnuconstruction}. 
In particular  $\pi_{*}\hat{\nu}_{\hat{\Delta}}\ll Leb_{\pi(\hat{\Delta})}$
for every $\hat{\Delta}\in\hat{K}_{\infty}$, where $\left\{ \hat{\nu}_{\hat{\Delta}}\right\} _{\hat{\Delta}\in\hat{K}_{\infty}}$
is a disintegration of $\hat \nu$ with respect to $\hat{K}_{\infty}$.
Consider the set $R(\hat{f})$ of regular points, that is, the set of points $\hat{z}\in M^{f}$ such that the limits 
$\lim_{n\to\infty}\frac{1}{n}\sum_{j=0}^{n-1}\hat{\phi}(\hat{f}^{-j}(\hat{z}))$ and $\lim_{n\to\infty}\frac{1}{n}\sum_{j=0}^{n-1}\hat{\phi}(\hat{f}^{j}(\hat{z}))$ exist and coincide for every continuous $\hat\phi : M^f\to \mathbb R$. 
The set $R(\hat{f})$ has $\hat\mu$-full measure as a consequence of the ergodic theorem 
(see e.g. \cite[Corollary II.1.4]{Mane87}).
Given a measurable partition $\mathcal{P}$ on $K_\infty$ and the Borelian measure $\nu$, let $\tilde{\nu}$ denote the quotient measure induced by $\mathcal{P}$.

\begin{lemma}
\label{lem:birkhoffconstantdelta} For $\tilde{\nu}$-almost every
$\Delta\in\mathcal{K}_{\infty}$ and every $\phi\in C(M)$ there exists $L_{\Delta}(\phi)\in\mathbb{R}$ such that 
$
\lim_{n\to\infty}\frac{1}{n}\sum_{j=0}^{n-1}\phi(f^{j}(x))=L_{\Delta}(\phi)
\quad \text{for $Leb_{\Delta}$-almost every $x\in\Delta$}.
$
\end{lemma}

\begin{proof}
Since $\hat{\mu}(R(\hat{f}))=1$ we get $\hat{\nu}(M^{f}\backslash R(\hat{f}))=0$.
In consequence, $\hat{\nu}_{\hat{\Delta}}(M^{f}\backslash R(\hat{f}))=0$ for $\tilde \nu$-almost every 
$\hat{\Delta}\in\hat{\mathcal{K}}_{\infty}$.
Fix $\hat{\Delta}\in\hat{\mathcal{K}}_{\infty}$ and consider the sets
$R_{\hat{\Delta}}(\hat{f}):=R(\hat{f}) \cap \hat\Delta$
 and 
\begin{equation}\label{eq:RDelta}
R_{\Delta}:=\pi(R_{\hat{\Delta}}(\hat{f})).
\end{equation}
Then it is clear that 
$\hat{\nu}_{\hat{\Delta}}(M^{f}\backslash R_{\hat{\Delta}}(\hat{f}))=0$
for almost every $\hat{\Delta}\in\hat{\mathcal{K}}_{\infty}$ and, consequently, $R_{\Delta}$
is a $\pi_{*}\hat{\nu}_{\hat{\Delta}}$-full measure subset. In particular
$R_{\Delta}$ is a full $Leb_{\Delta}$-measure subset, because
$\pi_{*}\hat{\nu}_{\hat{\Delta}}$ is absolutely continuous with respect to $Leb_{\Delta}$ with a density
bounded away from zero and infinity (cf. Theorem \ref{thm:abscontdisint}).

Given a continuous observable $\phi : M \to \mathbb R$ we get $\phi\circ\pi\in C(M^{f})$.
Moreover, if $x,y\in R_{\Delta}$ there are $\hat{x},\hat{y}\in R_{\hat{\Delta}}(\hat{f})$ so that
$\pi(\hat{x})=x$ and $\pi(\hat{y})=y$ and the limits
\begin{equation}\label{eq:int134}
\lim_{n\to\infty}\frac{1}{n}\sum_{j=0}^{n-1}\phi(x_{-j})=\lim_{n\to\infty}\frac{1}{n}\sum_{j=0}^{n-1}\phi\circ\pi(\hat{f}^{-j}(\hat{x}))=\lim_{n\to\infty}\frac{1}{n}\sum_{j=0}^{n-1}\phi\circ\pi(\hat{f}^{j}(\hat{x}))
\end{equation}
and
\begin{equation}\label{eq:int1346}
\lim_{n\to\infty}\frac{1}{n}\sum_{j=0}^{n-1}\phi(y_{-j})=\lim_{n\to\infty}\frac{1}{n}\sum_{j=0}^{n-1}\phi\circ\pi(\hat{f}^{-j}(\hat{y}))=\lim_{n\to\infty}\frac{1}{n}\sum_{j=0}^{n-1}\phi\circ\pi(\hat{f}^{j}(\hat{y}))
\end{equation}
do exist. Now, as the points $\hat{x},\hat{y}$ belong to $\hat{\Delta}$ we get that $\lim_{j\to\infty}d(x_{-j},y_{-j})=0$
and the limits in ~\eqref{eq:int134} and ~\eqref{eq:int1346} coincide.
On the other hand, since $\hat{x},\hat{y}\in R(\hat{f})$, we get that
$
\lim_{n\to\infty}\frac{1}{n}\sum_{j=0}^{n-1}\phi\circ\pi(\hat{f}^{j}(\hat{x}))$ and $\lim_{n\to\infty}\frac{1}{n}\sum_{j=0}^{n-1}\phi\circ\pi(\hat{f}^{j}(\hat{y}))$
coincide.
Thus, using that $f^{j}\circ\pi=\pi\circ\hat{f}^{j}$, we also conclude that
$
\lim_{n\to\infty}\frac{1}{n}\sum_{j=0}^{n-1}\phi(f^{j}(x))$ and $\lim_{n\to\infty}\frac{1}{n}\sum_{j=0}^{n-1}\phi(f^{j}(y))
$
are equal.
As $x$ and $y$ were chosen arbitrarily, we have that $\lim_{n\to\infty}\frac{1}{n}\sum_{j=0}^{n-1}\phi(f^{j}(x))$ is constant to some value $L_{\Delta}(\phi)$ for every $x\in R_{\Delta}$. Finally, since $Leb_{\Delta}(\Delta\backslash R_{\Delta})=0$
this completes the proof of the lemma.
\end{proof}

We are now in position to guarantee the existence of SRB measures.

\begin{proposition}\label{prop:existsrb}
The $f$-invariant probability measure $\mu=\nu+\eta$ has some ergodic component
$\mu_{*}$ that is supported on $\Lambda$, is a hyperbolic measure 
and an SRB measure. 
\end{proposition}

\begin{proof}
Consider $R=\bigcup_{\Delta\in\mathcal{K}_{\infty}}R_{\Delta}$ (recall ~\eqref{eq:RDelta}) and set
$\mathcal{R}=\bigcup_{n\in\mathbb{N}}f^{-n}(R)$. We have
that $\mathcal{R}$ is an $f$-invariant subset: $\mu(\mathcal{R}\cap f^{-1}(\mathcal{R}))=0$
by Poincar\'e's recurrence theorem.
We claim that for every $\phi\in C(M)$ there is $L(\phi)\in\mathbb{R}$ 
such that
$
\lim_{n\to\infty}\frac{1}{n}\sum_{j=0}^{n-1}\phi(f^{j}(x))=L(\phi)
$
 for every $x\in R$, which means that the constant given by Lemma~\ref{lem:birkhoffconstantdelta}
does not depend on the disk $\Delta$.
This part of the proof follows the ideas from the Hopf's argument for ergodicity.
By Lemma \ref{lem:birkhoffconstantdelta}, there exists
$L_{\Delta}(\phi)\in\mathbb{R}$ such that 
$
\lim_{n\to\infty}\frac{1}{n}\sum_{j=0}^{n-1}\phi(f^{j}(x))=L_{\Delta}(\phi)
\quad \text{for every $x\in R_{\Delta}$.}
$
Given $\Delta,\Delta'\in\mathcal{K}_{\infty}$, consider the
holonomy $H^{s}:\Delta\rightarrow\Delta'$ along the stable foliation. Since $R_{\Delta}$ and
$R_{\Delta'}$ are full measure sets for $Leb_{\Delta}$ and $Leb_{\Delta'}$
, respectively and the stable foliation is absolutely continuous
(\cite[Theorem V.8.1]{QXZ09}) we have that $H^{s}(R_{\Delta})=R_{\Delta'}$
modulo a zero $Leb_{\Delta}$ measure subset. As $\tilde{\phi}(x):=\lim_{n\to\infty}\frac{1}{n}\sum_{j=0}^{n-1}\phi(f^{j}(x))$
(whenever well defined) is constant along stable manifolds, then we
get that $L_{\Delta}(\phi)=L_{\Delta'}(\phi)=L(\phi)$,
for any $\Delta,\Delta'\in\mathcal{K}_{\infty}$, which proves the claim. 
Now, given $x\in\mathcal{R}$, there is $n_{0}(x)\in\mathbb{N}$ such that $f^{n_{0}}(x)\in R$.
In particular, if $\phi\in C(M)$ is arbitrary, we get
\begin{equation}
\lim_{n\to\infty}\frac{1}{n}\sum_{j=0}^{n-1}\phi(f^{j}(x))=\lim_{n\to\infty}\frac{1}{n}\sum_{j=0}^{n-1}\phi(f^{j}(f^{n_{0}}(x)))=L(\phi)
\quad \text{for every $x\in\mathcal{R}$}.
\label{eq:birkhoffconstR}
\end{equation}
Moreover,
$
\mu(\mathcal{R})\geq\mu(R)\geq\nu(R)=\pi_{*}\hat{\nu}(R)=\hat{\nu}(\pi^{-1}(R))\geq\hat{\nu}(R(\hat{f}))=\hat{\nu}(\hat{K}_{\infty})\geq\alpha>0.
$
Thus one can define $\mu_{*}:=\mu|_{\mathcal{R}}$, where $\mu|_{\mathcal{R}}(A)=\frac{\mu(A\cap\mathcal{R})}{\mu(\mathcal{R})}$,
for every measurable set $A\subset M$. It is clear that $\mu_{*}$ is a positive $f$-invariant
probability measure 
and, by \eqref{eq:birkhoffconstR}, it is ergodic.

We now prove that $\mu_*$ is a hyperbolic measure. One can write 
$\mu_{*}=\mu|_{\mathcal{R}}=\nu|_{\mathcal{R}}+\eta|_{\mathcal{R}}$.
As consequence of Proposition \ref{prop:hatnuishyp}, for every
$\Delta\in\mathcal{K}_{\infty}$ and $y\in\Delta$ there is
a pre-orbit $\hat{y}$ of $y$ that admits a $D\hat{f}^{-1}$-invariant
splitting $T_{\hat{y}}=E_{y}^{s}\oplus\mathcal{E}_{\hat{y}}^{u}$
where $E_{y}^{s}$ is uniformly contracting and $\mathcal{E}_{\hat{y}}^{u}$
is uniformly contracting for the past along the pre-orbit $\hat{y}$.
In particular, given $y\in R$ let $\hat{y}\in M^{f}$ be such pre-orbit
and $T_{\hat{y}}=E_{y}^{s}\oplus\mathcal{E}_{\hat{y}}^{u}$ be the corresponding splitting.
Since $\mu_*$ is ergodic and $\mu_{*}(R)>\nu(R)>0$ then, by Proposition~\ref{prop:oseledet_nat_ext},
 the Lyapunov exponents of $\mu_*$ coincide with the Lyapunov exponents at typical points $\hat y$ 
which are all non-zero. Thus $\mu_{*}$ is a hyperbolic measure.

We are left to show that $\mu_{*}$ is an SRB measure. 
It is easy to check that $\hat{\mu}_{*}:=\hat{\mu}|_{\pi^{-1}(\mathcal{R})}$ is the 
(uniquely defined) lift of $\mu_{*}$. Moreover, we have the following:

\begin{claim}
If $\{ \hat{\mu}_{*,\hat{\Delta}}\} _{\hat{\Delta}\in\hat{\mathcal{K}}_{\infty}}$
is a disintegration of $\hat{\mu}_{*}$ with respect to the partition $\hat{\mathcal{K}}_{\infty}$ then 
$\pi_{*}\hat{\mu}_{*,\hat{\Delta}}\ll Leb_{\Delta}$,
for $\tilde{\hat{\mu}}_{*}$-almost every $\hat{\Delta}\in\mathcal{\hat{K}}_{\infty}$.
\end{claim}

\begin{proof}[Proof of the claim] For each $\hat{\Delta}\in\hat{\mathcal{K}}_{\infty}$
we have that $\pi|_{\hat{\Delta}}:\hat{\Delta}\rightarrow\Delta$
is a homeomorphism.  
Taking $\hat{m}_{\hat{\Delta}}:=(\pi|_{\hat{\Delta}})_{*}^{-1}Leb_{\Delta}$ we have that 
$\pi_{*}\hat{\mu}_{*,\hat{\Delta}}\ll Leb_{\Delta}$ if, and
only if $\hat{\mu}_{*,\hat{\Delta}}\ll\hat{m}_{\hat{\Delta}}$. 
Hence we will prove the second.
Assume, by contradiction, that there exists $\hat{A}\subset\hat{K}_{\infty}$ so
that $\hat{\mu}_{*}(\hat{A})>0$ but $\hat{m}_{\hat{\Delta}}(\hat{A}\cap\hat{\Delta})=0$
for $\tilde{\hat{\mu}}_{*}$-almost every $\hat{\Delta}\in\hat{\mathcal{K}}_{\infty}$. 
As $\hat{\mu}_{*}(\hat{A})>0$ we get that $\hat{\mu}_{*,\hat{\Delta}}(\hat{A}\cap\hat{\Delta})>0$
for every $\hat{\Delta}$ in some subset $\hat{\mathcal{K}}_{\infty}^{'}\subset\hat{\mathcal{K}}_{\infty}$
with positive $\tilde{\hat{\mu}}_{*}$-measure. Consider the
saturared set $\hat{B}=\bigcup_{j=-\infty}^{\infty}\hat{f}^{j}(\hat{A})$, which is $\hat{f}$-invariant subset and, by ergodicity, 
has full $\hat{\mu}_{*}$-measure.
Then
$\hat{\mu}_{*,\hat{\Delta}}(\hat{B}\cap\hat{\Delta})=1$, for $\tilde{\hat{\mu}}_{*}$-almost every 
$\hat{\Delta}\in\mathcal{\hat{K}}_{\infty}$.
On the other hand, using that $\hat{m}_{\hat{\Delta}}(\hat{A}\cap\hat{\Delta})=0$
for $\tilde{\hat{\mu}}_{*}$-almost every $\hat{\Delta}\in\hat{\mathcal{K}}_{\infty}$ and that
 $\hat{f}_{*}^{j}\hat{m}_{\hat{\Delta}}\ll\hat{m}_{\hat f^j (\hat{\Delta})}$
for every $j\in\mathbb{Z}$,
we conclude that
$\hat{m}_{\hat{\Delta}}(\hat{B}\cap\hat{\Delta})=0$
for $\tilde{\hat{\mu}}_{*}$-almost every $\hat{\Delta}\in\hat{\mathcal{K}}_{\infty}$.
Hence $\hat{\mu}_{*,\hat{\Delta}}\perp\hat{m}_{\hat{\Delta}}$
for $\tilde{\hat{\mu}}_{*}$-almost every $\hat{\Delta}\in\mathcal{\hat{K}}_{\infty}$.
However, this contradicts the fact that $\hat{\mu}_{*,\hat{\Delta}}=\hat{\nu}_{\hat{\Delta}}+\hat{\eta}_{\hat{\Delta}}$ on $\hat{\mathcal{K}}_{\infty}$, with $\hat{\nu}_{\hat{\Delta}}\ll\hat{m}_{\hat{\Delta}}$ 
(cf. Theorem~\ref{thm:abscontdisint}).
This proves the claim.
\end{proof}

We have that the partition $\mathcal{P}:=\bigvee_{j\in\mathbb{N}}\hat{f}^{-j}\hat{\mathcal{K}}_{\infty}$
of $\mathcal{R}$ is subordinated to unstable manifolds (see e.g. \cite[Proposition VII.2.4 ]{QXZ09}) 
and, the claim assures that $\pi_{*}\hat{\mu}_{*,\hat{P}}\ll Leb_{\pi(\hat{P})}$ for almost all $\hat{P}\in\mathcal{P}$.
Now, if $\mathcal{Q}$ is any other partition subordinated to unstable manifolds
then the same property holds for $\mathcal{Q\vee P}$. Then, since every element of $\mathcal{Q}$
is obtained as sum of elements of $\mathcal{Q\vee P}$ we conclude that $\mu_{*}$ is an SRB measure.
This completes the proof of the proposition.
\end{proof}

In order to conclude the finiteness of SRB measures that cover $H$ we prove that each of their basins of attraction occupies a definite
proportion of the phase space. This is done in two steps.

\begin{lemma}\label{lem:finitenessofsrb}
There exists $b>0$ such that the following holds: if $\mu_{*}$ is an SRB measure given by 
Proposition~\ref{prop:existsrb} then there exists an open set $V\subset M$ contained in $B(\mu_{*})$
(mod zero) so that $Leb(U)>b$.
\end{lemma}

\begin{proof}
Note that $\hat{\mu}_{*}(B(\hat{\mu}_{*}))=1$. 
In particular, there is $\hat{\Delta}\in\mathcal{\hat{K}}_{\infty}$ such that $\hat{\mu}_{*,\hat{\Delta}}$
almost every $\hat{x}$ in $\hat{\Delta}$ belongs to $B(\hat{\mu}_{*})$. 
Consequently, $\hat{m}_{\hat{\Delta}}$-almost every $\hat{x}\in\hat{\Delta}$
belongs to $B(\hat{\mu}_{*})$. 
Moreover, $\pi(B(\hat{\mu}_{*}))\subset B(\mu_{*})$.
Thus $Leb_{\Delta}$-almost every point $x\in\Delta$ belongs to $B(\mu_{*})$.
The union of stable manifolds of points in $B(\mu_{*})\cap\Delta$
is also contained in $B(\mu_{*})$, because the Birkhoff's
average converges to the same limit along stable manifolds. Using
the absolute continuity of the stable foliation and that
stable manifolds have size uniformly bounded away from zero, we conclude
that $V=\bigcup_{y\in\Delta}W_{r}^{s}(y)$ is an open set
contained in $B(\mu_{*})$, except for a zero measure subset. 
Thus $Leb(V)$ is uniformly bounded away from zero, which proves the lemma.
\end{proof}

Actually every SRB measure can be obtained by the procedure described along Section~\ref{sec:existence_uniqueness}. 
Indeed:

\begin{lemma}
If $\mu$ is an SRB measure for $f\mid_{\Lambda}$ then 
there exists a disk $D\subset U$ such that $Leb_{D}(D\cap B(\mu)\cap H)=Leb_{D}(D)$.\label{lem:exist_good_disks}
\end{lemma}

\begin{proof}
Let $\hat{\mathcal{P}}$ be a partition of $M^{f}$ subordinated to unstable
manifolds of $\mu$. Then each atom of $\hat{\mathcal{P}}$
projects over a disk in some unstable manifold. The lift $\hat{\mu}$
of $\mu$ is an ergodic measure and is such $\hat{\mu}(B(\hat{\mu}))=1$.
In particular, if $\{ \hat{\mu}_{g,\hat{P}}\} _{\hat{P}\in\hat{\mathcal{P}}}$
is a disintegration of $\hat\mu_g$ with respect to $\hat{\mathcal{P}}$
then $\hat{\mu}_{\hat{P}}(B(\hat{\mu}))=\hat{\mu}_{\hat{P}}(\hat{P})$
for $\tilde{\hat \mu}$-almost every $\hat{P}\in\hat{\mathcal{P}}$.
By the latter and the fact that $\mu$ is an SRB measure, there exists $\hat{P}\in\hat{\mathcal{P}}$ such that 
$\hat{\mu}_{\hat{P}}(B(\hat{\mu}))=\hat{\mu}_{\hat{P}}(\hat{P})$ and that $\mu_{D}:=\pi_{*}\hat{\mu}_{\hat{P}}\ll Leb_{D}$, where $D:=\pi(\hat{P})$. 
Notice that $\mu_{D}(B(\mu)\cap D)=\mu_{D}(D)$, because
$
\mu_{D}(B(\mu))  =\pi_{*}\hat{\mu}_{\hat{P}}(B(\mu)) 
 \geq\hat{\mu}_{\hat{P}}(B(\hat{\mu}))=\hat{\mu}_{\hat{P}}(\hat{P})=\mu_{D}(D).
$
This proves that $\mu_{D}$-almost every $x\in D$ belongs to $B(\mu)$ and, since $\mu_D$ and $Leb_{D}$ 
are equivalent, $Leb_{D}$-almost every $x\in D$ belongs to $B(\mu)$.
Consider $\mathcal{C}=\bigcup_{y\in D}W_{loc}^{s}(y)$.
Since $Leb(H)>0$ and the holonomy along the stable foliation is absolutely
continuous we conclude that there exists a disk
$\tilde{D}$ crossing $\mathcal{C}$ such that $Leb_{\tilde{D}}(H)=Leb_{\tilde{D}}(\tilde{D})$,
which proves the lemma.
\end{proof}

\begin{corollary}
\label{prop:SRB_is_accpoint_averageLeb} 
If $\mu$  is an SRB measure for $f\mid_{\Lambda}$ then there is a
disk $D$ tangent to the cone field $\mathcal C$ so that 
$
\omega_{n}=\frac{1}{n}\sum_{j=0}^{n-1}f_{*}^{j}Leb_{D}
$
is convergent to $\mu$ (in the weak{*} topology), where $Leb_{D}$ is the normalized Lebesgue measure in a disk 
$D$ for which there are constants $C>0$, $\xi\in(0,1)$ and $c>0$ such that:
\begin{enumerate}
\item given $k\ge 1$ and an injectivity domain $\mathcal{D}_{k} \subset D$ for
$f^{k}$, the map 
\[
f^{j}(\mathcal{D}_{k})\ni x\mapsto\log\left|\det(Df(x)|_{T_{x}f^{j}\mathcal{D}_{k}})\right|
\]
is $(L,\xi)$-H\"older continuous for every $0\leq j\leq k-1$.
\item $\limsup_{n\to\infty}\frac{1}{n}\sum_{j=0}^{n-1}\log\big\Vert (Df(f^{j}(x))|_{T_{f^{j}(x)}f^{j}D})^{-1}\big\Vert 
	\leq-2c<0$
for $Leb_{D}$-almost every $x\in D$.
\end{enumerate}
Moreover the constants $c$ and $\xi$ depend only of the constants in (H1)-(H4). 
\end{corollary}

\begin{proof}
Let $D\subset M$ be a disk given by Lemma~\ref{lem:exist_good_disks}.
Proposition \ref{prop:iteratesoftgbundleholder} implies that $(1)$ holds.
Property $(2)$ is immediate from the fact that $Leb_{D}(H)=Leb_{D}(D)$
(cf. Lemma \ref{lem:exist_good_disks}). We are left to prove that $\mu$ is the
limit of the C\'esaro averages
$
\omega_{n}=\frac{1}{n}\sum_{j=0}^{n-1}f_{*}^{j}Leb_{D},
$
$n\in\mathbb{N}$. Assume that $\omega=\lim_{k\to\infty}\omega_{n_{k}}$ is an accumulation
point of $(\omega_{n})_n$. 
Given an arbitrary $\phi\in C(M)$, using the dominated convergence theorem and the fact that $Leb_{D}(B(\mu))=Leb_{D}(D)$,
we conclude that
\begin{align*}
\int\phi \,d\omega & 
	=\lim_{k\to\infty}\int\frac{1}{n_{k}}\sum_{j=0}^{n_{k}-1}\phi\circ f^{j} \, dLeb_{D} 
 	 =\int\lim_{k\to\infty}\frac{1}{n_{k}}\sum_{j=0}^{n_{k}-1}\phi\circ f^{j} \, dLeb_{D}=\int\phi \,d\mu.
\end{align*}
As $\phi$ was taken arbitrarily, it follows that $\omega=\mu$.
Hence $(\omega_n)_n$ is convergent to the SRB measure $\mu_g$.
Finally recall that the constants $c\text{ and }\xi$, given by Proposition~\ref{prop:iteratesoftgbundleholder},  
depend only on the cone field $C_a$ 
and (H3).
\end{proof}

Finally, to complete the proof of Theorem~\ref{thm:TEO1} we need only show uniqueness of the
SRB measure in the case the attractor is transitive.
By Lemma \ref{lem:finitenessofsrb}, as basins of attraction of
different SRB measures have empty intersection,  there is a finite number of SRB measures for $f$.
As each one of them contains
a set which has Lebesgue measure bounded away from zero, we must have
a finite number of them.
Moreover, if $f$ is transitive the SRB measure is unique. Indeed, if $\mu_{1}$
and $\mu_{2}$ are SRB measures given by Lemma \ref{lem:finitenessofsrb} 
there are open sets $V_{1}\subset B(\mu_{1})$
and $V_{2}\subset B(\mu_{2}) \, (\!\!\!\!\mod0)$.
The transitivity of $f$ implies that exist $N_{0}\in\mathbb{N}$ such that $f^{N_{0}}(V_{1})\cap V_{2}\neq\emptyset$.
Taking a smaller open set, if necessary, we may assume that $f^{N_{0}}|_{U_{1}}$
is a diffeomorphism onto its image, and so $f^{N_{0}}(V_{1})\cap V_{2} \subset B(\mu_{2})$
is an open set $\, (\!\!\!\!~\mod0)$.
Similarly, $f^{-N_{0}}(f^{N_{0}}(U_{1})\cap U_{2})\cap U_{1}\subset B(\mu_{1})$ $\, (\!\!\!\!\mod0)$.
Thus there exists $x\in M$ so that $x\in B(\mu_{1})$ and $f^{N_{0}}(x)\in B(\mu_{2})$, which implies that  
$\mu_{1}=\mu_{2}$. Thus, if $f$ is transitive the SRB measure is
unique.

\section{Statistical Stability\label{sec:stability}}

The main goal of this section is the proof of Theorem~\ref{thm:TEO2}.
More precisely, in this section we prove that if (H1), (H2), (H3) and (H4) hold robustly with uniform constants on the attractors
of an open set $\mathcal{V}$ of $C^{1+\alpha}$ endomorphisms then the corresponding  SRB measures vary continuously (in the weak$^*$ topology). By continuous variation we mean that if $(f_{n})_{n\in\mathbb{N}}$
is a sequence in $\mathcal{V}$ converging to $f\in\mathcal{V}$ and
if $(\mu_{n})_{n\in\mathbb{N}}$ is a sequence of probability
measures where $\mu_{n}$ is a SRB measure for $f_{n}$, for every $n\in\mathbb{N}$,
then any accumulation point of $(\mu_{n})_{n\in\mathbb{N}}$
in the weak{*} topology is a convex sum of the SRB measures for $f$.

First we establish the context.  
Denote by $End^{1+\alpha}(M)$ the set of non-singular endomorphisms of class $C^{1+\alpha}$ in $M$. 
Assume that $\Lambda$ is an attractor for $f$ and that $U\supset\Lambda$
an open neighborhood of $\Lambda$ such that $\cap_{n\geq0}f^{n}(\overline{U})=\Lambda$.
For every $g \in End^{1+\alpha}(M)$ that is $C^1$-close to $f$ define $\Lambda_{g}:=\bigcap_{n\geq0}g^{n}(\overline{U})$
is an attractor for $g$. 
Assume that there are an open neighborhood $\mathcal{V}\subset End^{1+\alpha}(M)$ of $f$,
constants $\lambda\in(0,1)$, $c>0$ and a family of cone fields $U\ni x\rightarrow C(x)$ 
of constant dimension $0<d_{u}\leq dim(M)$,
such that (H1)-(H4) holds with uniform constants on $U$: for every $g\in\mathcal{V}$ there exists 
a continuous splitting $T_{U}M=E^{s}(g)\oplus F(g)$ so that
\begin{enumerate}
\item [(H1)]$Dg(x)\cdot E_{x}^{s}(g)=E_{g(x)}^{s}(g)$
for every $x\in U$;
\item [(H2)]$\|Dg^{n}(x)|_{E^{s}_x(g)}\|\leq\lambda^{n}$ for every $x\in U$ and $n\in\mathbb{N}$
\item [(H3)] the cone field $U\ni x\mapsto C(x)$ contains the subspace $F(g)$, and is such that 
$Dg(x) (C(x))\subseteq C(g(x))$
for every $x\in U$, and 
\begin{equation*}
\limsup_{n\to\infty}\frac{1}{n}\sum_{j=0}^{n-1}\log\|(Dg(g^{j}(x))|_{C(g^{j}(x))}^{-1})\|\leq-2c<0
\end{equation*}
for Lebesgue almost every $x\in U$; and
\item [(H4)]$\left\Vert Dg(x)\cdot v\right\Vert \, \left\Vert Dg(x)^{-1}\cdot w\right\Vert \leq\lambda\, \left\Vert v\right\Vert \cdot\left\Vert w\right\Vert $ for every $v\in E_{x}^{s}(g)$, all $w\in Dg(x) ( C(x))$
and $x\in U$.\label{enu:domination}
\end{enumerate}

We will say that $\mathcal{V}$ is an open set of partially hyperbolic $C^{1+\alpha}$-endomorphisms with
uniform constants. For each $g\in\mathcal{V}$, we denote by $H_{g}$
the subset of $U$ that consists of points satisfying the non-uniform hyperbolicity condition in 
(H3). As a direct consequence of Theorem \ref{thm:TEO1} and the construction
of SRB measures along Section~\ref{sec:existence_uniqueness} we have the following:

\begin{corollary}
\label{prop:uniform_construction} Every $g\in\mathcal{V}$ has
finitely many SRB measures for $g$ supported on $\Lambda_{g}$ and, if $g\mid_{\Lambda_{g}}$ is transitive then the $SRB$ measure
is unique. Moreover, every SRB measure $\mu_g$ for $g$ is 
the limit of the sequence of measures $(\frac{1}{n}\sum_{j=0}^{n-1}g_{*}^{j}Leb_{D})_n$ for some $C^1$-disk $D$
tangent to the cone field.
\end{corollary}

Given $g\in\mathcal{V}$, let $\mu_{g}$ be an SRB measure for $g$. 
In particular, for any partition $\hat{\mathcal{P}}$ of $M^{g}$
subordinated to unstable manifolds we have that 
$
\pi_{*}\hat{\mu}_{g,\hat{\mathcal{P}}(\hat{x})}\ll Leb_{\pi(\hat{\mathcal{P}}(\hat{x}))},
$
for $\hat{\mu}_{g}$ almost every $\hat{x}\in M^{g}$, where $\hat{\mu}_{g}$ is the lift of $\mu_{g}$ in $M^{g}$, 
$\hat{\mathcal{P}}(\hat{x})$ is the atom of $\hat{\mathcal{P}}$ that contains $\hat{x}$, $\hat{\mu}_{g,\hat{\mathcal{P}}(\hat{x})}$ is the element of the disintegration of $\hat{\mu}_{g}$ in $\hat{\mathcal{P}}(\hat{x})$
and the map $\pi=\pi^g : M^g \to M$ is the natural projection. 
By the Radon-Nikodym theorem there is a measurable function $\rho_{\hat{x}}:\pi(\hat{\mathcal{P}}(\hat{x}))\rightarrow\mathbb{R}$
that is positive $\pi_{*}\hat{\mu}_{g,\hat{\mathcal{P}}(\hat{x})}$-almost everywhere and so that 
\begin{equation}
\pi_{*}\hat{\mu}_{g,\hat{\mathcal{P}}(\hat{x})}(B)=\int_{B}\rho_{\hat{x}} \, dLeb_{\pi(\hat{\mathcal{P}}(\hat{x}))}\label{eq:densities}
\end{equation}
for every measurable subset $B\subset\pi(\hat{\mathcal{P}}(\hat{x}))$.
Recall $\pi\mid_{\hat{\mathcal{P}}(\hat{x})}:\hat{\mathcal{P}}(\hat{x})\rightarrow\pi(\hat{\mathcal{P}}(\hat{x}))$
is a bijection and that $\pi(\hat{\mathcal{P}}(\hat{x}))$
is contained in $W_{loc}^{u}(\hat{x})$. 
Define $\hat{m}_{\hat{\mathcal{P}}(\hat{x})}:=(\pi\mid_{\hat{\mathcal{P}}(\hat{x})})^{*}Leb_{\pi(\hat{\mathcal{P}}(\hat{x}))}$, where
$
(\pi\mid_{\hat{\mathcal{P}}(\hat{x})})^{*}Leb_{\pi(\hat{\mathcal{P}}(\hat{x}))}(\hat{A})=Leb_{\pi(\hat{\mathcal{P}}(\hat{x}))}(\pi\mid_{\hat{\mathcal{P}}(\hat{x})}(\hat{A}))
$
for every measurable subset $\hat{A}\subset\hat{\mathcal{P}}(\hat{x})$.
Observe that $\pi_{*}\hat{m}_{\hat{\mathcal{P}}(\hat{x})}=Leb_{\pi(\hat{\mathcal{P}}(\hat{x}))}$
and that $\pi_{*}\hat{\mu}_{g, \hat{\mathcal{P}}(\hat{x})}\ll Leb_{\pi(\hat{\mathcal{P}}(\hat{x}))}$
if, and only if $\hat{\mu}_{g,\hat{\mathcal{P}}(\hat{x})}\ll\hat{m}_{\hat{\mathcal{P}}(\hat{x})}$.
Again, by the Radon-Nikodym theorem there is a measurable function
$\hat{\rho}:\hat{\mathcal{P}}(\hat{x})\rightarrow\mathbb{R}$
that is positive in $\hat{\mu}_{g,\hat{\mathcal{P}}(\hat{x})}$
almost every point such that
$
\hat{\mu}_{g,\hat{\mathcal{P}}(\hat{x})}(\hat{B})=\int_{\hat{B}}\hat{\rho} \,d\hat{m}_{\hat{\mathcal{P}}(\hat{x})},
$
for every measurable subset $\hat{B}\subset\hat{\mathcal{P}}(\hat{x})$.

We are now in a position to prove Theorem~\ref{thm:TEO2} on the statistical stability of these SRB measures.
The proof, strongly inspired by \cite{Vas07}, relies on a uniform control of the structures
used in the construction of SRB measures for these partially hyperbolic endomorphisms.
A brief sketch of the proof is as follows.
First we prove that if $\mu_{g}$ is an SRB measure for $g\in\mathcal{V}$
then there exists a family of disjoint unstable disks whose size is
bounded away from zero. This bound is uniform for every $g\in\mathcal{V}$.
So, given a sequence $(g_{n})_{n\in\mathbb{N}}$ in $\mathcal{V}$
converging to $g\in\mathcal{V}$ and a sequence $(\mu_{n})_{n\in\mathbb{N}}$,
where $\mu_{n}$ is a SRB measure for $g_{n}$, 
and an accumulation point $\mu$ of $(\mu_{n})_{n\in\mathbb{N}}$ 
the family of unstable disks $\mathcal{K}_{n}$ associated to $g_{n}$ and $\mu_{n}$
accumulate in a family $\mathcal{K}$ of unstable disks for $g$ with respect to $\mu$.
Using that each $\mu_{n}$ is absolutely continuous with respect
to Lebesgue measure on the disks of $\mathcal{K}_{n}$ with uniform constants)
we will then show that $\mu$ is absolutely continuity of  with respect to Lebesgue 
on the disks of $\mathcal{K}$. 
A further argument will show that almost all ergodic components of $\mu$ are
absolutely continuous on the disks of $\mathcal{K}$ and, consequently, are SRB measures, 
proving that $\mu$ is an convex combination of the latter.

Given a disk $\Delta\subset M$  denote $C_{r}^{g}(\Delta):=\bigcup_{y\in\Delta}W_{r}^{s}(g,y)$,
where $W_{r}^{s}(g,y)$ is the disk of radius $r>0$ centered
at $y$ in the local stable manifold $W_{loc}^{s}(g,y)$. 
We say that a disk $\tilde{\Delta}$ of the same
dimension of $\Delta$ \emph{crosses} $C_{r}^{g}(\Delta)$ if
the stable holonomy $H^{s}:\tilde{\Delta}\rightarrow\Delta$ is a
diffeomorphism. Denote by $\mathcal{K}_{r}^{g}(\Delta)$
the family of local unstable manifolds that cross $C_{r}^{g}(\Delta)$,
and by $K_{r}^{g}(\Delta)$ the union of its elements.

\begin{lemma}\label{lem:eta}
There are $\delta>0$, $r>0$ and $\eta>0$ so that: for every
$g\in\mathcal{V}$ and for every SRB measure $\mu_{g}$ for $g\mid_{\Lambda_g}$, there
exists a disk $\Delta:=\Delta(g,\mu_{g})$ of radius $\delta>0$
such that 
$\mu_{g}(K_{r}^{g}(\Delta))>\eta$. 
\end{lemma}

\begin{proof}
Fix $g\in\mathcal{V}$ and an SRB measure $\mu_{g}$ for $g\mid_{\Lambda_g}$. By 
Corollary~\ref{prop:SRB_is_accpoint_averageLeb} and Proposition \ref{lem:suppdisk},
$\mu_{g}$ is absolutely continuous with respect to Lebesgue
on unstable disks, and $supp(\mu_{g})$ is contained in the union of cylinders 
$C_{r_{g}}^{g}(\Delta)$, where $\Delta$ is an unstable disk of radius $\delta(g)>0$.
As the constant $\delta>0$ is the size of hyperbolic disks of the non uniformly hyperbolic endomorphism,
this size depends only on the size of invertibility domain of the exponential map, the size of the domain of 
inverse branches and the uniform continuity of the map and these quantities vary continuously
with $g\in\mathcal{V}$ (recall Proposition \ref{prop:contraidist}), $\delta$ can be taken uniformly 
in the neighborhood $\mathcal{V}$.
Moreover, since the stable manifolds vary continuously with the dynamics, we
can take $r_{g}:=r>0$ for every $g\in\mathcal{V}$. In particular there exists $b_0\in \mathbb N$ such that 
$supp(\mu_g)$ is contained in at most $b_0$ cylinders  of the form $C_{r}^{g}(\Delta)$, with $\Delta$ unstable
disks of radius $\delta>0$. 
This assures that there exists a disk $\Delta=\Delta(g,\mu_{g})$
of radius $\delta$ such that $\mu_{g}(C_{r}^{g}(\Delta))\geq \eta$, where $\eta=\frac1{b_0}>0$.
This proves the lemma.
\end{proof}

Given $g\in\mathcal{V}$ and $\mu_{g}$, let $C_{g}:=C_{r}^{g}(\Delta)$ and 
$\mathcal{K}_{g}:=\mathcal{K}_{r}^{g}(\Delta)$ be the family of local unstable manifolds that cross 
$C_{r}^{g}(\Delta)$, where $\Delta$ is given by the previous lemma. 
Let $\hat{\mathcal{K}}_{g}$ be the family of the lifts of the elements of $\mathcal{K}_{g}$. We have the following result:

\begin{proposition} \label{prop:uniform_abs_cont}
There exists $\tau>0$ such that if $g\in\mathcal{V}$,  $\mu_{g}$ is an SRB measure for $g$, 
$\hat{\mu}_{g}$ is the lift of $\mu_{g}$ and $\{ \hat{\mu}_{g,\hat{\Delta}}\} _{\hat{\Delta}\in\hat{\mathcal{K}}_{g}}$ 
is a disintegration of $\hat{\mu}_{g}$ with respect to $\hat{\mathcal{K}}_{g}$ then 
\begin{equation}
\tau^{-1}Leb_{\Delta}(B)\leq\pi_{*}\hat{\mu}_{g,\hat{\Delta}}(B)\leq\tau Leb_{\Delta}(B)\label{eq:absolute_continuity}
\end{equation}
for every measurable subset $B\subset\Delta$ and 
almost every $\Delta\in\mathcal{K}_{g}$. 
\end{proposition}

\begin{proof}
Combining Corollary~\ref{prop:SRB_is_accpoint_averageLeb} and Theorem~\ref{thm:abscontdisint}, 
for every $g\in\mathcal{V}$ there exists $\tau(g)>1$ such that (\ref{eq:absolute_continuity})
holds. Observe that the constant $\tau(g)$ depends on 
the constant of bounded distortion along the images of the disk $D_{g}$ (which can be used to recover 
$\mu_{g}$ as in Corollary~\ref{prop:SRB_is_accpoint_averageLeb}) by iterates of $g$, and 
the regularity of the stable holonomy. The first constant 
comes from the regularity of the Jacobian given at Proposition \ref{prop:iteratesoftgbundleholder},
which can be taken uniformly in $\mathcal{V}$. Moreover, the stable holonomy varies continuously in 
$\mathcal{V}$ because the stable manifolds vary continuously with $g\in\mathcal{V}$.
Hence there exists a uniform constant $\tau>1$  so that (\ref{eq:absolute_continuity})
holds for every $g\in\mathcal{V}$. 
\end{proof}

\medskip

Now, fix a sequence $\left\{ g_{n}\right\} _{n\in\mathbb{N}}$ in $\mathcal{V}$ and assume that it converges to 
$g\in\mathcal{V}$ in the $C^{1+\alpha}$-topology.
Denote $\mathcal{K}_{n}:=\mathcal{K}_{g_{n}}$, $\mathcal{\hat{K}}_{n}:=\mathcal{\hat{K}}_{g_{n}}$
and $C_{n}:=C_{g_{n}}=C_{r}^{g_n}(\Delta_n)$ for each $n\in\mathbb{N}$. Assume that for every $n\in \mathbb N$ the probability measure 
$\mu_{n}$ is an SRB measure for $g_{n}$ and that $\mu$ is an accumulation point of $(\mu_{n})_{n\in\mathbb{N}}$. 
It is clear that $\mu$ is a $g$-invariant probability.

Recall that each $C_{n}$ is a cylinder formed by the union of stable manifolds passing
through points of a disk $\Delta_{n}$ of radius $\delta$ tangent to the cone field $\cC$.
By compactness of $M$ and the Arzela-Ascoli theorem we may assume, up to consider 
some subsequence, that $\{\Delta_{n}\}_n$ converges to a disk $\Delta_{\infty}$ of uniform radius $\delta$.
Define $C_{\infty}:=\cup_{y\in\Delta_\infty} W_{r}^{s}(y,g)$,
where $r>0$ is given by the definition of cylinders $C_{n}$, and 
\[
\mathcal{K}_{\infty}:=\big\{ D\subset M:\ D=\lim_{k\to\infty}D_{n_{k}}\mbox{ where }D_{n_{k}}\in\mathcal{K}_{n_{k}},\ D\subset C\big\}.
\]
In rough terms, $\mathcal{K}_{\infty}$ is the family of disks tangent to the cone field $\cC$, 
obtained as the limit of unstable disks for some subsequence $\{g_{n_{k}}\}_{k}$.
As every disk $D_{n}\in\mathcal{K}_{n}$ is tangent to the cone field $\mathcal{C}$,
then the same property holds for the elements of $\mathcal{K}_{\infty}$.
Moreover:

\begin{lemma}
Every $D\in\mathcal{K}_{\infty}$ is an unstable disk (with respect to
some pre-orbit) for $g$. Moreover, if $\eta>0$ is given by Lemma~\ref{lem:eta} then 
$\mu({K}_{\infty})\geq\eta>0$.
\end{lemma}

\begin{proof}
Fix $D\in\mathcal{K}_{\infty}$ and, for notational simplicity, assume $D=\lim_{n\to\infty}D_{n}$.
If $x\in D$ is so that $D\subset B(x,\delta)$ then $D_{n}\subset B(x,\delta)$ for $n$ sufficiently large and,
consequently, every $g_{n}$ has well defined inverse branches in $B(x,\delta)$. 
Moreover, by definition, there exists $0<\sigma<1$ so that
for every $n\in\mathbb{N}$ there exists an inverse branch 
$g_{n,i_{n}}^{-1}:B(x,\delta)\rightarrow V_{i_{n}}$ for which
$
dist(g_{n,i_{n}}^{-1}(z),g_{n,i_{n}}^{-1}(y))\leq\sigma dist(z,y)
$
for every $z,y\in D_{n}$. 
By the construction of the disks $D$ we have that the sequence $\{D_n^{-1}\}_n$ of disks defined by
$D_{n}^{-1}:=g_{n,i_{n}}^{-1}(D_{n})$ are tangent to the cone field $\cC$. Thus, extracting a subsequence
if necessary, the Arzela-Ascoli theorem assures that $\{D_n^{-1}\}_n$ is convergent to some disk $D^{-1}$.
Moreover, $g(D^{-1})=D$, the disk $D^{-1}$ has radius smaller or equal to $\sigma\delta$ and 
the corresponding inverse branch of $g$ from $D$ to $D^{-1}$ is a contraction by $\sigma$. 
Applying this argument recursively we obtain contractive inverse branches for $g^{-j}$ on $D$ 
for every $j\in\mathbb{N}$, which implies that each $D\in\mathcal{K}$ is a unstable disk for some 
pre-orbit of $g$.

We are left to prove that $\mu(\mathcal{K}_{\infty})\geq\eta>0$.
Every $D\in\mathcal{K}_{\infty}$ is liftable to a set $\hat{D}\subset M^{g}$ for which 
$\pi=\pi^g\mid_{\hat D} : \hat D\to D$ is a bijection. Moreover, these sets are pairwise disjoint and 
admit contracting and expanding directions (recall Proposition \ref{prop:hatnuishyp}
and Lemma \ref{prop:disjointmanifolds}). By Lemma~\ref{lem:eta},
$\mu_{n}({K}_{n})\geq\eta>0$ for every $n\geq0$.
Moreover, if $V_{\varepsilon}({K}_{\infty})$ denotes
the $\varepsilon$-neighborhood of ${K}_{\infty}$ and $\varepsilon>0$ is chosen such that 
$\mu(\partial V_{\varepsilon}({K}_{\infty}))=0$ then
$
\mu(V_\vep ({K}_{\infty}))=\lim_{n\to\infty} \mu_n (V_\vep ({K}_{\infty}))  
	\ge \liminf_{n\to\infty} \mu_n ({K}_n) \ge \eta>0,
$
because ${K}_{n}\subset V_{\varepsilon}({K}_{\infty})$
for every large $n$.
Therefore, we conclude that 
$
\mu({K}_{\infty})=\liminf_{\varepsilon\to0}\mu(V_{\varepsilon}(\mathcal{K}_{\infty}))\geq\eta>0.
$
This proves the lemma.
\end{proof}

Denote by $\hat{\mathcal{K}}_{\infty}$ the family of lifts of disks
in $\mathcal{K}_{\infty}$ and by $\hat{K}_{\infty}$ the union of
its elements. Since the proof of the following lemma is identical to the one of Theorem~\ref{thm:abscontdisint} for $\mathcal{K}_{\infty}$
we shall omit it.

\begin{lemma}
There exists $\tau_{0}>1$ such that if $\{\hat{\mu}_{\hat{\Delta}}\}_{\hat{\Delta}\in\mathcal{K}_{\infty}}$
is a disintegration of $\hat{\mu}$ with respect to
$\hat{\mathcal{K}}_{\infty}$ then $\pi_{*}\hat{\mu}_{\hat{\Delta}}$
is absolutely continuous with respect to $Leb_{\Delta}$, where $\Delta=\pi(\hat{\Delta})$.
Moreover,
$
\tau_{0}^{-1}Leb_{\Delta}(B)\leq\pi_{*}\hat{\mu}_{\Delta}(B)\leq\tau_{0}Leb_{\Delta}(B)
$
for every measurable set $B\subset\Delta$.\label{lem:abscontdisint_inlimit}
\end{lemma}

We need the following instrumental result:

\begin{lemma}\cite[Lemma 6.2]{ABV00}.
Let $\upsilon$ be a finite measure on a measure space $Z$, with
$\upsilon(Z)>0$. Let $\mathscr{K}$ be a measurable partition
of $Z$, and $(\upsilon_{z})_{z\in Z}$ be a family of
finite measures on $Z$ such that:\label{lem:disintegration_abs_cont}
\begin{enumerate}
\item [(a)]the function $z\mapsto\upsilon_{z}(A)$ is measurable
and it is constant on each element of $\mathscr{K}$, for any measurable
subset $A\subset Z$.\label{enu:muz_is_const_on_Delta}
\item [(b)]$\left\{ w:\ \upsilon_{z}=\upsilon_{w}\right\} $ is a measurable
subset with full $\upsilon_{z}$-measure, for every $z\in Z$.\label{enu:muz_eq_muw_fullset}
\end{enumerate}
Assume that $\upsilon(A):=\int l(z)\upsilon_{z}(A)d\upsilon$
for some measurable function $l:Z\rightarrow\mathbb{R}_{+}$ and any
measurable subset $A\subset Z$. Let $\left\{ \upsilon_{\gamma}:\ \gamma\in\mathscr{K}\right\} $
and $\left\{ \upsilon_{z,\gamma}:\ \gamma\in\mathscr{K}\right\} $
be disintegrations of $\upsilon$ and $\upsilon_{z}$, respectively,
into conditional probability measures along the elements of the partition
$\mathscr{K}$. Then $\upsilon_{z,\gamma}=\upsilon_{\gamma}$ for
$\upsilon$-almost every $z\in Z$ and $\tilde{\upsilon}_{z}$-almost
every $\gamma\in\mathscr{K}$, where $\tilde{\upsilon}_{z}$ is the
quotient measure induced by $\upsilon_{z}$ on $\mathscr{K}$.
\end{lemma}

We now prove that almost every ergodic component of $\hat{\mu}$
is an SRB measure. 
The ergodic decomposition theorem (see e.g. \cite[Theorem II.6.4]{Mane87})
implies that there is a measurable subset $\hat{\Sigma}(\hat{g})\subset M^{g}$,
with full $\hat\mu$-measure, such that the limit
$
\hat{\mu}_{\hat{x}}:=\lim_{n\to\infty}\frac{1}{n}\sum_{j=0}^{n-1}\delta_{\hat{g}^{j}(\hat{x})}
$
exists and defines an ergodic measure, for every $\hat{x}\in\hat{\Sigma}(\hat{g})$.
Moreover, for every bounded and measurable observable $\hat{\phi}:M^{g}\rightarrow\mathbb{R}$:
(i) the map $\hat{x}\mapsto\int\hat{\phi}d\hat{\mu}_{\hat{x}}$ is measurable and 
$
\int\hat{\phi}d\hat{\mu}=\int(\int\hat{\phi}d\hat{\mu}_{x})d\hat{\mu}(\hat{x}),
$
and (ii)
$
\int\hat{\phi}d\hat{\mu}_{\hat{x}}=\lim_{n\to\infty}\frac{1}{n}\sum_{j=0}^{n-1}\hat{\phi}(\hat{g}^{j}(\hat{x}))
$
for $\hat{\mu}$ almost every $\hat{x}\in M^{g}$.
As before, if $\left\{ \hat{\mu}_{\hat{\Delta}}\right\} _{\hat{\Delta}\in\mathcal{\hat{K}}_{\infty}}$
is a disintegration of $\hat \mu$ then $\pi_{*}\hat{\mu}_{\hat{\Delta}}\ll Leb_{\Delta}$
if and only if $\hat{\mu}_{\hat{\Delta}}\ll\hat{m}_{\hat{\Delta}}$, where
$
\hat{m}_{\hat{\Delta}}:=(\pi|_{\hat{\Delta}})^{*}Leb_{\Delta}.
$
Fix $\hat{B}\subset M^{g}$ such that $\hat{m}_{\hat{\Delta}}(\hat{B}\cap\hat{\Delta})=0$
for every $\hat{\Delta}\in\mathcal{\hat{K}}_{\infty}$ and $\hat{\mu}(\hat{B})$
is maximal among all sets with this property. By Lemma \ref{lem:abscontdisint_inlimit}
we have that $\hat{\mu}_{\hat{\Delta}}\ll\hat{m}_{\hat{\Delta}}$.
Therefore, if $\tilde{\hat{\mu}}$ is the quotient measure 
given by Rokhlin's disintegration theorem, we conclude that
$\hat{\mu}(\hat{B})=\int\int\hat{\mu}_{\hat{\Delta}}(B)d\tilde{\hat{\mu}}(\hat{\Delta})=0$.

\medskip
Recall that the set of generic points $R(\hat{g})$ is defined by as the set of points $\hat{y}\in M^{g}$ 
such that the limits $\lim_{n\to\infty}\frac{1}{n}\sum_{j=0}^{n-1}\hat{\phi}(\hat{g}^{j}(\hat{y}))$ and 
$\lim_{n\to\infty}\frac{1}{n}\sum_{j=0}^{n-1}\hat{\phi}(\hat{g}^{-j}(\hat{y}))$ exist and coincide for every 
$\hat{\phi}\in C(M^{g})$.
Consider also the set 
$
\hat{Z}:=[\hat K_\infty\cap\Sigma(\hat{g})\cap R(\hat{g})]\backslash\hat{B}.
$
The previous lemma allow us to compare the disintegration of $\hat{\mu}$
with respect to $\mathcal{\hat{K}}_{\infty}$ with the disintegration
of the ergodic components of $\hat{\mu}$ with respect to $\hat{\mathcal{K}}_{\infty}$.
In what follows we prove that almost every ergodic component of $\mu$ is 
an SRB measure.

\begin{proposition}\label{prop:disintegration_of_ergodic_component} Let $\{\hat{\mu}_{\hat{x},\hat{\Delta}} \}_{\hat{\Delta}\in\mathcal{\hat{K}}_\infty}$ be a disintegration of $\hat{\mu}_{\hat{x}}$ along $\hat{\cK}_{\infty}$. Then $\hat{\mu}_{\hat{x},\hat{\Delta}}=\hat{\mu}_{\hat{\Delta}}$,
hence absolutely continuous with respect to Lebesgue,
for $\hat{\mu}$-almost every $\hat{x}\in {\hat{K}_\infty}$ and $\tilde{\hat{\mu}}$-almost every $\hat{\Delta}\in\mathcal{\hat{K}}_\infty$.
\end{proposition}

\begin{proof}
We apply Lemma \ref{lem:disintegration_abs_cont}
with $Z=\hat{K}_\infty$, $\upsilon=\hat{\mu}|_{\hat{K}}$, $\mathscr{K}=\mathcal{\hat{K}}_{\infty}$
and $\upsilon_{\hat{z}}=\hat{\mu}_{\hat{z}}$, for each $\hat{z}\in\hat{K}_\infty$.
Since $R(\hat{g})$ and $\Sigma(\hat{g})$
are full $\hat{\mu}$-measure sets and $\hat{\mu}(\hat{B})=0$ then
$
\hat{\mu}(\hat{Z})=\hat{\mu}(\hat{K}_\infty\cap R(\hat{g})\cap\Sigma(\hat{g}))=\hat{\mu}(\hat{K}_\infty)\geq\eta>0.
$
Given a fixed measurable subset  $\hat{A}\subset M^{g}$, the map $\hat Z \ni \hat{x}\mapsto\hat{\mu}_{\hat{x}}(\hat{A})$ 
is measurable, by the ergodic decomposition theorem. Fix $\hat{\Delta}\in\mathcal{\hat{K}}_{\infty}$ and $\hat{x,}\hat{y}\in\hat{Z}\cap\hat{\Delta}$.
Given $\hat{\phi}\in C(M^{g})$ we have that
$
\int\hat{\phi}d\hat{\mu}_{\hat{x}}=\lim_{n\to\infty}\frac{1}{n}\sum_{j=0}^{n-1}\hat{\phi}(\hat{g}^{j}(\hat{x}))=\lim_{n\to\infty}\frac{1}{n}\sum_{j=0}^{n-1}\hat{\phi}(\hat{g}^{-j}(\hat{x}))
$
and
$
\int\hat{\phi}d\hat{\mu}_{\hat{y}}=\lim_{n\to\infty}\frac{1}{n}\sum_{j=0}^{n-1}\hat{\phi}(\hat{g}^{j}(\hat{y}))=\lim_{n\to\infty}\frac{1}{n}\sum_{j=0}^{n-1}\hat{\phi}(\hat{g}^{-j}(\hat{y}))
$
because $\hat{x},\hat{y}\in R(\hat{g})\cap\Sigma(\hat{g})$.
But, since $\hat{x},\hat{y}$ belong to the same unstable set $\hat{\Delta}$ then
$\int\hat{\phi}d\hat{\mu}_{\hat{x}}=\int\hat{\phi}d\hat{\mu}_{\hat{y}}$ and,
since $\hat{\phi}\in C(M^{g})$ was taken arbitrary, $\hat{\mu}_{\hat{x}}=\hat{\mu}_{\hat{y}}$.
This proves that $\hat{x}\mapsto\hat{\mu}_{\hat{x}}(\hat{A})$ is constant on 
each $\hat{\Delta}\in\mathcal{\hat{K}}_{\infty}$.
Moreover, it is a standard in ergodic theory to show that the set $E_{\hat{z}}:=\left\{ \hat{w}:\ \hat{\mu}_{\hat{w}}=\hat{\mu}_{\hat{z}}\right\}$ is a full $\hat{\mu}_{\hat{z}}$-measure subset in $M^g$. 

We are left to show that there is a measurable function $l:Z\rightarrow\mathbb{R}_{+}$
such that $\hat{\mu}(\hat{A})=\int l(\hat{x})\cdot\hat{\mu}_{\hat{x}}(\hat{A})d\hat{\mu}(\hat{x})$
for every measurable subset $\hat{A}\subset\hat{Z}$. Denote by $\mathds{1}_{\hat{A}}$
the characteristic function in the set $\hat{A}$. Then
$
\hat{\mu}_{\hat{x}}(\hat{A})=\lim_{n\to\infty}\frac{1}{n}\sum_{j=0}^{n-1}\mathds{1}_{\hat{A}}(\hat g^j(\hat{x})),
$
which is positive if and only if $\hat{x}$ 
has positive frequency of visits to $\hat{A}\subset\hat{Z}$.
Given $\hat{z}\in\hat{Z}$ define
$
l(\hat{z}):=\min\{ r>0:\hat{g}^{-r}(\hat{z})\in\hat{Z}\} .
$
By Poincar\'e's Recurrence theorem the map $l(\cdot)$ is well
defined $\hat{\mu}$-almost everywhere in $\hat{Z}$. Then, following \cite[Lemma 6.1]{ABV00},
$
\hat{\mu}(\hat{A})=\int_{\Sigma}\hat{\mu}_{\hat{x}}(\hat{A})d\hat{\mu}(\hat{x})=\int_{\hat{Z}}l(\hat{z})\hat{\mu}_{\hat{z}}(\hat{A})d\hat{\mu}(\hat{z}).
$
So, we can apply Lemma~\ref{lem:disintegration_abs_cont} to conclude
the proof of this proposition.
\end{proof}

Consider $K=\pi(\hat{K}_{\infty})$. The following completes the proof of Theorem~\ref{thm:TEO2}.

\begin{corollary}
For $\mu$-almost every $x\in K$ the measure $\mu_{x}$ is an SRB measure for $g$. In particular,
$\mu$ is a convex combination of SRB measures for $g$.
\end{corollary}

\begin{proof}
Observe that for each $\hat{y}\in M^{g}$ we have that $\pi_{*}\delta_{\hat{y}}=\delta_{y_{0}}$ and, from the linearity and continuity of the projection $\pi_{*}:\mathcal{M}(M^{g})\rightarrow\mathcal{M}(M)$,
we have that $\pi_{*}\hat{\mu}_{\hat{x}}=\mu_{x_{0}}$ for every $\hat{x}\in\Sigma(\hat{g})$,
where $\mu_{x_{0}}=\lim_{n\to\infty}\frac{1}{n}\sum_{j=0}^{n-1}\delta_{f^{j}(x_{0})}$.
Moreover if $\hat{B}\subset M^{g}$ is $\hat{\mu}$-full measure set
then $\pi(\hat{B})\subset M$ is a $\mu$-full measure set,
where $\mu=\pi_{*}\hat{\mu}$. 
Therefore we conclude that for $\mu$-almost every $x\in K_\infty$,
$\mu_{x}$ is such that its lift $\hat{\mu}_{x}$ has disintegration
that coincides with the disintegration of $\hat{\mu}$ along the unstable
disks of $\hat{\mathcal{K}}_{\infty}$ almost everywhere. But,
as the disintegration of $\hat{\mu}$ is absolutely continuous with respect to the Lebesgue measure 
in the unstable disks in $\hat{\mathcal{K}}_{\infty}$, we get that $\mu_x$ is an SRB measure.
\end{proof}

\section{Proof of the corollaries}\label{subsec:open}

\subsection{Proof of Corollary~\ref{cor:entropy-cont}}
Let $\cU$ be a open set of transitive $C^{1+\alpha}$ non-singular endomorphisms that satisfy (H1)-(H4)
with uniform constants and, for any $f\in \cU$  let $\mu_f$ be its unique SRB measure. By ergodicity of $\mu_f$,  Pesin's entropy formula for endomorphisms (\cite[Theorem~VII.1.1]{QXZ09}, after\cite{Liu1}), assures 
that
$$
h_{\mu_f }(f) = \int \sum_{\lambda_i(f,\mu_f)>0} \lambda_i(f,\mu_f) \dim E_i \, d\mu_f .
$$
Therefore, by Oseledets theorem,
\begin{align}\label{eqLL}
\int \log |\det Df| \, d\mu_f 
	& = \int \sum_{\lambda_i(f,\mu_f)>0} \lambda_i(f,\mu_f) \dim E_i \, d\mu_f 
	+ \int \sum_{\lambda_i(f,\mu_f)<0} \lambda_i(f,\mu_f) \dim E_i\, d\mu_f \nonumber \\
	& = h_{\mu_f }(f)
	+ \int \sum_{\lambda_i(f,\mu_f)<0} \lambda_i(f,\mu_f) \dim E_i \, d\mu_f.
\end{align}
Now, recall that $\mu_f$ is a hyperbolic measure and that all Oseledets subspaces associated to negative Lyapunov exponents are contained in the stable subbundle $E^s$. In consequence, applying the Oseledets theorem to the
continuous cocycle $Df\mid {E^s}$ we conclude that 
\begin{align*}
\int \log |\det Df| \, d\mu_f 
	= h_{\mu_f }(f)
	+ \int  \log |\det Df \mid_{E^s}| d\mu_f
\end{align*}
or, equivalently, 
\begin{align}\label{formula-ent}
h_{\mu_f }(f)= \int \log \frac{|\det Df|}{|\det Df \mid_{E^s}|} \, d\mu_f. 
\end{align}
Since the stable subbundle $E^s$ for $f$ varies continuously with the non-singular endomorphism $f$ (with the Grassmannian topology)
and the function $\cU \to C^0(M,\mathbb R)$ given by $f \mapsto \log \frac{|\det Df|}{|\det Df \mid_{E^s}|}$ is continuous, the continuity of the entropy function follows from the latter together with the continuity of the SRB measure $\cU \ni f\mapsto \mu_f$ (in the weak$^*$ topology). This proves the corollary.

\subsection{Proof of Corollary~\ref{thm:open}}\label{subsecopen}

Let $\mathcal{U}\subset End^{1+\alpha}(M)$ be the ($C^1$-open) set of non-singular endomorphisms on $M$
satisfying (A1)-(A4). The corollary is an immediate consequence of Theorems~\ref{thm:TEO1} and ~\ref{thm:TEO2} 
once we prove that assumptions (H1)-(H4) hold with uniform constants.
Given $g\in \cU$ condition (A1) implies that for every $x\in U$
$$
E^s_x(g)= \bigcap_{n\ge 0} Dg^n(x)^{-1}( C_{a}^{-}(g^n(x)))\subsetneq C_{a}^{-}(x)
$$
defines an $Dg$-invariant constracting subbundle $E^s(g)$ with contraction rate $\lambda$. 
Moreover, if  $v\in E_{x}^{s}$, $w\in Df(x) ( C_{a}(x))$ and $x\in M$ then 
$\left\Vert Dg(x)\cdot v\right\Vert \cdot\left\Vert Dg(x)^{-1}\cdot w\right\Vert 
	\leq\lambda \max_{x\in M} L(x)^{-1} \cdot\left\Vert v\right\Vert \cdot\left\Vert w\right\Vert 
	\le \sqrt{\lambda} \left\Vert v\right\Vert \cdot\left\Vert w\right\Vert 	
$ 
provided that 
\begin{equation}\label{eqL1domin}
L:=\min_{x\in M} L(x)  \ge \sqrt{\lambda}
\end{equation}
This implies (H1), (H2) and (H4) hold with the uniform constant $\tilde\lambda=\sqrt{\lambda}\in (0,1)$.

It remains to prove that assumption (H3) holds. It is enough to prove that there exists $c>0$ such that,
for every disk $D$ tangent to the cone field $C(\cdot)$,
$
\limsup_{n\to\infty}\frac{1}{n}\sum_{j=0}^{n-1}\log\|(Dg(g^{j}(x))|_{C(g^{j}(x))})^{-1}\|\leq-2c<0
$
for Lebesgue almost every $x\in D$.
Fix any disk $D$ tangent to $ C(\cdot)$. 
Condition~(A3) together with the positive $Dg$-invariance of the cone field $C$ 
implies that $\det\left|Dg(x)|_{T_{x}D}\right|>\sigma$ for every $x\in D$ and that this property holds for the 
positive iterates of the disk $D$.
Now, given $n\ge 1$ and $\gamma\in(0,1)$, consider the set 
\[
R(D,n,\gamma):=\left\{ x\in D:\ \#\left\{ 0\leq j\leq n-1:\ g^{j}(x)\in\mathcal{O}\right\} \geq\gamma n\right\} 
\]
of points 
that spends a fraction of time larger then $\gamma$ in the region $\mathcal{O}$.
\begin{lemma}
\label{lem:freq_visit} There are 
$\gamma_{0}\in(\frac{1}{2},1)$, $\vep>0$ and $K>0$ (depending only on $p$, $q$ and $\sigma$)
such that $Leb_{D}(R(D,n,\gamma_{0}))\leq Ke^{-\vep n}$
for every large $n\ge 1$.
\end{lemma}
\begin{proof}
The proof is entirely analogous  to \cite[Lemma A.1]{ABV00}.
\end{proof}
\noindent We are now in a position to determine $L$ in terms of the frequency of visits described in the previous
lemma. Assume that $L$ is sufficiently close to $1$ so that 
\begin{equation}\label{eq:determineL}
L^{\gamma_{0}}\cdot\lambda_{u}^{(1-\gamma_{0})}\geq e^{c}>1
\end{equation}
for some $c>0$. By this choice, 
if one partitions the iterates $g^j(x)$ of a point $x\in D$ depending on if these
belong to $\mathcal O$ or not,
and set $r(n,x):=\left\{ 0\leq j\leq n-1:g^{j}(x)\in\mathcal{O}\right\} $
and $s(n,x):=\left\{ 0,1,\dots,n-1\right\} \backslash r(n,x)$.
Now, using Lemma~\ref{lem:freq_visit}, the Borel-Cantelli lemma implies that for
$Leb_{D}$-a.e. $x$ there is $n(x)$ such that $x\notin R(D,n,\gamma_{0})$
for every $n\geq n(x)$. In other words, we have that $\#r(n,x)<\gamma_{0}n$
for Lebesgue almost every $x\in D$ and all $n\geq n(x)$. 
Therefore, by $Dg$-invariance of the cone field $C$,
\begin{align}
\frac{1}{n}\sum_{j=0}^{n-1}\log\|(Dg(g^{j}(x))|_{C(g^{j}(x))})^{-1}\| 
	& \leq\frac{1}{n}\#r(n,x)\log L^{-1}-\frac{1}{n}\#s(n,x)\log\lambda_{u} \nonumber \\
	& \leq \gamma_{0}\log L^{-1}+(1-\gamma_{0})\log\lambda_{u} \nonumber  \\
	 & =-(\log(L^{\gamma_{0}}\cdot\lambda_{u}^{(1-\gamma_{0})}))\leq-c<0 \label{uniformestim}
\end{align}
for Lebesgue almost every $x\in D$ and every large $n\ge 1$.
Since $C$ is a cone field preserved by all elements $g\in \cU$ and the constants are uniform
we conclude that all elements $g\in \cU$ satisfy (H1)-(H4) with uniform constants, as claimed.

\section{Final comments: singular endomorphisms and future directions} \label{sec:future}

Let us make some final comments concerning singular endomorphisms.
In \cite{Tsu05}, Tsujii considered the space $\mathcal{PH}^u$ of $C^r$-endomorphisms on surfaces that admit a continuous splitting
$TM=E^c\oplus E^u$ with $\dim E^c =\dim E^u=1$ and for which there exist $C,\lambda>0$ such that 
$\|Df^n(x)\mid_{E^u_x} \ge Ce^{\lambda n}$ and $\|Df^n(x)\mid_{E_x^c} \le Ce^{\lambda n} \|Df^n(x)\mid_{E^u_x}\|$
for all $x\in M$ and $n\ge 0$. The space $PH^u$ forms a $C^1$-open set of the space of 
endomorphisms and $C^r$-generic endomorphisms $f$ in $\mathcal{PH}^u$ admit a finite number of physical measures, where $r\ge 19$
(cf. \cite{Tsu05}).
In particular these results hold for the class of robust transitivity endomorphisms with persistence of critical points considered 
in \cite{Lizana} (the kernel of $Df$ at critical points is contained in the subbundle $E^c$). 

\vspace{.1cm}

In the present paper we consider a complementary situation. We consider partially hyperbolic endomorphisms (more generally attractors)   
$\mathcal{PH}^s$ that admit a continuous splitting $TM=E^s\oplus E^c$, which is dominated, such that $E^s$ is uniformly contracting and 
there exists an invariant cone field with a non-uniform expansion condition. These form an open set of partially hyperbolic non-singular 
endomorphisms. On the one hand, the existence of a (invariant) stable subbundle and an invariant cone field seems unnatural in the
presence of critical points (e.g. the endomorphism $f: \mathbb T^2 \times [0,1] \to 
\mathbb T^2 \times [0,1]$ given by $f(x,y,z)=(3x+ y, x+y, h^n(z))$, where $h(z)=4 z(1-z^2)$ does not admit such a decomposition if
is $n\ge 1$ is large). On the other hand, the lack of domination is itself a very hard obstacle even in the case of diffeomorphisms.
Indeed, the H\'enon-like maps (\cite{MoraV}) given by
$
h_{a,b}(x,y)=(1-ax^2, 0) + r(a,x,y)
$
with $\|r\|_{C^3}\le b$ (for $b$ small enough) are $C^3$-perturbations of the singular endomorphism
$h_{a,0}(x,y)=(1-ax^2, 0)$ which admits a non-strict positively invariant cone field. Since this is not robust, inspired by \cite{MoraV,Ures}
we pose the following question:

\medskip

\noindent {\bf Question~1:}
\emph{Assume that $f$ is an endomorphism with $C^1$-persistence of singularities on a surface $M$ (there exists a $C^1$-open neighborhood of $f$
formed by singular endomorphisms). 
If $f\notin \mathcal{PH}^u$ does there exist a $C^1$-perturbation such that $f$ admits a sink?  
Can $f$ be approximated by $C^1$ endomorphisms with infinitely many sinks (hence infinitely many physical measures)?
}
\medskip

In view of the previous discussion it seems only reasonable to consider families of singular endomorphisms for which 
singular behavior is isolated from the domination.  In view of \cite{Tsu05} it is not reasonable to expect an affirmative answer to 
the following question:
\medskip

\noindent {\bf Question~2:}
\emph{Assume that $A: \mathbb T^2\to \mathbb T^2$ is an Anosov endomorphism and 
take $F_0: \mathbb T^2 \times \mathbb S^1 \to \mathbb T^2 \times \mathbb S^1$ be given by $F_0=A\times Id$.
If $\mathcal S_\vep$ denotes the space of 
skew-products 
$
F(x,y)=(A(x,y),f(x,y,z))
$
such that $\|F-F_0\|_{C^2}<\vep$ and $\vep$ is small, does there exist a $C^r$-residual subset 
$\mathcal R \subset \mathcal S_\vep$ of singular partially hyperbolic endomorphisms 
such that every $G\in \mathcal R$ has a finite number of physical measures?}

\medskip

A different question concerns the existence of singularities. Piecewise smooth partial hyperbolic endomorphisms 
with singularities arise naturally e.g. as Poincar\'e maps of Lorenz attractors.  If $f\in \mathcal{PH}^u$ is a partially 
hyperbolic endomorphism with singularities whose derivative behaves like a power of the distance to the singular
region satisfying an extra slow recurrence condition (see \cite{ABV00} for precise definition)  it is most likely the ideas 
in the proof of Theorem~\ref{thm:TEO1}  can be carried out to prove the existence and finiteness of SRB measures. 
Finally, it is interesting to understand if this construction can probably made to work (maybe with some extra difficulties) 
if one removes the hypothesis of an everywhere invariant cone field and assume that the set of points 
that admit infinitely many iterates at which cone fields are in an invariant position and backward contrition
holds simultaneously has positive Lebesgue positive measure. 
 This seems to be the case of H\'enon-like maps for the parameters obtained by the parameter exclusion process in \cite{MoraV}.

Finally, it would be interesting to describe when these SRB measures are absolutely continuous with respect to
the Lebesgue measure on the ambient space. On the one hand, in our context any absolutely continuous invariant measure is 
a convex combination of SRB measures (see Corollary VII.1.1.1 of \cite{QXZ09}). 
In the context of Theorem~\ref{thm:TEO1}, it follows from ~\cite{LL} any SRB measure $\mu$ is absolutely continuous with respect
to volume if and only if the entropy production $e_\mu$ is zero or, equivalently,  the folding entropy $F(\mu)$ would satisfy
$
F(\mu)=\int \log |\det Df| \, d\mu_f,
$
in which case the folding entropy would vary continuously with the dynamics (we refer the reader to \cite{LL} for the definitions).
It follows from \cite{Shu} that this is equivalent to say that $\mu$ is absolutely continuous with respect to the stable foliation.
This motivates the following:

\medskip

\noindent {\bf Question~3:}
\emph{ In the context of Theorem~\ref{thm:TEO2}, are SRB measures of $C^2$-generic endomorphisms singular with respect 
to the volume? If so, do these partially hyperbolic endomorphisms admit inverse SRB measures (in the sense of \cite{Mih10a,MU})? 
}
\medskip

Finally, as the existence of SRB measures relies on explicit geometric constructions, we expect to obtain further statistical properties
of these hyperbolic SRB measures. In particular, it is interesting to establish the smooth dependence of
these measures (and their entropies) with respect to the endomorphism. These problems are usually referred as linear
response problems (see e.g. ~\cite{Bal14,BCV16} and references therein).

\subsection*{Acknowledgments:} 
This work is part of the first author's PhD thesis at Universidade Federal da Bahia. 
This work was finished during the visit of the second author to Fudan University, Peking University 
and Soochow University, whose research conditions, kind hospitality and fruitful discussions are deeply acknowledged.
The authors are deeply grateful to J.F. Alves, V. Ara\'ujo, I. Melbourne,V. Pinheiro 
and C. V\'asquez
for valuable comments.


\end{document}